\documentclass[reqno]{amsart}
\usepackage[foot]{amsaddr}
\usepackage[english]{babel}
\usepackage[T1]{fontenc}
\usepackage[utf8]{inputenc}
\usepackage[square,numbers]{natbib}
\usepackage{indentfirst}
\usepackage{doi}
\usepackage{preamble}

\usepackage{amssymb}
\usepackage{amsmath}
\usepackage{hyperref}       % hyperlinks
\usepackage{url}            % simple URL typesetting
\usepackage{booktabs}       % professional-quality tables
\usepackage{amsfonts}       % blackboard math symbols
\usepackage{nicefrac}       % compact symbols for 1/2, etc.
\usepackage{microtype}      % microtypography
\usepackage[table,dvipsnames]{xcolor}         % colors
\usepackage{preamble}
\usepackage{subcaption}
\usepackage{lipsum}
\usepackage{diagbox}
\usepackage{tabularray}
\UseTblrLibrary{booktabs} % Allows \toprule, \midrule, \bottomrule inside tblr

\usepackage{siunitx}
\newcounter{thmcounter}[section]

\usepackage{thmtools}

\theoremstyle{definition}

\theoremstyle{plain}
\declaretheorem[name=Assumption,sibling=thmcounter]{assumption}
\theoremstyle{definition}
\declaretheorem[name=Setting,sibling=thmcounter]{setting}
\theoremstyle{remark}
\newtheorem{remark}[thmcounter]{Remark}

\usepackage[colorinlistoftodos,prependcaption,textsize=small]{todonotes}
\newcommand{\fp}{\texttt{[FP]}}
\newcommand{\mvgh}{\mathtt{mv}_{h}}
\newcommand{\articletitle}{Accelerating Natural Gradient Descent for PINNs with
    Randomized Numerical Linear Algebra}

\newcommand{\bilinear}[3][ ]{%
    \ifthenelse{\equal{#1}{}}%
    {\left({#2},\,{#3}\right)}%
    {#1\left({#2},\,{#3}\right)}}

\newcommand{\NN}{u_{\vtheta}^{\mathcal{NN}}}

\makeatletter
\renewcommand\paragraph{\@startsection{paragraph}{4}%
  \z@\z@{-\fontdimen2\font}{\normalfont\itshape}}
\makeatother

\begin{document}

% amsart uppercases the title without expanding macros, hence \expandafter.
% The short title is for the running heads, where the full one does not fit.
\newcommand{\settitle}[1]{\title[Accelerating Natural Gradient Descent for
    PINNs]{#1}}
\expandafter\settitle\expandafter{\articletitle}

\author[I.~Bioli]{Ivan Bioli$^{\dagger,\ddagger}$}
\email{\href{mailto:ivan.bioli@unipv.it}{ivan.bioli@unipv.it}}

\author[C.~Marcati]{Carlo Marcati$^{\S}$}
\email{\href{mailto:carlo.marcati@univ-lyon1.fr}{carlo.marcati@univ-lyon1.fr}}

\author[G.~Sangalli]{Giancarlo Sangalli$^{\dagger,\P}$}
\email{\href{mailto:giancarlo.sangalli@unipv.it}{giancarlo.sangalli@unipv.it}}

% With amsaddr's foot option, each affiliation is listed once in the footnote,
% matched to the authors by the markers above
\address{$^{\dagger}$Department of Mathematics “F. Casorati”, University of
  Pavia, 27100 Pavia, Italy}
\address{$^{\ddagger}$Department of Civil Engineering and Architecture,
  University of Pavia, 27100 Pavia, Italy}
\address{$^{\S}$Université Lyon 1, Centrale Lyon, INSA Lyon, Université Jean
  Monnet, CNRS, ICJ UMR5208, 69622 Villeurbanne, France}
\address{$^{\P}$Istituto di Matematica Applicata e Tecnologie Informatiche
  “E. Magenes”, CNR, 27100 Pavia, Italy}

\thanks{Corresponding author: Ivan Bioli}

% \subjclass{\textcolor{red}{TODO}}

% Keywords are set at the end of the abstract: amsart's \keywords would hide
% them in the first-page footnote
\begin{abstract}
  Natural Gradient Descent (NGD) has emerged as a promising optimization algorithm
for training neural network-based solvers for partial differential equations
(PDEs), such as Physics-Informed Neural Networks (PINNs). However, its practical
use is often limited by the high computational cost of solving linear systems
involving the Gramian matrix. While matrix-free NGD methods based on the
conjugate gradient (CG) method avoid explicit matrix inversion, the
ill-conditioning of the Gramian significantly slows the convergence of CG. In
this work, we extend matrix-free NGD to broader classes of problems than
previously considered and propose the use of Randomized Numerical Linear Algebra
(RandNLA) techniques for efficient preconditioning of the inner CG solver. The
resulting algorithms demonstrate substantial performance improvements over
existing NGD-based methods on a range of PDE problems discretized using neural
networks, and offer competitive results compared to other state-of-the-art
optimizers.

\end{abstract}

\maketitle

\sloppy

\section{Introduction}
Neural network (NN) based approaches for solving partial differential equations
(PDEs) have recently attracted significant attention as promising alternatives
to traditional numerical methods for both forward and inverse problems. A common
strategy is to incorporate the governing PDE directly into the loss function, as
in Physics-Informed Neural Networks (PINNs)
\cite{dissanayake_neural-network-based_1994,raissi_physics-informed_2019,sirignano_dgm_2018},
Variational PINNs \cite{kharazmi_variational_2019,kharazmi_hp-vpinns_2021}, the
Deep Ritz method \cite{e_deep_2017}, and various subsequent extensions
\cite{badia_finite_2024,han_solving_2018,berman_neural_2024,de_ryck_wpinns_2024,rojas_robust_2024}.
These methods hold the promise of providing mesh-free approximations of PDE
solutions and of exploiting the approximation power of neural networks
\cite{devore_neural_2021,de_ryck_approximation_2021}, particularly in high
dimensions \cite{Ziang2021}, positioning them as promising tools for
high-dimensional, inverse, or complex domain problems.
Despite this theoretical promise, the practical adoption and reliability of
NN-based PDE solvers are often hindered by significant optimization challenges.
Standard first-order algorithms, notably Adam \cite{kingma_adam_2017}, have been
reported to often perform poorly, stagnating at solutions with low accuracy
\cite{krishnapriyan_characterizing_2021,wang_understanding_2021,rathore_challenges_2024,wang_when_2022}.
This issue is generally attributed not to the expressiveness of neural networks,
but to the ill-conditioning of the loss landscape
\cite{rathore_challenges_2024,de_ryck_operator_2024,krishnapriyan_characterizing_2021}.
Consequently, specialized optimization strategies tailored to PDE-related
loss functions are increasingly being advocated
\cite{cuomo_scientific_2022,muller_position_2024,ryck_numerical_2024}.

To counteract these training pathologies, various strategies have been proposed.
A body of work focuses on loss weighting schemes or adaptive selection of
collocation points to balance contributions from different terms in the loss and
reduce quadrature errors
\cite{wang_understanding_2021,bonito_convergence_2025,van_der_meer_optimally_2022,wang_respecting_2024,wu_comprehensive_2023}.
While these techniques can mitigate some failure modes, they do not fully
resolve the underlying ill-conditioning introduced by discretizing a PDE
operator. Only recently, more principled quasi-second-order (or curvature-based) optimization algorithms
have begun to emerge
\cite{muller_achieving_2023,muller_position_2024,de_ryck_operator_2024,zeng_competitive_2022,dangel_kronecker-factored_2024,rathore_challenges_2024,hao_gauss_2024,jnini_gauss-newton_2024,jnini_curvature-aware_2026},
often derived from preconditioning the underlying infinite-dimensional PDE
operator
\cite{de_ryck_operator_2024,muller_position_2024,muller_achieving_2023,hao_gauss_2024}.
Among them, Natural Gradient Descent (NGD) has emerged as a particularly
promising and unifying framework, as several earlier approaches can be
interpreted as NGD with different choices of the optimization metric
\cite{nurbekyan_efficient_2023,muller_position_2024}.

NGD has demonstrated remarkable performance, achieving high accuracy in few
iterations by using a search direction of the form $\vd = - (\mG(\vtheta) + \mu
    \mI)^{-1}\nabla L(\vtheta)$, where $\vtheta \in \R^p$ denotes the neural network
parameters, $L$ is the loss function, $\mG(\vtheta)$ the positive semidefinite
Gramian matrix (see \Cref{sec:ngd}), and $\mu > 0$ a regularization parameter.
However, solving linear systems with $(\mG(\vtheta) + \mu \mI)$ via direct
methods incurs a cubic cost $\calO(p^3)$ in the number of network parameters,
making NGD impractical for large-scale networks. To alleviate this, matrix-free
strategies have been proposed
\cite{jnini_gauss-newton_2024,zeng_competitive_2022}, where the system is solved
approximately using the conjugate gradient (CG) method using only matrix--vector
products (matvecs) with $\mG(\vtheta)$, without assembling the matrix. In
practice, however, $\mG(\vtheta)$ is often ill-conditioned, as shown in
\Cref{fig:decay_sample}, leading to slow CG convergence and motivating the need
for preconditioning. In this work, we address this issue using tools from
Randomized Numerical Linear Algebra (RandNLA)
\cite{martinsson_randomized_2020,murray_randomized_2023}.
\subsection{Contributions}
The core contributions of this work are as follows:
\begin{enumerate}
    \item We provide a general
          methodology for efficiently computing matrix--vector products with the Gramian
          using automatic differentiation, generalizing beyond the specific metrics
          considered in prior works
          \cite{muller_position_2024,jnini_gauss-newton_2024,zeng_competitive_2022}.
    \item We propose employing techniques from RandNLA
          to construct a preconditioner for $(\mG(\vtheta) + \mu \mI)$. In particular,
          we focus on the randomized Nystr{\"o}m approximation
          \cite{frangella_randomized_2023,tropp_fixed-rank_2017} and the Randomly Pivoted
          Partial Cholesky (RPCholesky) decomposition
          \cite{chen_randomly_2024,epperly_embrace_2025}. By relying on a low-rank
          approximation of the Gramian, these preconditioners leverage its empirically
          observed strong spectral decay (see \Cref{fig:decay_sample}), turning it from a computational challenge into
          an asset.
    \item We present a modular and scalable framework for efficient NGD
          that decouples the underlying NGD implementation (whether full or matrix-free)
          from the preconditioner construction, enabling flexible algorithmic
          combinations.
\end{enumerate}
Consequently, we introduce two novel NGD variants for optimizing neural
network-based PDE solvers, termed \emph{Nystr{\"o}mNGD} and \emph{RPCholNGD}.
These methods combine matrix-free NGD with randomized Nystr{\"o}m and RPCholesky
preconditioning, respectively. Finally, we present extensive numerical
experiments demonstrating that our proposed methods achieve significant
acceleration compared to existing NGD-based approaches, which serve as our
primary baseline, while also exhibiting competitive performance against other
state-of-the-art optimizers.

\subsection{Related work and comparison with existing optimization methods}
Compared to full-memory quasi-Newton methods such as BFGS and SSBroyden
\cite{wright_numerical_2006,urban_unveiling_2025,kiyani_optimizing_2025}, our
approach differs both conceptually and computationally.
Conceptually, NGD is rooted in a first-optimize-then-discretize paradigm: one
chooses a metric in function space, possibly tailored to the underlying PDE, and
then pulls it back to parameter space through the neural network ansatz.
Consequently, the Gramian need not coincide with an approximation of the Hessian
of the parameter-space loss. In contrast, quasi-Newton methods operate directly
in parameter space and iteratively approximate the Hessian (or inverse Hessian)
of $L(\vtheta)$, thereby losing the functional interpretation associated with
the underlying infinite-dimensional problem. Computationally, Nystr{\"o}mNGD and
RPCholNGD combine matrix-free Gramian matrix--vector products (matvecs) with
rank-$\ell$ preconditioners. This yields a tunable memory footprint of order
$\calO(p\ell)$, where $p$ is the number of neural network parameters, in
contrast to the $\calO(p^2)$ storage required by full-memory quasi-Newton
methods. This low-rank and matrix-free structure also distinguishes our methods
from standard NGD implementations that rely on explicitly assembling or
factorizing the full Gramian.

While matrix-free NGD implementations have been explored for specific metrics
\cite{jnini_gauss-newton_2024,zeng_competitive_2022}, none of these prior works
address the inherent ill-conditioning of the Gramian, which severely hampers CG
convergence in practice. To overcome this bottleneck, our work bridges NGD with
RandNLA, whose use in Machine Learning is well-established
\cite{frangella_sketchysgd_2024,derezinski_recent_2024,udell_randomized_2023}.
Randomized Nystr{\"o}m preconditioning \cite{frangella_randomized_2023} has been
applied to PINNs in \cite{rathore_challenges_2024}, but only as a final stage
after Adam and L-BFGS, applying the Nystr{\"o}m approximation, designed for
symmetric positive semidefinite matrices, to the Hessian, which may be
indefinite. Randomized techniques have been applied to NGD
\cite{mckay_near-optimal_2025,guzman-cordero_improving_2025} in the so-called sketch-and-solve paradigm,
whereas we employ them to construct a preconditioner. Alternative notions of preconditioning, where the
preconditioner is embedded into the loss function, are explored in
\cite{liu_preconditioning_2024,badia_finite_2024}. Kronecker-Factored
Approximate Curvature (K-FAC) has also
been used to approximate the Gramian
\cite{dangel_kronecker-factored_2024}, but not as a preconditioner, and it
requires problem-specific adaptations.

\subsection{Paper outline}
The remainder of the paper is structured as follows. In
\Cref{sec:preliminaries}, we introduce the abstract formulation of the problem
and the neural network framework. \Cref{sec:ngd_matrix_free} reviews Natural
Gradient Descent and presents its matrix-free formulation.
\Cref{sec:modular_frameworks} presents our modular and scalable framework for
efficient Natural Gradient Descent, including a detailed description of our
novel preconditioned matrix-free NGD algorithms, Nystr{\"o}mNGD and RPCholNGD.
\Cref{sec:numerical_experiments} provides numerical experiments on a range of
PDE-related problems, comparing our proposed methods against state-of-the-art
optimizers. Finally, \Cref{sec:discussion} summarizes our findings and outlines
directions for future research.

\begin{figure}
    \centering
    \begin{subfigure}[t]{0.49\textwidth}
        \includegraphics[width=\textwidth]{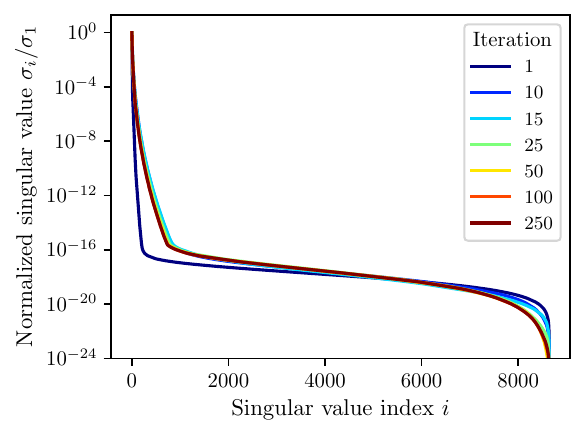}
    \end{subfigure}
    \hfill
    \begin{subfigure}[t]{0.49\textwidth}
        \includegraphics[width=\textwidth]{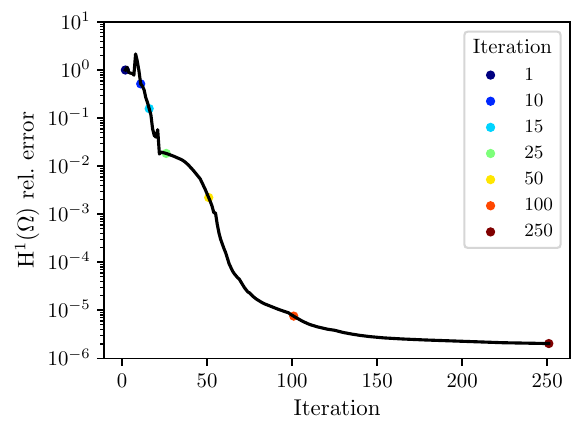}
    \end{subfigure}
    \caption{Spectral decay of the Gramian matrix during PINN training for the
        3D Poisson problem with exact solution $u(x,y,z) = \sin(\pi x)\sin(\pi
            y)\sin(\pi z)$, using Energy Natural Gradient Descent
        \cite{muller_achieving_2023}. The neural network is  fully connected  with
        three hidden layers of 64 neurons each. Left: Spectral decay of the Gramian
        $\mG(\vtheta)$ at various iterations, where each color denotes a different
        iteration. Right: Evolution of the $\rmH^1(\Omega)$ relative error during
        training.}
    \label{fig:decay_sample}
\end{figure}

\section{Preliminaries}
\label{sec:preliminaries}
\subsection{Notation}
Hilbert and Banach spaces are indicated using roman capital letters. Linear
operators are denoted with calligraphic letters. Vectors in $\R^n$ are in
boldface, and matrices use boldface capital letters. The Euclidean gradient of a
function $L: \R^n \to \R$ is denoted by $\nabla L$.

Let $\rmV$ be a real Hilbert space. Its inner product and norm are denoted by
$\innerh[\rmV]{\cdot}{\cdot}$ and $\norm{\cdot}_{\rmV}$. Its (topological) dual
is indicated $\dual{\rmV}$, and the duality pairing by
$\dualpair[\rmV]{\cdot}{\cdot}$. The space of bounded linear operators from
$\rmV$ to another Hilbert space $\rmW$ is denoted by $\calL(\rmV, \rmW)$. We say
that $\calA \in \calL(\rmV, \dual{\rmV})$ is symmetric positive semidefinite
(SPSD) if the associated bilinear form $a(u,v) = \dualpair[\rmV]{\calA u}{v}$ is
symmetric and satisfies $a(u,u) \geq 0$ for all $u \in \rmV$. It is symmetric
positive definite (SPD) if $a$ is also coercive, i.e., if $\exists\, \alpha > 0$
such that $a(u,u) \geq \alpha \norm{u}_{\rmV}^2$ for all $u \in \rmV$.

\subsection{Abstract setup}
\label{sec:abstract_setup}
Let $\rmV$ be a Hilbert space equipped with the inner product
$\innerh[\rmV]{\cdot}{\cdot}$. We endow $\rmV$ with a potentially
point-dependent metric $\bilinear[g_u]{\cdot}{\cdot}$, giving $\rmV$ the
structure of a Riemannian manifold. Specifically, the map $u \mapsto g_u$ is
assumed to be smooth from $\rmV$ to $\mathcal{L}^2_{\mathrm{sym}}(\rmV,\R)$, the
space of continuous, symmetric, bilinear forms on $\rmV$. We denote $\norm{w}_{g_u} = \sqrt{\bilinear[g_u]{w}{w}}$. Representative examples of these spaces and
metrics are provided in \Cref{rmk:metrics_examples}. We consider a general
minimization problem of the form
\begin{equation}
    \label{eq:prob_infinitedim}
    \inf_{u\in\rmV} E(u),
\end{equation}
where $E:\rmV\to\R$ is a smooth function (with respect to the Riemannian metric
$g$), referred to as the \emph{energy}. The Riemannian gradient and Hessian of
$E$ (in the metric $g$) are denoted by $\grad[g] E$ and $\Hess[g] E$, respectively.

To approximate the solution to \eqref{eq:prob_infinitedim}, we introduce a parametrized ansatz
\[
    \calN = \left\{ u_{\vtheta} \defas \rho(\vtheta) \,:\, \vtheta \in \Theta \right\} \subseteq \rmV,
\]
where $\rho: \Theta \to \rmV$ is a smooth map and $\Theta \subseteq \R^p$ is open.
We then consider the finite-dimensional optimization problem
\begin{equation}
    \label{eq:prob_finite}
    \inf_{\vtheta \in \Theta} L(\vtheta) \defas \inf_{\vtheta \in \Theta} E(u_{\vtheta}) = \inf_{\vtheta \in \Theta} (E \circ \rho)(\vtheta).
\end{equation}
Here, $u_{\vtheta}$ could be a neural network with parameters $\vtheta$,
though the framework applies more generally. We will refer to
$\Theta\subseteq\R^p$ as the \emph{parameter space}, and to $\rmV$ as the
\emph{function space}. Note that while the original optimization problem
\eqref{eq:prob_infinitedim} may be infinite dimensional, the ansatz is
parametrized by a finite-dimensional $\vtheta\in\R^p$.

\subsubsection{Discretization of PDEs}
\label{sec:discretization_of_PDEs}
The focus of this article is on problems arising from the discretization of
PDEs. In this case, $\rmV$ is a Hilbert space of functions defined on a domain
$\Omega \subset \R^d$, e.g., $\rmV = \rmH^1(\Omega)$, and $E$ is derived from
the strong or weak formulation of a PDE on $\Omega$. This
abstract setting \eqref{eq:prob_finite} encompasses, for
instance, the Deep Ritz method \cite{e_deep_2017}, PINNs
\cite{raissi_physics-informed_2019,dissanayake_neural-network-based_1994},
Variational PINNs (VPINNs) \cite{kharazmi_variational_2019}, and other PINN
variants
\cite{kharazmi_hp-vpinns_2021,badia_finite_2024,han_solving_2018,de_ryck_wpinns_2024,rojas_robust_2024}.

As a general framework for PDEs we consider the following. Let $\rmV$ be a
Hilbert space and let $\rmX, \rmY$ be Banach spaces. We seek $u \in \rmV$ such that
\begin{equation}
    \label{eq:pde_general}
    \begin{cases}
        \calF(u) = f, & \text{in } \Omega,         \\
        \calB(u) = g, & \text{on } \partial\Omega,
    \end{cases}
\end{equation}
where $\calF:\rmV\to\rmX$ is a differential operator, $f\in\rmX$ is the source
term, $\calB:\rmV\to\rmY$ is a boundary operator, and $g\in\rmY$ is the boundary
data. The energy can then be defined as
\[
    E(u) = \frac{1}{2}\norm{\calF(u) - f}_{\rmX}^2 + \frac{1}{2}\norm{\calB(u) - g}_{\rmY}^2,
\]
where possibly weighting factors can be included, and
data misfit terms can be added.

The choice of function spaces and well-posedness depend on the specific PDE and
on whether a strong or weak formulation is used. For example, consider the Poisson
equation with homogeneous Dirichlet boundary conditions:
\begin{equation}
    \label{eq:poisson}
    \begin{cases}
        -\Delta u = f, & \text{in } \Omega,         \\
        u = g,         & \text{on } \partial\Omega.
    \end{cases}
\end{equation}
For a strong formulation, one can employ the spaces $\rmV = \rmH^2(\Omega)$, $\rmX = \rmL^2(\Omega)$, and $\rmY = \rmL^{2}(\partial\Omega)$, with
energy
\[
    E(u) = \frac{1}{2}\norm{\Delta u + f}_{\rmL^2(\Omega)}^2 + \frac{1}{2}\norm{u - g}_{\rmL^{2}(\partial\Omega)}^2.
\]
With weak formulations, boundary conditions can be imposed either weakly (with $\rmY = \rmH^{1/2}(\partial\Omega)$) or strongly.
We adopt the latter approach, for which the natural spaces are $\rmV = \rmH^1_0(\Omega)$, $\rmX = \dual{\left(\rmH^1_0(\Omega)\right)} = \rmH^{-1}(\Omega)$.
Given a lifting $\tilde{g} \in \rmH^1(\Omega)$ of the Dirichlet data, the solution can be written as $u = u_0 + \tilde{g}$, where the energy is
\[
    E(u_0) = \frac{1}{2}\norm{\calA u_0 - \tilde{f}}_{\rmH^{-1}(\Omega)}^2,
    \qquad \text{with} \qquad
    \inner{\calA u}{v}_{\rmH^{-1} \times \rmH^1_0}
    =
    \int_{\Omega} \nabla u \cdot \nabla v,
    \quad
    \tilde{f} = f - \calA \tilde{g}.
\]

\subsection{Neural networks}
\label{sec:neural_networks}
Neural networks (NNs) serve as the primary parametrized ansatz in this work. For
simplicity, we focus on feedforward networks, though our methods are
architecture-agnostic and apply to other network types as well. A neural network
with input dimension $d$ and output
dimension $d'$ is a function $\NN: \R^d \to \R^{d'}$ defined as the composition of
affine and nonlinear transformations. Let $L$ be the number of hidden layers,
let $n_0 = d, n_{L+1} = d'$, and $n_1, \dots, n_L$ be the number of neurons in
each hidden layer, and let $\sigma$ be a nonlinear activation function. The
mapping induced by the NN
is defined as:
\[
    \begin{split}
        \vx_0                & = \vx \in \R^{d},                                                                                  \\
        \vx_\ell             & = \sigma(\mW_\ell \vx_{\ell-1} + \vb_\ell) \in \R^{n_\ell}, \qquad \text{for } \ell = 1, \dots, L, \\
        \NN(\vx) = \vx_{L+1} & = \mW_{L+1} \vx_L + \vb_{L+1} \in \R^{d'},
    \end{split}
\]
where $\mW_\ell \in \R^{n_\ell \times n_{\ell-1}}$ and $\vb_\ell \in
    \R^{n_\ell}$ are the weights and biases of layer $\ell$, respectively. The
parameters of the network are $\vtheta = [\vecc(\mW_1); \vecc(\vb_1); \dots; \vecc(\mW_{L+1});
    \vecc(\vb_{L+1})]\in\R^p$, stacked into a single vector, with $p =
    \sum_{\ell=1}^{L+1} (n_\ell n_{\ell-1} + n_\ell)$. No activation function is applied
in the last layer.

\subsubsection{Finite Element Interpolated Neural Networks (FEINNs)}
\label{sec:feinns}
While standard NNs are typically used with collocation-based methods like PINNs,
in order to deal with weak formulations of PDEs,
we consider Finite Element
Interpolated Neural Networks (FEINNs)
\cite{badia_finite_2024,badia_adaptive_2024,berrone_variational_2022}, where a
neural network is interpolated onto a finite element (FE) space.

Let $\Omega \subset \R^d$ be a Lipschitz polygonal/polyhedral domain with a conforming mesh $\calT_h$, and let $\rmV_h \subset \rmV = \rmH^1(\Omega)$ be the
FE space with $\rmV_{h,0} = \rmV_h \cap \rmH^1_0(\Omega)$. Denote the Lagrangian
basis of $\rmV_{h}$ by $\{\varphi_i\}_{i=1}^{N}$, where $\varphi_i$ is
associated to node
$\vx_i$, and let $\Pi_h: C^0(\Omega) \to \rmV_{h,0}$ be the FE interpolation
operator obtained by evaluation at interior nodes. Given a discrete lifting
$\bar{u}_h$ of Dirichlet data, the FEINN is defined as
\[
    u_{\vtheta}= \bar{u}_h + \Pi_h\NN =  \bar{u}_h + \sum_{i\,:\,\vx_i\in\mathring{\Omega}} \NN(\vx_i) \varphi_i.
\]
For simplicity, we assumed that Dirichlet boundary conditions are imposed on the
whole boundary, but mixed boundary conditions can also be accommodated in the
framework. Further details are provided in
\cite{badia_finite_2024,badia_adaptive_2024}.

This approach offers several advantages. First, Dirichlet boundary conditions
can be enforced by applying them to the FE interpolant using
standard FE techniques, avoiding the need to impose them directly on the neural
network output \cite{berrone_variational_2022,berrone_enforcing_2023}. Second,
it alleviates integration challenges, a known issue in neural network-based PDE
solvers \cite{rivera_quadrature_2022}, especially in the Deep Ritz method
\cite{taylor_stochastic_2025}. Finally, spatial derivatives can be computed at
the FE level, eliminating the need for automatic differentiation with respect to
the input of the neural network.

\section{Matrix-free Natural Gradient Descent}
\label{sec:ngd_matrix_free}
\subsection{Natural Gradient Descent}
\label{sec:ngd}
The Natural Gradient Descent (NGD) method
\cite{amari_natural_1998,martens_new_2020,nurbekyan_efficient_2023,muller_position_2024}
targets optimization problems of the form \eqref{eq:prob_finite} following the
first-optimize-then-discretize paradigm. Within the functional setting
established in \Cref{sec:abstract_setup}, NGD moves along the steepest descent
direction according to the metric $g$ on $\rmV$, the ``natural'' one for the
function $u_{\vtheta}$, and pulls it back to the parameter space.

Let $\vtheta$ be the current iterate. The steepest descent direction in function
space is $d = -\grad[g] E(u_{\vtheta})$, the Riemannian gradient of $E$ in
the metric $g$, and the corresponding update
\[
    u \leftarrow u_{\vtheta} - \alpha \grad[g] E(u_{\vtheta}).
\]
In NGD, the parameter update direction is $-\vd$, where the \emph{natural
    gradient} $\vd \in \R^p$ is such that the first-order expansion
\[
    u_{\vtheta - \alpha \vd} \approx u_{\vtheta} - \alpha \rmD\rho(\vtheta)[\vd] = u_{\vtheta} - \alpha \sum_{j=1}^{p} d_j \cdot \partial_{\theta_j}u_{\vtheta},
\]
best approximates (in the metric $g$) the update in function space, i.e.,
\[
    \vd = \argmin_{\vw \in \R^p} \norm{\grad[g] E(u_{\vtheta}) - \rmD\rho(\vtheta)[\vw]}_{g_{u_{\vtheta}}}.
\]
This corresponds to the $g_{u_{\vtheta}}$-orthogonal projection of
$\grad[g] E(u_{\vtheta})$ onto the model's \emph{generalized tangent space}
\[
    \Tang_{\vtheta}\calN
    =
    \Spann\left(
    \left\{
    \partial_{\theta_1}u_{\vtheta},
    \dots,
    \partial_{\theta_p}u_{\vtheta}
    \right\}
    \right)
    \subset \rmV.
\]
Specifically, the projection $v = \sum_{i=1}^{p} d_i \partial_{\theta_i}u_{\vtheta}$ satisfies
\[
    \bilinear[g_{u_{\vtheta}}]
    {\sum_{j=1}^{p} d_j \partial_{\theta_j}u_{\vtheta}}
    {\partial_{\theta_i}u_{\vtheta}}
    =
    \bilinear[g_{u_{\vtheta}}]
    {\grad[g] E(u_{\vtheta})}
    {\partial_{\theta_i}u_{\vtheta}}
    =
    \rmD E(u_{\vtheta})[\partial_{\theta_i}u_{\vtheta}]
    =
    \partial_{\theta_i} L(\vtheta),
\]
or, in compact form,
\begin{equation}
    \label{eq:natural_gradient}
    \mG(\vtheta)\vd = \nabla_{\vtheta} L(\vtheta),
    \qquad
    \left[\mG(\vtheta)\right]_{ij}
    =
    \bilinear[g_{u_{\vtheta}}]
    {\partial_{\theta_i}u_{\vtheta}}
    {\partial_{\theta_j}u_{\vtheta}},
\end{equation}
where $\mG(\vtheta)\in\R^{p\times p}$ is often referred to as the Gramian
\cite{zeinhofer_unified_2024,muller_position_2024} or (Fisher) information matrix
\cite{amari_natural_1998,martens_new_2020,nurbekyan_efficient_2023}.
If
$\partial_{\theta_1}u_{\vtheta}, \dots, \partial_{\theta_p}u_{\vtheta}$
are not linearly independent, the minimizer is not unique, and the
minimum-norm solution is selected, namely,
\begin{equation}
    \label{eq:natural_gradient_pinv}
    \vd = \mG(\vtheta)^{\dagger}\nabla_{\vtheta} L(\vtheta),
\end{equation}
where $\mG(\vtheta)^{\dagger}$ denotes the Moore--Penrose pseudoinverse of
$\mG(\vtheta)$.
As observed in the literature
\cite{zeinhofer_unified_2024,bonfanti_challenges_2024}, it is often beneficial
to add a damping term to improve numerical stability and convergence, yielding
the regularized direction
\begin{equation}
    \label{eq:regularized_ng}
    \vd = \left(\mG(\vtheta) + \mu \mI\right)^{-1}\nabla_{\vtheta} L(\vtheta),
\end{equation}
where $\mu > 0$ is a regularization parameter, typically adapted during the
optimization process. Once the (regularized) natural gradient descent is
computed, the parameters are updated as
\[
    \vtheta \leftarrow \vtheta - \alpha \vd = \vtheta - \alpha (\mG(\vtheta) + \mu \mI)^{-1}\nabla_{\vtheta} L(\vtheta),
\]
where the step-size $\alpha$ is typically chosen via a line-search procedure.

\subsection{A matrix-free formulation}
\label{sec:matrix-free}
Leaving aside the resolution of \eqref{eq:regularized_ng}, even assembling the
Gramian $\mG(\vtheta)$ in \eqref{eq:natural_gradient} can be computationally
prohibitive in both floating-point operations (flops) and memory. Indeed, $\mG(\vtheta)$ is a $p \times p$
matrix, where $p$ is the number of network parameters,
which can be on the order of tens of thousands even for relatively small
neural networks. To mitigate this issue, we consider a matrix-free approach:
$\mG(\vtheta)$ is never explicitly assembled, and only
matrix--vector products are computed as needed. This is particularly efficient when
solving linear systems with $(\mG(\vtheta) + \mu \mI)$ using iterative methods
such as Conjugate Gradient (CG), which require only matvecs.

Prior work has demonstrated that matrix--vector products with the Gramian
\eqref{eq:natural_gradient} can be computed via automatic differentiation for
specific metrics \cite{jnini_gauss-newton_2024,muller_position_2024,zeng_competitive_2022}.
Here, we show that autodiff-based matvecs extend to broader metric classes.
Specifically, we show that matrix--vector products with the Gramian
\eqref{eq:natural_gradient} can be computed via automatic differentiation,
provided the metric satisfies the following assumption and it is discretized via
quadrature.
\begin{restatable}{assumption}{matfreeassumption}
    \label{assumption:metric}
    Let $\Omega \subset \R^d$ be a domain, and let $\rmV$ be a Hilbert space of functions on $\Omega$. Assume that the metric
    $g_u$ takes the form
    \begin{equation}
        \label{eq:metric_as_integral}
        \bilinear[g_u]{v}{w} = \int_{\Omega} \calF_u v(\vx) \cdot \calF_u w(\vx) \, \varphi_u(\vx) \,\mathrm{d}\vx,
    \end{equation}
    for all $v, w \in \Tang_u \rmV \subseteq \rmV$, where $\calF_u: \rmV \to
        \rmL^2(\Omega; \varphi_u)^k$ is linear and continuous, and $\varphi_u$ is a positive function.
\end{restatable}

\begin{restatable}{thm}{matfreethm}
    \label{thm:matfree}
    Let \Cref{assumption:metric} hold, and let $\mathtt{sg}:\R^p\to\R^p$ denote
    the stop-gradient operator, formally defined by $\mathtt{sg}(\vtheta) =
        \vtheta$ and $\nabla\mathtt{sg}(\vtheta) = \mathbf{0}$. If the integral
    \eqref{eq:metric_as_integral} is approximated using a quadrature rule with
    points $\{\vx_r\}_{r=1}^q$ and weights $\{w_r\}_{r=1}^q$, then
    \begin{equation}
        \label{eq:G_matrixfree}
        \begin{split}
            \mG(\vtheta)
             & = \sum_{r=1}^q w_r\, \varphi_{u_{\vtheta}}(\vx_r)\, \left(\mJ_{\vtheta}\calF_{u_{\mathtt{sg}(\vtheta)}} u_{\vtheta}(\vx_r)\right)^\top \mJ_{\vtheta}\calF_{u_{\mathtt{sg}(\vtheta)}} u_{\vtheta}(\vx_r) \\
             & =\mJ_{\vtheta}\mF(\vtheta)^\top \begin{bsmallmatrix}
                                                   w_1 \varphi_{u_{\vtheta}}(\vx_1)\mI_k & & \\
                                                   & \ddots & \\
                                                   & & w_q \varphi_{u_{\vtheta}}(\vx_q)\mI_k
                                               \end{bsmallmatrix} \mJ_{\vtheta}\mF(\vtheta),
            \quad\,\,\, \vF(\vtheta) = \begin{bsmallmatrix}
                                           \calF_{u_{\mathtt{sg}(\vtheta)}} u_{\vtheta}(\vx_1)\\
                                           \vdots\\
                                           \calF_{u_{\mathtt{sg}(\vtheta)}} u_{\vtheta}(\vx_q)
                                       \end{bsmallmatrix},
        \end{split}
    \end{equation}
    where $\mJ_{\vtheta}$ denotes the Jacobian with respect to $\vtheta$.
\end{restatable}
\begin{proof}
    \label{proof:thm_matfree}
    Let $\ve_h\in\R^p$ denote the $h$th vector of the canonical basis, i.e.,
    the vector that has a $1$ in the $h$th entry and $0$ in all other entries.
    Using the linearity and continuity of $\calF_{u_{\vtheta}}$,
    \begin{equation*}
        \begin{split}
            \calF_{u_{\vtheta}} \partial_{\theta_h}u_{\vtheta} & = \calF_{u_{\mathtt{sg}(\vtheta)}} \partial_{\theta_h}u_{\vtheta} = \calF_{u_{\mathtt{sg}(\vtheta)}}\left(\lim_{t\to 0}\frac{u_{\vtheta + t \ve_h} - u_{\vtheta}}{t}\right)                                                               \\
                                                               & =\lim_{t\to 0} \calF_{u_{\mathtt{sg}(\vtheta)}} \left(\frac{u_{\vtheta + t \ve_h} - u_{\vtheta}}{t}\right) =\lim_{t\to 0} \frac{\calF_{u_{\mathtt{sg}(\vtheta)}} u_{\vtheta + t \ve_h} - \calF_{u_{\mathtt{sg}(\vtheta)}} u_{\vtheta}}{t} \\
                                                               & = \partial_{\theta_h}\calF_{u_{\mathtt{sg}(\vtheta)}} u_{\vtheta}.
        \end{split}
    \end{equation*}
    Note that the use of the stop-gradient operator is crucial to ensure that we
    do not differentiate $\calF_{u_{\vtheta}}$ with respect to $\vtheta$. Hence,
    \begin{equation*}
        \begin{split}
            [\mG(\vtheta)]_{ij} & = \bilinear[g_{u_{\vtheta}}]{\partial_{\theta_i}u_{\vtheta}}{\partial_{\theta_j}u_{\vtheta}}
            = \int_{\Omega} \calF_{u_{\vtheta}} \partial_{\theta_i}u_{\vtheta} \cdot \calF_{u_{\vtheta}} \partial_{\theta_j}u_{\vtheta} (\vx) \,\varphi_{u_{\vtheta}} \,\mathrm{d}\vx                                            \\
                                & = \int_{\Omega} \partial_{\theta_i}\calF_{u_{\mathtt{sg}(\vtheta)}} u_{\vtheta} \cdot \partial_{\theta_j}\calF_{u_{\mathtt{sg}(\vtheta)}} u_{\vtheta} \,\varphi_{u_{\vtheta}} \,\mathrm{d}\vx,
        \end{split}
    \end{equation*}
    or, in matrix form,
    \[
        \mG(\vtheta) = \left(\D_{\vtheta} \calF_{u_{\mathtt{sg}(\vtheta)}}u_{\vtheta}\right)^* \, \D_{\vtheta} \calF_{u_{\mathtt{sg}(\vtheta)}}u_{\vtheta},
    \]
    where the Hilbert adjoint $\left(\D_{\vtheta}
        \calF_{u_{\mathtt{sg}(\vtheta)}}u_{\vtheta}\right)^*$ is taken with respect
    to the inner product on $\rmL^2(\Omega, \varphi_u)^k$, and we identify a
    matrix with the induced linear map.

    Applying a quadrature formula with nodes $\{\vx_r\}_{r=1}^q$ and weights
    $\{w_r\}_{r=1}^q$, we get
    \begin{equation*}
        \begin{split}
            [\mG(\vtheta)]_{ij} & = \sum_{r=1}^q w_r \, \partial_{\theta_i}\calF_{u_{\mathtt{sg}(\vtheta)}} u_{\vtheta} (\vx_r) \cdot \partial_{\theta_j}\calF_{u_{\mathtt{sg}(\vtheta)}} u_{\vtheta}(\vx_r) \,\varphi_{u_{\vtheta}}(\vx_r)
        \end{split}
    \end{equation*}
    and thus, in matrix form:
    \begin{equation*}
        \begin{split}
            \mG(\vtheta) & = \sum_{r=1}^q w_r\, \varphi_{u_{\vtheta}}(\vx_r)\, \left(\mJ_{\vtheta}\calF_{u_{\mathtt{sg}(\vtheta)}} u_{\vtheta}(\vx_r)\right)^\top \mJ_{\vtheta}\calF_{u_{\mathtt{sg}(\vtheta)}} u_{\vtheta}(\vx_r) \\
                         & = \mJ_{\vtheta}\mF(\vtheta)^\top \begin{bsmallmatrix}
                                                                w_1 \varphi_{u_{\vtheta}}(\vx_1)\mI_k & & \\
                                                                & \ddots & \\
                                                                & & w_q \varphi_{u_{\vtheta}}(\vx_q)\mI_k
                                                            \end{bsmallmatrix} \mJ_{\vtheta}\mF(\vtheta).
        \end{split}
    \end{equation*}
\end{proof}
As a consequence of \Cref{thm:matfree}, matrix--vector products with
$\mG(\vtheta)$ can be computed using only products with the Jacobian and the
transpose Jacobian of $\mF$. These operations can be efficiently implemented via
automatic differentiation, using forward-mode autodiff for Jacobian-vector
products (JVPs) and backward-mode autodiff for vector-Jacobian products (VJPs).
Note that most autodiff frameworks include a stop-gradient function.

\begin{remark}
    \label{rmk:metrics_examples}
    \Cref{assumption:metric} encompasses standard metrics used in PDE-related
    function spaces. For instance, the $\rmL^2(\Omega)$ and $\rmH^1(\Omega)$
    metrics correspond to the choices $\calF_u v = v$ and $\calF_u v = (\nabla
        v, v)$, respectively. It also includes the Gauss--Newton metric on $
        \rmL^2(\Omega)$, for which matrix-free NGD implementations have been
    proposed in \cite{jnini_gauss-newton_2024,muller_position_2024} for specific
    problems. However, \Cref{assumption:metric} is significantly more general
    and covers a much broader class of metrics. In practice, it includes many of
    the metrics relevant in the context of PDEs, since these typically need to
    be expressible as integrals to be computationally tractable. For example, it
    includes the Newton metric $\bilinear[g_u]{v}{w} = \int_{\Omega} \nabla v
        \cdot \nabla w + 3 u^2 v w$ in \cite[Section 4.2]{muller_position_2024}, and
    the energy inner product for the Gross--Pitaevskii equation as considered in
    \cite{henning_sobolev_2020}, $\bilinear[g_u]{v}{w} = \int_{\Omega} \nabla v
        \cdot \nabla w + V v w + \beta \abs{u}^2 v w$.
\end{remark}

\subsubsection{A link with operator preconditioning and the case of Finite Element Interpolated Neural Networks}
\label{sec:operator_preconditioning_and_feinns}
The Gramian \eqref{eq:natural_gradient} can be interpreted in the context of
operator preconditioning
\cite{hiptmair_operator_2006,mardal_preconditioning_2011,kirby_functional_2010,gunnel_note_2014}.
Specifically, NGD can be seen as a preconditioned gradient descent method, with
the Gramian serving as the preconditioner \cite{de_ryck_operator_2024}. For a
\emph{linear} parametrized model $u_{\vtheta} = \sum_{i=1}^p \theta_i
    \,\varphi_i$, the Gramian reduces to $\left[\mG(\vtheta)\right]_{ij} =
    \bilinear[g_{u_{\vtheta}}]{\varphi_i}{\varphi_j}$. If $(\rmV,
    \innerh[\rmV]{\cdot}{\cdot})$ is a Hilbert space and $g$ is induced by a positive
definite linear operator $\calG: \rmV \to \dual{\rmV}$, then
\[
    [\mG(\vtheta)]_{ij} = [\mG]_{ij} = \dualpair[\rmV]{\calG \varphi_i}{\varphi_j},
\]
which is independent of $\vtheta$. This coincides exactly with preconditioners
arising from operator preconditioning in Galerkin discretizations of PDEs
\cite{hiptmair_operator_2006}.

The connection with operator preconditioning highlights that the choice of
metric crucially influences the convergence speed of NGD. In general, the link
between NGD and Riemannian gradient descent \cite{muller_achieving_2023}
suggests selecting the metric $g$ to reduce the condition number of the
Riemannian Hessian. Further details are provided in \Cref{apx:sec:metric}, where
we present a unified framework for metric selection.

In the case of FEINNs, \Cref{thm:matfree} can be specialized to connect with
operator preconditioning for finite elements. Assume, for simplicity, that the
interior nodes of the mesh are the first $N_I$,
so that the FEINN can be written as $u_{\vtheta} = \bar{u}_h + \sum_{i=1}^{N_I}
    \NN(\vx_i) \varphi_i$. Then the Gramian \eqref{eq:natural_gradient} takes the form
\begin{equation}
    \label{eq:feinn_gramian}
    \mG(\vtheta) = \left(\mJ_{\vtheta}\mP(\vtheta)\right)^\top \mG_h \left(\mJ_{\vtheta}\mP(\vtheta)\right), \qquad \text{where} \qquad \mP(\vtheta) = \begin{bsmallmatrix}
        \NN(\vx_1)\\
        \vdots\\
        \NN(\vx_{N_I})
    \end{bsmallmatrix},
\end{equation}
where $\mG_h \in \R^{N_I \times N_I}$ with $[\mG_h]_{ij} = \dualpair[\rmV]{\calG
        \varphi_i}{\varphi_j}$ is the FEM operator preconditioner. Note that
computing matvecs with the Gramian requires only matvecs with $\mG_h$,
but not with its inverse, and JVPs and VJPs with $\mP(\vtheta)$, which
is an evaluation of the neural network at the interior nodes of the
mesh.

\section{A modular and scalable framework for Natural Gradient Descent}
\label{sec:modular_frameworks}
In this section, we present a framework for efficiently computing the Natural
Gradient Descent (NGD) direction. The framework is inherently modular: we break
the computation down into subtasks and propose several efficient methods for
each, allowing the optimal choice to be tailored to the specific problem at
hand. The computational cost estimates and the formulation of NGD in this
section are based on the following assumption.
\begin{assumption}
    We consider the regime where a highly accurate quadrature rule is employed and the number of quadrature points $q$ exceeds the number of neural network parameters $p$.
\end{assumption}
\noindent When $q \ll p$ (for instance, if mini-batching is used at each
NGD iteration), techniques based on the Woodbury matrix identity can be employed
to reduce the computational cost \cite{guzman-cordero_improving_2025}.

To compute the natural gradient $\vd = \left(\mG(\vtheta) + \mu
    \mI\right)^{-1}\nabla_{\vtheta} L(\vtheta)$, a primary distinction lies in how
the linear system associated with the matrix $\left(\mG(\vtheta) + \mu
    \mI\right)$ is solved: directly or iteratively. When using a direct solver
approach, the required tasks at each step are:
\begin{enumerate}
    \item Assemble the full Gramian $\mG(\vtheta)$;
    \item Compute the regularization parameter $\mu > 0$;
    \item Solve the regularized linear system $\left(\mG(\vtheta) + \mu
              \mI\right)\vd = \nabla_{\vtheta} L(\vtheta)$ using a direct
          solver. Because the regularized Gramian is symmetric positive
          definite and dense, the preferred direct solver is the Cholesky
          decomposition followed by triangular solves.
\end{enumerate}
Alternatively, when employing an iterative approach, the solver of choice is
the preconditioned Conjugate Gradient (pCG) method. The corresponding steps are:
\begin{enumerate}
    \item Define a matrix--vector product function $\vv \mapsto \mG(\vtheta) \,\vv$;
    \item Compute the regularization parameter $\mu > 0$;
    \item Define a suitable preconditioner (if any);
    \item Solve the regularized linear system $\left(\mG(\vtheta) + \mu
              \mI\right)\vd = \nabla_{\vtheta} L(\vtheta)$ using pCG.
\end{enumerate}
A schematic overview of the framework is provided in \Cref{fig:ngd_framework_tree}.

In the remainder of this section, we outline how to efficiently execute the
operations detailed in the two approaches above. We analyze them in terms of
both flops and memory footprint, trying to elucidate which technique is the most
appropriate based on the specific problem and available computational
resources.

\begin{figure}[htp]
    \centering
    \begin{tikzpicture}[
        node distance=0.7cm,
        box/.style={draw, rectangle, rounded corners, thick, align=center},
        root/.style={box, fill=gray!20, minimum width=6cm, inner ysep=1.5ex},
        ngd/.style={box, fill=blue!10, minimum width=6cm, inner ysep=1.5ex},
        category/.style={box, fill=green!10, text width=6cm, inner ysep=1.5ex},
        listbox/.style={box, text width=6cm, align=center, inner ysep=1.5ex},
        final/.style={box, fill=gray!20, minimum width=6cm, inner ysep=1.5ex}
    ]

    % 1. Euclidean Gradient
    \node[root] (grad) {\textbf{Compute Euclidean Gradient} \\ $\vg = \nabla_{\vtheta} L(\vtheta)$};

    % 2. Natural Gradient
    \node[ngd, below=0.6cm of grad] (ngddir) {\textbf{Compute Natural Gradient} \\ $\left(\mG(\vtheta) + \mu\mI\right)\vd = \vg$};

    \draw[->, thick] (grad.south) -- (ngddir.north);

    % 3. Categories
    \node[category, below=0.8cm of ngddir, xshift=-3.5cm] (direct) {\textbf{Direct Solver} \\ (Dense $\mG$)};
    \node[category, below=0.8cm of ngddir, xshift=3.5cm] (iterative) {\textbf{Iterative Solver} \\ (Matrix-free pCG)};

    % Split arrows
    \draw[thick] (ngddir.south) -- ++(0,-0.4) coordinate (split);
    \draw[->, thick] (split) -| (direct.north);
    \draw[->, thick] (split) -| (iterative.north);

    % --- LEFT BRANCH (Direct Solver) ---

    % Step 1 Left
    \node[listbox, fill=orange!10, below=0.6cm of direct] (step1L) {
    \textbf{Assembly} \\{\scriptsize(\Cref{sec:full_gramian})}\\[1ex]
    \begin{tabular}{@{{\Large$\circ$}~}l@{}}
        Assemble full $\mG(\vtheta)$
    \end{tabular}
    };

    \draw[->, thick] (direct.south) -- (step1L.north);

    % --- RIGHT BRANCH (Iterative Solver) ---

    % Step 1 Right
    \node[listbox, fill=orange!10, below=0.6cm of iterative] (step1R) {
    \textbf{Matvec Strategy} \\{\scriptsize(\Cref{sec:matvecs_strategies})}\\[1ex]
    \begin{tabular}{@{{\Large$\circ$}~}l@{}}
        Precompute Gramian  \\
        Precompute Jacobian \\
        Batch-wise Jacobian
    \end{tabular}
    };

    \draw[->, thick] (iterative.south) -- (step1R.north);

    % Step 2 boxes: tops aligned at 0.6cm below step1R.south (the taller box)
    \path (step1R.south) ++(0,-0.6cm) coordinate (step2_top);

    % Step 2 Left
    \node[listbox, fill=red!10, anchor=north] (step2L) at (direct.center |- step2_top) {
        \textbf{Solve} \\[1ex]
        \begin{tabular}{@{{\Large$\circ$}~}l@{}}
            Cholesky solve
        \end{tabular}
    };

    % Step 2 Right
    \node[listbox, fill=red!10, anchor=north] (step2R) at (iterative.center |- step2_top) {
        \textbf{Preconditioner} \\[1ex]
        \begin{tabular}{@{{\Large$\circ$}~}l@{}}
            None                                                   \\
            Randomized Nystr\"om {\scriptsize(\Cref{sec:nystrom})} \\
            RPCholesky  {\scriptsize(\Cref{sec:rpcholesky})}
        \end{tabular}
    };

    \draw[->, thick] (step1L.south) -- (step2L.north);
    \draw[->, thick] (step1R.south) -- (step2R.north);

    % --- CONVERGENCE (Final block) ---

    % Merge at a fixed level below the lower step2 box (step2R, the taller one)
    \path (step2R.south) ++(0,-0.5cm) coordinate (joinR);
    \path (step2L.south |- joinR) coordinate (joinL);

    \draw[thick] (step2L.south) -- (joinL);
    \draw[thick] (step2R.south) -- (joinR);
    \draw[thick] (joinL) -- (joinR) coordinate[midway] (joinM);

    % Final box placed clearly below the merge line as the next step
    \path (joinM) ++(0,-0.6cm) coordinate (final_pos);
    \node[final, anchor=north] (final) at (final_pos) {\textbf{Linesearch
        and Update} \\ {\scriptsize(\Cref{sec:putting_together})} \\ $\alpha =
        \mathtt{Linesearch}(\vtheta, -\vd)$ \\ $\vtheta \gets \vtheta - \alpha
        \vd$};
    \draw[->, thick] (joinM) -- (final.north);

\end{tikzpicture}
    \caption{Modular framework for Natural Gradient Descent. The computation of
        the natural gradient direction is abstracted into a solver routine, which
        can be implemented using either a direct solver (left branch) or an
        iterative matrix-free solver (right branch). The iterative solver further
        supports different strategies for computing matrix--vector products
        (matvecs) with the Gramian and various preconditioning techniques.
        Once the search direction is computed, a line search is performed to update
        the parameters.}
    \label{fig:ngd_framework_tree}
\end{figure}

\subsection{Assembly of the Gramian and matvecs}
\label{sec:assembly_and_matvecs}
For simplicity, we present the case where $k = 1$, i.e.,
$\calF_{u_{\mathtt{sg}(\vtheta)}} u_{\vtheta}(\vx)$ is a scalar. The vector
case follows in a completely analogous manner by summing the contributions of
the $k$ components.

In the subsequent complexity analyses, we let \fp\ denote the flops cost of a single forward pass
$\calF_{u_{\mathtt{sg}(\vtheta)}} u_{\vtheta}(\vx)$. We recall that when using
automatic differentiation (AD), the computational cost of each
Jacobian-Vector-Product (JVP) and Vector-Jacobian-Product (VJP) is a small
constant multiple of the forward pass cost \cite{griewank_evaluating_2008,margossian_review_2019}.
However, the memory requirements differ between reverse-mode AD (VJP) and
forward-mode AD (JVP). For a feedforward neural network, the cost of a forward
pass scales as $\calO(p)$, linearly with the number of network parameters. Note,
however, that when executed on modern GPUs, batch evaluation across $q$ points
benefits significantly from parallelization.

\subsubsection{Computing the full Gramian}
\label{sec:full_gramian}
Evaluating the Gramian as in \eqref{eq:G_matrixfree} requires computing
\begin{equation*}
    \mJ_{\vtheta}\vF(\vtheta) = \begin{bsmallmatrix}
        \nabla_{\vtheta}\calF_{u_{\mathtt{sg}(\vtheta)}} u_{\vtheta}(\vx_1)^\top\\
        \vdots\\
        \nabla_{\vtheta}\calF_{u_{\mathtt{sg}(\vtheta)}} u_{\vtheta}(\vx_q)^\top
    \end{bsmallmatrix} \in\R^{q \times p} \quad \text{where} \quad  \nabla_{\vtheta}\calF_{u_{\mathtt{sg}(\vtheta)}} u_{\vtheta}(\vx_i) \in\R^{p} \text{ for all } i=1,\dots, q.
\end{equation*}
Using reverse-mode AD, this operation requires $\calO(q\fp)$ flops
\cite{novak_fast_2022}, which is essentially equivalent to the cost of computing
the loss gradient $\nabla_{\vtheta}L(\vtheta)$. Subsequently, the tensor
contraction to form the Gramian via \eqref{eq:G_matrixfree} requires
$\calO(p^2q)$ flops. The memory footprint is approximately $\calO(pq)$,
dominated by the storage of the full Jacobian. If this exceeds available memory,
the quadrature points can be partitioned into $N_B$ batches of size $B$. The
Jacobian and the contraction can then be computed $B$ rows at a time, as
illustrated in \Cref{fig:jacobian_batch}. More formally, assuming $q = B\cdot
    N_B$, we rewrite \eqref{eq:G_matrixfree} as
\begin{equation}
    \label{eq:G_matrixfree_batch}
    \begin{split}
         & \mG(\vtheta) = \sum_{b=1}^{N_B} \left(\mJ_{\vtheta}\vF_b(\vtheta)\right)^\top \begin{bsmallmatrix}
                                                                                             w_{(b-1)B+1} \varphi_{u_{\vtheta}}(\vx_{(b-1)B+1}) & & \\
                                                                                             & \ddots & \\
                                                                                             & & w_{bB} \varphi_{u_{\vtheta}}(\vx_{bB})
                                                                                         \end{bsmallmatrix} \mJ_{\vtheta}\vF_b(\vtheta), \\
         & \vF_b(\vtheta) = \begin{bsmallmatrix}
                                \calF_{u_{\mathtt{sg}(\vtheta)}} u_{\vtheta}(\vx_{(b-1)B+1})\\
                                \vdots\\
                                \calF_{u_{\mathtt{sg}(\vtheta)}} u_{\vtheta}(\vx_{bB})
                            \end{bsmallmatrix},
    \end{split}
\end{equation}
and compute and add the contribution of each batch sequentially. This approach
restricts the memory requirement to $\calO(p^2+pB)$, while the overall FLOP
count remains largely unchanged, though the degree of parallelization is
constrained by the batch size.

\begin{figure}
    \includegraphics[width=\textwidth]{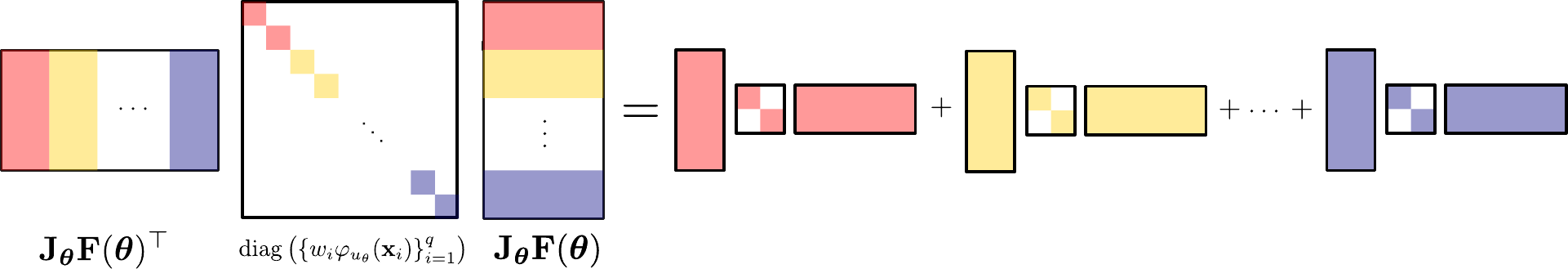}
    \caption{Computation of the Gramian in batches to reduce memory footprint. Each batch of $B$ points is represented by a distinct color. The Jacobian is evaluated for one batch of quadrature points at a time, and the corresponding contribution to the Gramian is computed and accumulated before proceeding to the next batch.}
    \label{fig:jacobian_batch}
\end{figure}

\subsubsection{Matvecs with the Gramian}
\label{sec:matvecs_strategies}
We now discuss strategies for computing matrix--vector products (matvecs) with the
Gramian $\mG(\vtheta)$. We evaluate these strategies based on three metrics:
setup cost (precomputing the auxiliary quantities required to define the
mapping $\vv\mapsto\mG(\vtheta)\vv$), matrix--matrix product (MMP) cost
(evaluating $\mG(\vtheta)\mV$ for $\mV\in\R^{p\times \ell}$), and memory footprint.
Note that the cost of a single matvec is simply the MMP cost evaluated at $\ell = 1$.

\paragraph{Precompute the Gramian}
The first strategy simply consists of assembling the full Gramian as detailed in
\Cref{sec:full_gramian} and subsequently performing MMPs directly. The setup
cost is $\calO(q\fp + p^2q)$, while each MMP requires $\calO(p^2\ell)$ flops. By
processing the quadrature points in batches of size $B$, the memory cost can be
bounded by $\calO(p^2+pB)$.

\paragraph{Precompute the Jacobian only}
A second strategy computes the Jacobian $\mJ_{\vtheta}\vF(\vtheta)$ but avoids
the explicit contraction required to assemble the Gramian. Instead, MMPs are
evaluated in a matrix-free fashion according to \eqref{eq:G_matrixfree}: one
sequentially multiplies by $\mJ_{\vtheta}\vF(\vtheta)$, the diagonal matrix
$\diag\left(w_1\varphi_{u_{\vtheta}}(\vx_1), \dots,
    w_q\varphi_{u_{\vtheta}}(\vx_q)\right)$, and finally $\mJ_{\vtheta}\vF(\vtheta)^\top$.
The setup requires computing $\mJ_{\vtheta}\vF(\vtheta)$ at a cost of
$\calO(q\fp)$, and each MMP costs $\calO(pq\ell)$. However, memory usage cannot
be reduced below $\calO(pq)$ because the entire Jacobian must be stored.

\paragraph{Batch-wise Jacobian recomputation on the fly}
A third strategy avoids precomputation entirely by evaluating the Jacobian
batch-wise and on the fly during each MMP computation. Following
\eqref{eq:G_matrixfree_batch}, for each batch $b = 1, \dots, N_B$, we evaluate
the partial Jacobian $\mJ_{\vtheta}\vF_b(\vtheta)$, compute its matrix-free
contribution to the MMP, and accumulate the result sequentially. This approach
is mathematically equivalent to performing Vector-Jacobian Products (VJPs) and
Jacobian-Vector Products (JVPs) on the batched maps
$\vtheta\mapsto\vF_b(\vtheta)\in\R^B$. The setup cost is strictly negligible
since no precomputation occurs. While each MMP requires
$\calO(q\fp + pq\ell)$ FLOPs, the memory footprint is substantially reduced to
$\calO(pB)$.

\bigskip
A summary of the computational and memory costs for each strategy is provided in
\Cref{tab:matvec_costs}. If memory is not a bottleneck and the number of
simultaneous matvecs $\ell$ is smaller than the parameter count $p$ (which is
typically the case), precomputing only the Jacobian is the most efficient choice
in terms of combined setup and MMP costs. Conversely, if storing the full
Jacobian exceeds available memory, recomputing it batch-wise and on the fly
for each MMP becomes the most viable alternative. For the sake of completeness, we
note that assembling the full Gramian can still be advantageous in the specific
edge case where out-of-memory errors are driven by an exceptionally large number
of quadrature points $q$, provided the parameter count $p$ is small enough that
the $\calO(p^2)$ storage requirement remains easily feasible.

\begin{table}[htp]
    \centering
    \caption{Summary of computational and memory costs for the three
        matrix--vector product strategies. Here $p$ is the number of parameters,
        $\fp$ is the cost of a forward pass, $\ell$ is the number of simultaneous
        matvecs, $q$ is the number of quadrature points, and $B$ is the batch size
        for processing quadrature points. We assume that $p \leq q$, i.e., we focus
        on the regime where a highly accurate quadrature rule is employed and the
        number of quadrature points $q$ exceeds the number of neural network
        parameters $p$.}
    \label{tab:matvec_costs}
    \begin{tblr}{
            colspec = {lccc},
            row{even} = {gray!15}, % Alternates row colors
            row{1} = {white},      % Keeps the header white
        }
        \toprule
        \textbf{Strategy}        & \textbf{Setup Cost}  & \textbf{MMP Cost} & \textbf{Memory Footprint} \\
        \midrule
        Precompute Gramian       & $\calO(q\fp + p^2q)$ & $\calO(p^2\ell)$  & $\calO(p^2+pB)$           \\
        Precompute Jacobian only & $\calO(q\fp)$        & $\calO(pq\ell)$   & $\calO(pq)$               \\
        {Recompute batch Jacobian                                                                       \\ on the fly} & -- & $\calO(q\fp + pq\ell)$ & $\calO(pB)$ \\
        \bottomrule
    \end{tblr}
\end{table}

\subsubsection{Computing the diagonal of the Gramian}
\label{sec:gram_diagonal}
Computing the diagonal of the Gramian $\mG(\vtheta)$ is often useful. While this
is trivial if the full Gramian has already been assembled, it can also be
computed efficiently without forming the full matrix by evaluating only the
diagonal entries of \eqref{eq:G_matrixfree}. Specifically, it holds that
\begin{equation}
    \label{eq:gram_diagonal}
    \diag\left(\mG(\vtheta)\right) = \sum_{i=1}^q w_i \, \varphi_{u_{\vtheta}}(\vx_i)\, \left(\nabla_{\vtheta}\calF_{u_{\mathtt{sg}(\vtheta)}} u_{\vtheta}(\vx_i)\right)^{\odot 2},
\end{equation}
where $^{\odot 2}$ denotes the element-wise square. The accumulation in
\eqref{eq:gram_diagonal} requires $\calO(pq)$ flops and can be straightforwardly
executed in batches, thereby reducing the memory footprint to $\calO(pB)$.

\subsubsection{The case of Finite Element Interpolated Neural Networks}
As described in \Cref{sec:feinns}, in FEINNs the neural network is interpolated
onto the finite element (FE) space, and integration is performed at the FE
level. Consequently, the Gramian assumes the form given in
\eqref{eq:feinn_gramian}. Therefore, in the computational cost analyses
presented above for assembling the Gramian and performing associated matvecs,
the number of quadrature points $q$ is replaced by the number of interior mesh
nodes $N_I$, and \fp\ denotes the cost of a single forward pass $\NN(\vx)$.
Furthermore, we must account for the cost of multiplying by the FEM operator
$\mG_h$. Letting $\mvgh = \mvgh(N_I)$ denote the cost of a matvec with $\mG_h$,
the additional cost to compute the full Gramian is $\calO(p\,\mvgh)$. For each
matrix-free MMP when only the full Jacobian is precomputed, the additional cost
is $\calO(\ell\,\mvgh)$.

Regarding batch processing, the FEM operator $\mG_h$ is typically very sparse
but not diagonal. Consequently, batching cannot be performed as trivially as
before, since the interactions between different rows must be addressed. Let
$\{I_b\}_{b=1}^{N_B}$ represent the index sets for the batches of interior
nodes. For each batch $b = 1, \dots, N_B$, we identify the set of interacting
indices, $J_b = \{j: [\mG_h]_{ij} \neq 0 \text{ for some } i \in I_b\}$, and
sequentially accumulate the batch contributions as follows:
\begin{equation}
    \label{eq:feinn_gramian_batch}
    \begin{split}
         & \mG(\vtheta) = \sum_{b=1}^{N_B} \left(\mJ_{\vtheta}\mP(\vtheta)[I_b, :]\right)^\top \mG_h[I_b, J_b] \left(\mJ_{\vtheta}\mP(\vtheta)[J_b, :]\right), \\
         & \mJ_{\vtheta}\mP(\vtheta)[I, :] = \begin{bsmallmatrix}
                                                 \nabla_{\vtheta}\NN(\vx_{i_1})^\top\\
                                                 \vdots\\
                                                 \nabla_{\vtheta}\NN(\vx_{i_{|I|}})^\top
                                             \end{bsmallmatrix}\in\R^{|I|\times p}, \qquad I = \{i_1, \dots, i_{|I|}\}.
    \end{split}
\end{equation}
To minimize redundant row computations in the Jacobian, it is advantageous to
group spatially proximate interior nodes into the same batch, ensuring that the
set of interacting indices $J_b$ remains comparable in size to $I_b$. By
defining $B = \max_b |J_b|$, the memory footprint is constrained to $\calO(pB)$.
The overall FLOP count is more complex to analyze, as it depends heavily on the
sparsity pattern of $\mG_h$ and the specific batching strategy; however, it is
expected to be on the order of $\calO(N_I\fp + pN_I\ell + \ell\, \mvgh)$ for MMPs,
and $\calO(p\,\mvgh)$ for computing the Gramian diagonal.

\subsection{Preconditioning matrix-free Natural Gradient Descent}
\label{sec:preconditioning}
When iteratively solving the regularized linear system $\left(\mG(\vtheta) + \mu
    \mI\right)\vd = \nabla_{\vtheta} L(\vtheta)$, the Conjugate Gradient (CG) method
is typically employed. In practice, however, $\mG = \mG(\vtheta)$ is often
ill-conditioned, leading to slow convergence unless a suitable preconditioner is
applied. Furthermore, $\mG$ lacks the sparsity or grid-based structure typical
of FEM and FD matrices, which are essential for classical preconditioners such
as Incomplete Cholesky, multigrid/multilevel methods, and domain decomposition
techniques. Instead, the main characteristics of $\mG$ are that it is symmetric
positive semidefinite, it typically exhibits rapid singular value decay (see
\Cref{fig:decay_sample}), and it allows for efficient matvecs and efficient
access to its diagonal. Given these properties, Randomized Nystr{\"o}m
\cite{tropp_fixed-rank_2017,frangella_randomized_2023} and Randomly Pivoted
Partial Cholesky (RPCholesky)
\cite{chen_randomly_2024,epperly_make_2025,epperly_embrace_2025,steinerberger_randomly_2024}
represent, to the best of our knowledge, the most effective methods currently
available for constructing a preconditioner. In this section, we review both
Randomized Nystr{\"o}m preconditioning and RPCholesky, explaining how to
integrate them into our framework to accelerate matrix-free NGD.

\subsubsection{Randomized Nystr{\"o}m}
\label{sec:nystrom}
Let $\mG \in \R^{p \times p}$ be a symmetric positive semidefinite matrix.
The \emph{randomized} Nystr{\"o}m approximation \cite[Section
    14]{martinsson_randomized_2020} with a random test matrix $\mOmega \in \R^{p
        \times \ell}$ is defined as
\begin{equation}
    \label{eq:nystrom_approx}
    \hat{\mG}_{\mathrm{nys}} \defas (\mG\mOmega)\left(\mOmega^\top\mG\mOmega\right)^{\dagger}(\mG\mOmega)^\top.
\end{equation}
The randomness in $\mOmega$ ensures that $\hat{\mG}_{\mathrm{nys}}$ approximates
$\mG$ well with high probability \cite[Section
    14.4]{martinsson_randomized_2020}. In this work, we choose $\mOmega$ to be
either a standard normal matrix (i.e., with i.i.d. standard Gaussian entries) or
a sampling matrix \cite{alaoui_fast_2015,bach_sharp_2013} (i.e., with columns
sampled uniformly at random from the standard basis). While other choices for
$\mOmega$ are possible, such as structured random matrices
\cite{avron_faster_2017,martinsson_randomized_2020,murray_randomized_2023,halko_finding_2011},
they are beyond the scope of this study.
For a standard Gaussian test matrix, the randomized Nystr{\"o}m approximation
admits strong theoretical guarantees
\cite[Theorem 14.1]{martinsson_randomized_2020}. In particular, for any target
rank $k < \ell - 1$, it provides a good low-rank approximation of $\mG$ in
expectation, with error close to the optimal rank-$k$ error up to an additive
term depending on the eigenvalue tail, which decays as the sketch size
$\ell$ increases. Consequently, the Nystr{\"o}m approximation is especially
effective when $\mG$ exhibits strong spectral decay, as is the case in our
setting.

We emphasize that the formula \eqref{eq:nystrom_approx} is
numerically unstable and should not be used directly. A more robust and
efficient implementation, which also returns the eigendecomposition of
$\hat{\mG}_{\mathrm{nys}}$, is provided in \cite[Algorithm
    16]{martinsson_randomized_2020} and reported in
\Cref{apx:sec:randomized_nystrom}.

\paragraph{Randomized Nystr{\"o}m preconditioning}
Given an eigendecomposition $\hat{\mG}_{\mathrm{nys}} = \mU \hat{\mLambda}
    \mU^\top$ with $\hat{\mLambda} = \diag(\hat{\lambda}_1, \dots,
    \hat{\lambda}_\ell)$ and $\hat{\lambda}_1 \geq \dots \geq \hat{\lambda}_\ell$,
the randomized Nystr{\"o}m preconditioner \cite{frangella_randomized_2023} for
$(\mG + \mu \mI)$ with $\mu \geq 0$  is constructed as
\begin{equation}
    \label{eq:nystrom_inverse}
    \mP^{-1} = \mU(\hat{\mLambda} + \mu\mI)^{-1}\mU^\top + \frac{1}{\hat{\lambda}_{\ell} + \mu}(\mI - \mU\mU^\top).
\end{equation}
To ensure good performance in expectation, the sampling parameter $\ell$ should
be scaled proportionally to the effective dimension $d_{\mathrm{eff}}(\mu) =
    \sum_{i=1}^p \lambda_i(\mG) / (\lambda_i(\mG) + \mu)$. This metric acts as a
smoothed count of the eigenvalues of $\mG$ that exceed $\mu$ \cite[Theorem
    1.1]{frangella_randomized_2023}. When $\mG$ exhibits strong spectral decay,
$d_{\mathrm{eff}}(\mu) \ll p$, making the randomized Nystr{\"o}m preconditioner
particularly effective \cite[Theorem 5.1]{frangella_randomized_2023}.

\paragraph{Computational complexity}
Each application of the randomized Nystr{\"o}m preconditioner costs
$\calO(p\ell)$ flops. Beforehand, one must compute the Nystr{\"o}m approximation
itself. Let $\rmC_{\mathrm{mv}}(\ell)$ denote the cost of computing the MMP
$\mG(\vtheta)\mOmega$, which varies depending on the selected strategy from
\Cref{sec:matvecs_strategies}. The overall cost to construct the randomized
Nystr{\"o}m approximation is $\calO(\rmC_{\mathrm{mv}}(\ell) + p \ell^2)$ flops,
with a memory requirement of $\calO(p\ell)$. This overhead is in addition to the
specific setup costs and memory footprint of the chosen matvec strategy.

\subsubsection{Randomly Pivoted Partial Cholesky (RPCholesky)}
\label{sec:rpcholesky}
Randomly Pivoted Partial Cholesky (RPCholesky)
\cite{chen_randomly_2024,epperly_embrace_2025,epperly_make_2025} is a randomized
variant of the standard Pivoted Partial Cholesky algorithm
\cite{cox_analysis_1990}. It computes a rank-$\ell$ approximation $\mG \approx
    \hat{\mG}_{\mathrm{RPChol}} \defas \mF\mF^\top$, where $\mF\in\R^{p \times \ell}$, for a symmetric positive
semidefinite matrix $\mG\in\R^{p \times p}$. This is achieved by performing
$\ell$ steps of the pivoted Cholesky decomposition, where pivots are selected by
sampling proportionally to the diagonal entries of the residual matrix. Starting
with the initial residual $\mG^{(0)} = \mG$ and its diagonal $\vd^{(0)} =
    \diag(\mG^{(0)})$, RPCholesky iterates the following steps for $i = 1, \dots,
    \ell$:
\begin{enumerate}
    \item Sample the $i$th pivot $s_i$ according to the current diagonal:
          \[
              s_i \sim \diag(\mG^{(i-1)}) = \vd^{(i-1)},
          \]
          where $s \sim \vd\in\R^p_+$ denotes the probability distribution $\mathbb{P}(s = j) = d_j /
              \sum_k d_k$.
    \item Extract and rescale the selected pivot column:
          \[
              \begin{split}
                  \vf^{(i)} & = \mG^{(i-1)}[:,s_i] = \mG[:,s_i] - \mF[:,1:i-1]\mF[s_i,1:i-1]^\top, \\
                  \mF[:, i] & = \vf^{(i)} / \sqrt{\vf^{(i)}_{s_i}}.
              \end{split}
          \]
    \item Update the residual matrix and its diagonal:
          \[
              \begin{split}
                  \mG^{(i)} & = \mG^{(i-1)} - \mF[:, i]\mF[:, i]^\top,          \\
                  \vd^{(i)} & = \vd^{(i-1)} - \left(\mF[:, i]\right)^{\odot 2}.
              \end{split}
          \]
\end{enumerate}
As the update equations demonstrate, the algorithm can be executed without ever
explicitly forming the full residual matrices $\mG^{(i)}$. Instead, it
incrementally builds the approximation factor $\mF$ one column at a time while
updating only the diagonal vector $\vd^{(i)}$. A block variant of the algorithm
selects $b > 1$ pivots simultaneously, performing the same operations on a batch
of $b$ columns at once. Detailed pseudocode for both the standard and block
versions is provided in \Cref{alg:rpcholesky,alg:block_rpcholesky}. Notably, the
approximation error in the trace norm (i.e., the Schatten 1-norm) can be easily
monitored, since
\begin{equation*}
    \norm{\mG - \mF[:, 1:i]\mF[:, 1:i]^\top}_{*} = \trace(\mG^{(i)}) = \sum_{j=1}^p \vd^{(i)}_j.
\end{equation*}
This allows the decomposition to be terminated early once a desired accuracy is
reached.

RPCholesky is mathematically equivalent to the Nystr{\"o}m approximation when
the sampling matrix is constructed from the selected pivots \cite[Property
    2.1]{chen_randomly_2024}. In RPCholesky, however, the indices are sampled
adaptively based on the residual at the current iteration. Trace-norm error bounds show that this adaptive approximation is particularly
effective when $\mG$ exhibits strong spectral decay, as is the case in our
setting \cite[Theorem 4.2]{epperly_embrace_2025}.

\paragraph{RPCholesky preconditioning}
Once the RPCholesky approximation $\hat{\mG}_{\mathrm{RPChol}} = \mF\mF^\top$
has been constructed, the corresponding RPCholesky preconditioner
\cite{diaz_robust_2025} for the regularized matrix $(\mG + \mu\mI)$ with $\mu >
    0$ is defined as
\begin{equation*}
    \mP = \hat{\mG}_{\mathrm{RPChol}} + \mu \mI = \mF\mF^\top + \mu\mI.
\end{equation*}
Note that $\mP$ can be efficiently inverted by computing a thin singular value
decomposition (SVD) $\mF = \mU\mSigma\mV^\top$ as
\begin{equation}
    \label{eq:rpchol_svdprecond}
    \mP^{-1} = \frac{1}{\mu} (\mI - \mU\mU^\top) + \mU(\mSigma^2 + \mu\mI)^{-1}\mU^\top.
\end{equation}
Alternatively, one can compute the $\ell\times\ell$ Cholesky decomposition
$\mF^\top\mF + \mu \mI_{\ell} = \mR^\top\mR$ and use it to efficiently invert
$\mP$ via the Woodbury matrix identity:
\begin{equation}
    \label{eq:rpchol_precond}
    \mP^{-1} = \frac{1}{\mu}\mI - \frac{1}{\mu}\mF\mR^{-1}\mR^{-\top}\mF^\top.
\end{equation}
Although this second alternative is theoretically less numerically stable, we
did not encounter stability issues in our experiments. Because it is
significantly faster on GPUs than the SVD-based approach, we adopt this second
formulation.

The performance of the RPCholesky preconditioner is governed by the
\emph{$\mu$-tail rank} of $\mG$, defined as $\rank_\mu(\mG) \defas \min\left\{r
    \geq 0 : \sum_{i > r} \lambda_i(\mG) \leq \mu\right\}$. When $\mG$ exhibits
rapid spectral decay, we have $\rank_\mu(\mG) \ll p$, making RPCholesky
preconditioning particularly efficient \cite[Theorem 2.2]{diaz_robust_2025}.

\paragraph{Computational complexity}
Each application of the RPCholesky preconditioner requires $\calO(p\ell)$ flops.
However, one must also account for the initial cost of computing the RPCholesky
approximation. Let $\rmC_{\mathrm{col}}(b)$ denote the cost of extracting $b$
columns of the matrix $\mG(\vtheta)$. This operation is essentially equivalent
to a matrix--matrix product (MMP) with a $p \times b$ matrix, and its cost varies
depending on the strategy selected from \Cref{sec:matvecs_strategies}. Furthermore,
let $\rmC_{\mathrm{diag}}$ denote the cost of computing the diagonal of $\mG$,
which is necessary for pivot selection and can be evaluated efficiently as
detailed in \Cref{sec:gram_diagonal}. The total cost to construct an RPCholesky
approximation of rank $\ell = tb$ is thus $\calO\left(\ell \cdot\rmC_{\mathrm{col}}(b)/ b  +
    \rmC_{\mathrm{diag}} + p\ell^2\right)$ flops, with a memory footprint of
$\calO(p\ell)$. This overhead is in addition to the specific setup costs and
memory requirements of the chosen matvec strategy. Notably, block processing
yields speedups both in column extraction (by evaluating MMPs in batches) and
in the execution of the algorithm itself (by casting computations as efficient
matrix--matrix operations).

\subsubsection{A comparison of Randomized Nystr{\"o}m and RPCholesky
    preconditioning}
\label{sec:compare_nystrom_rpcholesky}
Randomized Nystr{\"o}m and RPCholesky are both effective when $\mG$ has strong
spectral decay, but they differ in how the low-rank approximation is
constructed and controlled. Randomized Nystr\"om builds a global sketch
$\mG\mOmega$ and, for Gaussian test matrices, enjoys near-optimal low-rank
approximation guarantees in expectation. Moreover, the quality of the resulting
preconditioner is governed by the effective dimension $d_{\mathrm{eff}}(\mu)$,
which acts as a smoothed count of the eigenvalues above the regularization
threshold $\mu$. RPCholesky instead incrementally selects columns by sampling
from the diagonal of the current residual. It is therefore mathematically
equivalent to a Nystr\"om approximation whose sampling matrix is formed by the
selected pivots, but with the important distinction that these pivots are chosen
adaptively from the residual at each iteration. This makes RPCholesky more
directly controllable: the trace-norm approximation error can be monitored
during the factorization, the rank can be selected adaptively rather than fixed
in advance, and the quality of the preconditioner is governed by the
$\mu$-tail rank of $\mG$.

For a rank-$\ell$ approximation, Nystr\"om requires $\ell$ matvecs, whereas
RPCholesky needs access to $\ell$ columns of $\mG$. Their computational costs
are $\calO(\rmC_{\mathrm{mv}}(\ell) + p \ell^2)$ for Nystr\"om and
$\calO\left(\ell \cdot\rmC_{\mathrm{col}}(b)/ b + \rmC_{\mathrm{diag}} +
    p\ell^2\right)$ for blocked RPCholesky with block size $b$. These
complexity estimates highlight the main practical differences between the two
approaches. First, RPCholesky must compute the diagonal of $\mG$ to define the
sampling probabilities. Second, and more importantly, RPCholesky accesses the
columns of $\mG$ sequentially in blocks of size $b$, whereas Nystr\"om forms
$\mG \mOmega$ by performing the required $\ell$ matvecs in a single batched
operation. This makes Nystr\"om easier to parallelize and particularly
advantageous for large-scale problems in which batching is essential, since it
avoids the $\calO(q\fp)$ overhead of recomputing the Jacobian for each batch in
every MMP; see \Cref{tab:matvec_costs}. RPCholesky, on the other hand, is
preferable when efficient access to the diagonal and to selected columns is
available, and when one wants finer online control of the approximation quality
through adaptive rank selection and residual monitoring.

\subsection{Computing the regularization parameter \texorpdfstring{$\mu$}{mu}}
\label{sec:regularization_mu}
The regularization parameter $\mu > 0$ is crucial for ensuring numerical
stability and mitigating the slow initial convergence typical of
quasi-second-order methods like NGD. Prior works have proposed heuristics that
adapt $\mu$ dynamically
\cite{bonfanti_challenges_2024,gavin_levenberg-marquardt_2019} or decrease it
proportionally with the loss \cite{zeinhofer_unified_2024}, allowing early
iterations to mimic gradient descent before gradually shifting to NGD near a
local minimum. We instead propose a heuristic based on the numerical rank of
$\mG(\vtheta)$.

Let $\lambda_1 \geq \dots \geq \lambda_p$ denote the eigenvalues of the Gramian
$\mG(\vtheta)$. We aim at setting
the regularization parameter to an approximation of $\mu \approx \gamma \cdot
    \epsilon_{\mathrm{mach}} \cdot \lambda_1$, with $\gamma = 10$ as a reasonable
default. This corresponds to the threshold commonly employed to determine the
numerical rank, i.e., the cut-off below which eigenvalues are considered to be
numerically zero \cite{press_numerical_2007}. Although the exact
largest eigenvalue $\lambda_1$ is typically not directly available, an estimate
$\hat{\lambda}_1$ is often readily accessible, for instance, from the
Nystr{\"o}m approximation. Alternatively, it can be efficiently computed using a
few iterations of the power method (we found 2 to 4 iterations to be
sufficient). Note that more advanced strategies for the scheduling of $\mu$, based on the
knowledge of all eigenvalues of the Gramian, have recently appeared in the
literature \cite{Schwencke2025}. The inclusion of those strategies in
the framework considered in this paper is a direction of future work.

Finally, we observed that it can be beneficial to enforce stronger
regularization in early iterations, when the optimization is still far from a
local minimum. This can be achieved using a heuristic of the form $\mu =
    \max(\gamma \cdot \epsilon_{\mathrm{mach}} \cdot \hat{\lambda}_1, c \cdot
    \norm{\nabla L(\vtheta)}^{\alpha})$ or $\mu = \max(\gamma \cdot
    \epsilon_{\mathrm{mach}} \cdot \hat{\lambda}_1, c \cdot |L(\vtheta)|^{\alpha})$,
where the scaling factor $c$ and the exponent $\alpha$ are problem-dependent.

\subsection{Putting it all together: Preconditioned matrix-free NGD algorithms}
\label{sec:putting_together}
We are now ready to combine the ingredients described so far
and introduce
the two primary algorithmic contributions of this paper: matrix-free NGD with
randomized Nystr{\"o}m preconditioning (Nystr{\"o}mNGD) and matrix-free NGD with
RPCholesky preconditioning (RPCholNGD). To emphasize the modularity of our
framework, we first formulate Natural Gradient Descent as a meta-algorithm
(\Cref{alg:ngd_framework}) that delegates the computation of the descent
direction to a specific solver routine. This allows us to seamlessly swap
between a direct solver based on the full Gramian (\Cref{alg:dir_full_gramian})
and a matrix-free iterative solver (\Cref{alg:dir_matrix_free}). The matrix-free
solver, in turn, accepts a preconditioner builder, enabling the construction of
Nystr{\"o}mNGD (\Cref{alg:prec_nystrom}) and RPCholNGD (\Cref{alg:prec_rpchol})
simply by exchanging the preconditioning strategy.

A key advantage of Nystr{\"o}mNGD and RPCholNGD over standard NGD is their
reduced and tunable memory footprint. Excluding the memory required for the
matvec function, which depends on the strategy chosen from
\Cref{sec:matvecs_strategies} but can be bounded by $\calO(pB)$ using a batch
size $B$, the preconditioner requires storing at most $2\ell$ vectors of length
$p$. This footprint is comparable to the memory complexity of L-BFGS with a
history size of $\ell$. Consequently, the parameter $\ell_{\max}$ provides
explicit control over memory consumption, allowing the method to scale to larger
networks. Naturally, smaller values of $\ell_{\max}$ may result in a weaker
preconditioner and a higher number of inner pCG iterations, but the algorithm
remains fully operational.

\paragraph{Algorithmic details}
In the general NGD meta-algorithm, once the search direction $\vd_k$ is computed,
a line search is performed to update the iterate (lines~\ref{alg_line:linesearch}
and~\ref{alg_line:update}). In practice, we found that a standard Armijo
backtracking line search \cite{wright_numerical_2006} suffices, bypassing the
need for logarithmic grid searches as proposed in
\cite{muller_achieving_2023,jnini_gauss-newton_2024}.

\paragraph{Nystr{\"o}mNGD: Matrix-free NGD with randomized Nystr{\"o}m
    preconditioning} In Nystr{\"o}mNGD (\Cref{alg:prec_nystrom}), the randomized
Nystr{\"o}m approximation already yields an estimate $\hat{\lambda}_1$ of the
largest eigenvalue of $\mG(\vtheta)$, which we use to compute the
regularization parameter $\mu$. Furthermore, the rank parameter $\ell$ dictates
the computational cost of the preconditioner and should ideally scale with the
effective dimension $d_{\mathrm{eff}}(\mu)$ \cite[Theorem 5.1]{frangella_randomized_2023}.
Since $d_{\mathrm{eff}}(\mu)$ is not directly accessible, we adaptively
estimate $\ell$. Following \cite[Section 5.4.2]{frangella_randomized_2023},
we achieve this by monitoring the ratio $\hat{\lambda}_\ell / \mu$: if it
exceeds $10$, we increase $\ell$ up to a prescribed maximum $\ell_{\max}$
(line~\ref{alg_line:increase_ell_nystrom}). Otherwise, we reduce $\ell$ to
the smallest index satisfying $\hat{\lambda}_\ell / \mu < 10$, adding a
small offset (e.g., $1$ or $2$) to prevent oscillations
(line~\ref{alg_line:decrease_ell_nystrom}).

\paragraph{RPCholNGD: Matrix-free NGD with RPCholesky preconditioning}
In RPCholNGD (\Cref{alg:prec_rpchol}), the RPCholesky decomposition does not
naturally provide an approximation of the largest eigenvalue of $\mG(\vtheta)$.
We therefore estimate it using the power method.
Additionally, the rank parameter $\ell$ is determined adaptively by terminating
the RPCholesky factorization once the nuclear norm error falls below the threshold
$\varepsilon = \norm{\mu\mI}_* = \mu\cdot p$.

\begin{algorithm}[htp]
    \begin{algorithmic}[1]
        \Require{Initial parameter $\vtheta_0 \in \R^p$, maximum iterations $K$, solver routine $\mathtt{ComputeDirection}$, initial solver state $\mathcal{S}_0$}
        \Ensure{Sequence of iterates $\vtheta_1, \dots, \vtheta_K$}
        \vspace{0.5pc}
        \For{$k = 0, 1, \dots, K-1$}
        \State $\vg_k = \nabla_{\vtheta} L(\vtheta_k)$ \Comment{Compute Euclidean gradient}
        \State $[\vd_k, \mathcal{S}_{k+1}] = \mathtt{ComputeDirection}(\vtheta_k, \vg_k, \mathcal{S}_k)$ \Comment{Compute NGD direction and update state}
        \State $\alpha_{k} = \mathtt{Linesearch}(\vtheta_k, -\vd_k)$
        \Comment{Linesearch along $\alpha \mapsto \vtheta_k - \alpha \vd_k$} \label{alg_line:linesearch}
        \State $\vtheta_{k+1} = \vtheta_{k} - \alpha_k \vd_k$ \Comment{Update
            parameters} \label{alg_line:update}
        \EndFor
    \end{algorithmic}
    \caption{Modular Natural Gradient Descent}
    \label{alg:ngd_framework}
\end{algorithm}

\begin{algorithm}[htp]
    \begin{algorithmic}[1]
        \Require{Current parameter $\vtheta$, gradient $\vg$, $\gamma \in\R_+$ (e.g., $\gamma = 10$)}
        \Ensure{NGD direction $\vd$, empty state $\emptyset$}
        \vspace{0.5pc}
        \State Assemble full Gramian $\mG(\vtheta)$ \Comment{\Cref{sec:full_gramian}}
        \State Estimate $\lambda_1 \left(\mG(\vtheta)\right) \approx \hat{\lambda}_1$ via power method (e.g., 2--4 iterations)
        \State $\mu = \gamma \cdot \hat{\lambda}_1 \cdot \epsilon_{\mathrm{mach}}$ \Comment{\Cref{sec:regularization_mu}}
        \State Compute Cholesky decomposition $\mG(\vtheta) + \mu\mI = \mR^\top\mR $
        \State Solve $\vd = \mR^{-1}\mR^{-\top}\vg$ \Comment{Forward and backward substitution}
        \State \Return $[\vd, \emptyset]$
    \end{algorithmic}
    \caption{$\mathtt{ComputeDirection}$: Direct Full Gramian Solver}
    \label{alg:dir_full_gramian}
\end{algorithm}

\begin{algorithm}[htp]
    \begin{algorithmic}[1]
        \Require{Current parameter $\vtheta$, gradient $\vg$, state $\mathcal{S}$, builder $\mathtt{BuildPreconditioner}$, CG parameters $\mathtt{maxit}$}
        \Ensure{NGD direction $\vd$, updated state $\mathcal{S}_{\mathrm{next}}$}
        \vspace{0.5pc}
        \State Define $\mG_{\mathrm{fun}} = \vv \mapsto \mG(\vtheta) \vv$ via
        \eqref{eq:G_matrixfree} \Comment{Matvecs with Gramian, see
            \Cref{sec:matvecs_strategies}}
        \State $[\mP_{\mathrm{fun}}, \mu, \mathcal{S}_{\mathrm{next}}] = \mathtt{BuildPreconditioner}(\mG_{\mathrm{fun}}, \vtheta, \vg, \mathcal{S})$ \Comment{Build preconditioner and compute $\mu$}
        \State $\vd = \texttt{pCG}(\mG_{\mathrm{fun}} + \mu \cdot \mathrm{Id},\,
            \vg,\, \mathtt{maxit},\, \mP_{\mathrm{fun}})$
        \Comment{Solve iteratively for $\vd$}
        \State \Return $[\vd, \mathcal{S}_{\mathrm{next}}]$
    \end{algorithmic}
    \caption{$\mathtt{ComputeDirection}$: Preconditioned Matrix-Free Solver}
    \label{alg:dir_matrix_free}
\end{algorithm}

\begin{algorithm}[htp]
    \begin{algorithmic}[1]
        \Require{Matvec function $\mG_{\mathrm{fun}}$, current parameter
            $\vtheta$, gradient $\vg$, state $\mathcal{S} = \{\ell\}$, $\gamma \in
                \R_+$ (e.g., $\gamma = 10$), maximum rank $\ell_{\max}$, rank increase $\delta\ell$}
        \Ensure{Matrix-free preconditioner $\mP_{\mathrm{fun}}$, regularization parameter $\mu$, updated state $\{\ell_{\mathrm{next}}\}$}
        \vspace{0.5pc}
        \State $[\mU, \hat{\mLambda}] =
            \mathtt{RandomizedNystrom}(\mG_{\mathrm{fun}}, \ell)$
        \Comment{Randomized Nystr{\"o}m approximation,
            see \Cref{apx:sec:randomized_nystrom}}
        \State $\mu = \gamma \cdot \hat{\lambda}_1 \cdot \epsilon_{\mathrm{mach}}$ \Comment{\Cref{sec:regularization_mu}}
        \State $\mP_{\mathrm{fun}} = \vv \mapsto (\hat{\lambda}_{\ell} +
            \mu) \mU(\hat{\mLambda} + \mu\mI)^{-1}\mU^\top \vv + (\mI -
            \mU\mU^\top) \vv$ \Comment{Nystr{\"o}m preconditioner \eqref{eq:nystrom_inverse}}
        \If{$\hat{\lambda}_\ell > 10 \mu$}
        \State $\ell_{\mathrm{next}} = \min(\ell + \delta\ell,
            \ell_{\max})$ \Comment{Increase rank if smallest eigenvalue is too
            large} \label{alg_line:increase_ell_nystrom}
        \Else
        \State $\ell_{\mathrm{next}} = \min\left( \min\{i : \hat{\lambda}_i
            < 10 \mu\} + 1,\, \ell_{\max}\right)$ \Comment{Decrease rank if
            decay is reached early} \label{alg_line:decrease_ell_nystrom}
        \EndIf
        \State \Return $[\mP_{\mathrm{fun}}, \mu, \{\ell_{\mathrm{next}}\}]$
    \end{algorithmic}
    \caption{$\mathtt{BuildPreconditioner}$: Nystr{\"o}m Preconditioning (Nystr{\"o}mNGD)}
    \label{alg:prec_nystrom}
\end{algorithm}

\begin{algorithm}[htp]
    \begin{algorithmic}[1]
        \Require{Matvec function $\mG_{\mathrm{fun}}$, current parameter $\vtheta$, gradient $\vg$, state $\mathcal{S} = \emptyset$, $\gamma \in \R_+$ (e.g., $\gamma = 10$), maximum rank $\ell_{\max}$}
        \Ensure{Matrix-free preconditioner $\mP_{\mathrm{fun}}$, regularization parameter $\mu$, empty state $\emptyset$}
        \vspace{0.5pc}
        \State Compute diagonal $\vd^{(0)} = \diag(\mG(\vtheta))$
        \Comment{\Cref{sec:gram_diagonal}}
        \State Estimate $\lambda_1 \left(\mG(\vtheta)\right) \approx \hat{\lambda}_1$ via power method (e.g., 2--4 iterations)
        \State $\mu = \gamma \cdot \hat{\lambda}_1 \cdot \epsilon_{\mathrm{mach}}$ \Comment{\Cref{sec:regularization_mu}}
        \State $\mF = \mathtt{RPCholesky}(\mG_{\mathrm{fun}},\,\vd^{(0)},\, \ell_{\max},\,
            \varepsilon = \mu\cdot p)$ \Comment{RPCholesky approximation, see
            \Cref{apx:sec:rpchol}}
        \State Compute Cholesky decomposition $\mF^\top\mF + \mu\mI_\ell = \mR^\top\mR$
        \State $\mP_{\mathrm{fun}} = \vv \mapsto \mu^{-1}\vv - \mu^{-2}\mF\mR^{-1}\mR^{-\top}\mF^\top \vv$ \Comment{RPCholesky preconditioner \eqref{eq:rpchol_precond}}
        \State \Return $[\mP_{\mathrm{fun}}, \mu, \emptyset]$
    \end{algorithmic}
    \caption{$\mathtt{BuildPreconditioner}$: RPCholesky Preconditioning (RPCholNGD)}
    \label{alg:prec_rpchol}
\end{algorithm}

\section{Numerical experiments}
\label{sec:numerical_experiments}
We evaluate the performance of Nystr{\"o}mNGD and RPCholNGD on a diverse set of
PDE problems with known analytical solutions. Our benchmarks encompass both
strong formulations, discretized via Physics-Informed Neural Networks (PINNs),
and weak formulations, discretized using Finite Element Interpolated Neural
Networks (FEINNs). The code is publicly available at
\url{https://github.com/IvanBioli/preconditioned-natural-gradient.git}: it provides
the Python package \texttt{randngd} (also on PyPI: \texttt{pip install randngd}),
which implements the NGD methods of this paper, and example scripts that run all the
experiments below; version 0.1.1 is archived on Zenodo~\cite{bioli_randngd_2026}.

\paragraph{Evaluated methods and baselines}
In \Cref{sec:comparison_ngd}, we compare several algorithmic configurations within our modular NGD framework
(legend names in parentheses): NGD with a direct solver (NGD full); NGD with a
matrix-free solver and no preconditioning (NGD matrix-free no prec.);
Nystr{\"o}mNGD using a Gaussian random test matrix (Nystr{\"o}mNGD Gaussian);
Nystr{\"o}mNGD with a random sampling test matrix with sampling probability
proportional to $\diag(\mG(\vtheta))$ (Nystr{\"o}mNGD Sampling); and finally,
RPCholNGD.

In \Cref{sec:comparison_others}, then, we benchmark our preconditioned matrix-free NGD methods against state-of-the-art first- and quasi-second-order optimizers commonly employed in the literature. We include SSBroyden due to its recently reported exceptional performance on PINNs \cite{urban_unveiling_2025,kiyani_optimizing_2025}. We also evaluate Adam \cite{kingma_adam_2017}, BFGS \cite{wright_numerical_2006}, and L-BFGS \cite{wright_numerical_2006}, as they are the most standard first-order, quasi-Newton, and limited-memory quasi-Newton algorithms for PINNs, respectively. Additional optimizers, such as Shampoo \cite{gupta_shampoo_2018,anil2021scalable} and Muon \cite{jordan2024muon}, were also tested but are omitted from the final comparison as they did not yield improvements over SSBroyden. Because the
per-iteration computational cost varies significantly across these methods, we
evaluate performance based on both iteration count and wall-clock time.
We remark that our choice of numerical examples is motivated by the goal of
comparing with those present in the NGD
literature, as our primary focus is on improving NGD methods rather than
outperforming all optimizers for all classes of PDEs.

All experiments are conducted in double precision on a single NVIDIA RTX 4090
GPU (24 GB), with results averaged over 10 independent runs. In the resulting
plots, solid lines indicate medians while shaded regions represent the
interquartile range ($25$th to $75$th percentiles).

\paragraph{General setup and hyperparameters}
Unless specified otherwise, experiments employ a fully connected neural network
with three hidden layers of width 64 and $\tanh$ activations. Except in the
scalability experiments in \Cref{sec:scalability}, in all NGD variants we
compute the full Jacobian without batching, and for matrix-free implementations,
we employ the ``Precompute Jacobian only'' strategy (see
\Cref{sec:matvecs_strategies}).
For the preconditioned NGD methods, Nystr{\"o}mNGD and RPCholNGD, we cap the
maximum preconditioner rank at $\ell_{\max} = 500$ and run pCG for a maximum of
$\mathtt{maxit} = 20$ iterations. Within Nystr{\"o}mNGD, the rank increase
parameter is fixed to $\delta\ell = 20$. In RPCholNGD, we use the block variant
of RPCholesky with a block size of $b=20$. To ensure a fair comparison across
methods, the history size for L-BFGS is set to $500$ to match the memory
footprint of the preconditioned NGD methods. Similarly, for the
un-preconditioned matrix-free NGD, we allocate a maximum of $520$ pCG
iterations, matching the total maximum number of matrix--vector products
(matvecs) permitted in the preconditioned variants.

To enforce stronger regularization during the early stages of optimization, if not stated otherwise, all
NGD methods utilize the adaptive heuristic  $\mu = \max(\gamma \cdot
    \hat{\lambda}_1 \cdot \epsilon_{\mathrm{mach}}, 10^{-4} \cdot L(\vtheta)^2)$.
The largest eigenvalue estimate, $\hat{\lambda}_1$, is computed using four
iterations of the power method, except in Nystr{\"o}mNGD where the estimate is
naturally provided by the Nystr{\"o}m approximation. The parameter $\gamma$ is
set to $\gamma = 10$.

A comprehensive description of the problem-specific configurations, including
neural architectures, spatial discretizations, iteration budgets, and baseline
optimizer hyperparameters, as well as an extended sensitivity analysis of
Nystr{\"o}mNGD and RPCholNGD, is provided in \Cref{apx:sec:numerical_details}
and \Cref{apx:sec:sensitivity}.

\subsection{Comparison of natural gradient descent variants}
\label{sec:comparison_ngd}
\subsubsection{Poisson problem in 3D -- PINNs}
\label{sec:poisson3d_pinn}
We consider the Poisson problem on the 3D unit cube
\cite{raissi_physics-informed_2019,cuomo_scientific_2022}
\begin{equation}
    \label{eq:poisson3d}
    \begin{cases}
        -\Delta u = f & \Omega = (0, 1)^3, \\
        u = g_D       & \partial\Omega,
    \end{cases}
\end{equation}
with exact solution $u(x,y,z) = \sin(\pi x)\sin(\pi y)\sin(\pi z)$,  discretized
via Physics-Informed Neural Networks (PINNs) as in
\cite{zeinhofer_unified_2024}. Problem~\eqref{eq:poisson3d} is linear and is
considered in strong form. Using the notation of
\Cref{sec:discretization_of_PDEs}, the function spaces are $\rmV =
    \rmH^2(\Omega)$, $\rmX = \rmL^2(\Omega)$, and $\rmY = \rmL^2(\partial\Omega)$,
with operators
\[
    \calF(u) = -\Delta u, \qquad \calB(u) = \restr{u}{\partial\Omega}.
\]
The corresponding least-squares loss is
\begin{equation}
    \label{eq:energy_poisson}
    E(u) = \frac{1}{2}\norm{\Delta u + f}_{\rmL^2(\Omega)}^2 + \frac{1}{2}\norm{u - g_D}_{\rmL^2(\partial\Omega)}^2,
\end{equation}
and the induced metric reads
\begin{equation}
    \label{eq:metric_poisson}
    \bilinear[g]{u}{v} = \innerh[\rmL^2(\Omega)]{\Delta u}{\Delta v} + \innerh[\rmL^2(\partial\Omega)]{u}{v}
    = \int_{\Omega}\Delta u\,\Delta v\,\mathrm{d}\vx + \int_{\partial\Omega}u\,v\,\mathrm{d}\vs.
\end{equation}

\begin{remark}
    \label{rmk:coercivity_poisson}
    The metric \eqref{eq:metric_poisson} is coercive only with respect to the
    $\rmH^{1/2}(\Omega)$ norm, but not in $\rmH^2(\Omega)$. This is due to the
    use of the $\rmL^2(\partial\Omega)$ norm for the Dirichlet boundary term.
    Coercivity in $\rmH^2(\Omega)$ can be restored by enforcing the boundary
    conditions strongly \cite[Remark~6]{zeinhofer_unified_2024} or by penalizing
    in the $\rmH^{3/2}(\partial\Omega)$ norm \cite{bonito_convergence_2025}.
    However, in line with standard PINN practice, we retain the
    $\rmL^2(\partial\Omega)$ penalty for simplicity.
\end{remark}

\paragraph{Results}
As shown in \Cref{fig:poisson_h1_iter,fig:poisson_h1_time}, NGD with a direct
solver (NGD full), both variants of Nystr{\"o}mNGD, and RPCholNGD achieve
similar accuracy in a comparable number of iterations; however, the matrix-free
preconditioned NGD variants do so in significantly less time. Conversely,
unpreconditioned matrix-free NGD fails to reach the accuracy of the
preconditioned variants, despite being allocated the same budget of Gramian
matvecs. This underscores the crucial role of preconditioning: given a fixed
matvec budget, it is more effective to allocate the majority of it toward
constructing a robust preconditioner. Without preconditioning, the linear system
is solved inaccurately, yielding a poor approximation of the natural gradient
descent direction and ultimately degrading the final accuracy. Furthermore, per
iteration, the preconditioned NGD variants are significantly faster than
unpreconditioned matrix-free NGD. This speedup occurs because assembling the
preconditioner allows for the parallel execution of multiple matvecs with the
Gramian, unlike the sequential matvecs required by CG. Finally,
\Cref{fig:poisson_ell_iter} shows that the rank parameter $\ell$ increases
rapidly during early iterations, consistent with \Cref{fig:poisson_svals_decay},
before stabilizing near the maximum allowed rank $\ell_{\max} = 500$. Note,
however, that despite reaching this maximum rank, the method still achieves a
compression factor of approximately $17\times$ relative to the total number of
parameters $p = 8641$.

\begin{figure}[htbp]
    \centering
    \begin{subfigure}[t]{0.49\textwidth}
        \includegraphics[width=\textwidth]{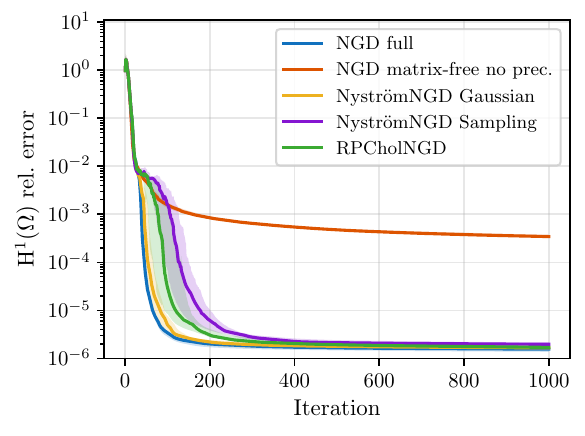}
        \caption{Relative $\rmH^1(\Omega)$ error vs.\ Iterations.}
        \label{fig:poisson_h1_iter}
    \end{subfigure}
    \hfill
    \begin{subfigure}[t]{0.49\textwidth}
        \includegraphics[width=\textwidth]{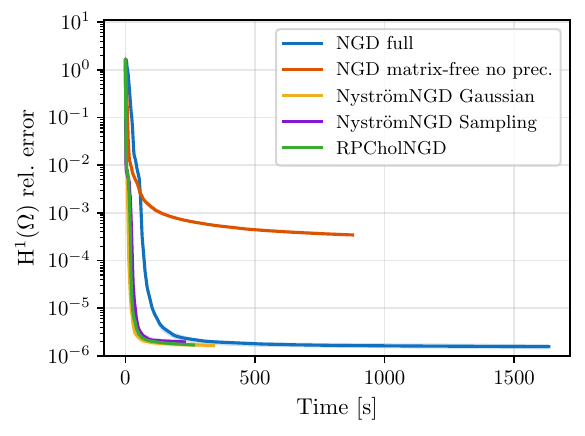}
        \caption{Relative $\rmH^1(\Omega)$ error vs.\ wall-clock time.}
        \label{fig:poisson_h1_time}
    \end{subfigure}
    \hfill
    \begin{subfigure}[t]{0.49\textwidth}
        \includegraphics[width=\textwidth]{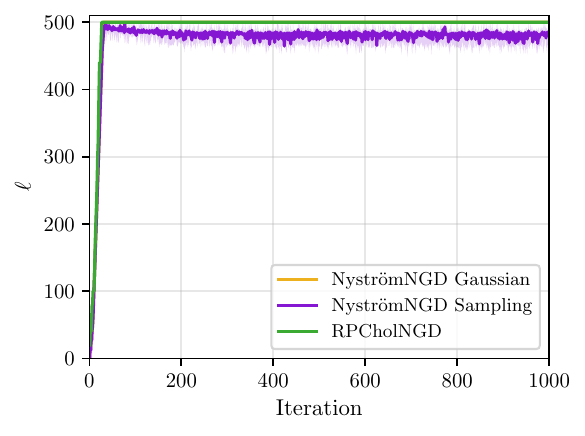}
        \caption{Evolution of the adaptive rank parameter $\ell$ of the
            randomized preconditioners during training; the imposed cap is $\ell_{\max} = 500$.}
        \label{fig:poisson_ell_iter}
    \end{subfigure}
    \hfill
    \begin{subfigure}[t]{0.49\textwidth}
        \includegraphics[width=\textwidth]{figures_new/natgradnew/PINNs/poisson3D_layers3x64x64x64x1/_svals_rel_singleiters.pdf}
        \caption{Decay of the normalized singular values
        $\{\sigma_i/\sigma_1\}_{i=1}^p$ of the Gramian during the first 250
        NGD-full iterations; each color corresponds to a different iterate and
        shows how rapidly the spectrum decays.}
        \label{fig:poisson_svals_decay}
    \end{subfigure}
    \caption{3D Poisson equation on the unit cube discretized via PINNs, with
        exact solution $u(x,y,z)=\sin(\pi x)\sin(\pi y)\sin(\pi z)$. Panels (a) and
        (b) compare NGD full, matrix-free NGD without preconditioning,
        Nystr{\"o}mNGD, and RPCholNGD in terms of relative $\rmH^1(\Omega)$ error
        versus iterations and versus wall-clock time. Panels (c) and (d) report,
        respectively, the adaptive rank of the preconditioners and the singular
        value decay of the Gramian, highlighting the low-rank structure exploited by
        the preconditioners.}
    \label{fig:main-error-poisson}
\end{figure}

\subsubsection{Heat equation in 3+1D -- PINNs}
\label{sec:heat3d1_pinn}
We consider the heat equation
\begin{equation}
    \label{eq:heat}
    \begin{cases}
        \partial_t u - \Delta u = f & \text{in } I \times \Omega = (0,1) \times (0,1)^3, \\
        u = g_D                     & \text{on } I \times \partial\Omega,                \\
        u(0, \cdot) = u_0           & \text{in } \Omega,
    \end{cases}
\end{equation}
with exact solution $u(t,x,y,z) = \big(\cos(\pi x) + \cos(\pi y) + \cos(\pi
        z)\big) e^{-\pi^2 t / 4}$, as in \cite{zeinhofer_unified_2024}. We consider
\eqref{eq:heat} in strong form discretized via PINNs, with function spaces $\rmV = \rmL^2(I,
    \rmH^2(\Omega)) \cap \rmH^1(I, \rmL^2(\Omega))$, $\rmX = \rmL^2(I,
    \rmL^2(\Omega))$, $\rmY = \rmL^2(I, \rmL^2(\partial\Omega)) \times
    \rmL^2(\Omega)$, and operators
\[
    \calF(u) = \partial_t u - \Delta u,
    \qquad
    \calB(u) = \big(\restr{u}{I \times \partial\Omega},\, u(0, \cdot)\big).
\]
The corresponding least-squares loss reads
\begin{equation}
    \label{eq:heat_energy}
    E(u) = \frac{1}{2}\norm{\partial_t u - \Delta u - f}_{\rmL^2(I, \rmL^2(\Omega))}^2
    + \frac{1}{2}\norm{u - g_D}_{\rmL^2(I, \rmL^2(\partial\Omega))}^2
    + \frac{1}{2}\norm{u(0, \cdot) - u_0}_{\rmL^2(\Omega)}^2,
\end{equation}
with induced metric
\begin{equation}
    \label{eq:heat_metric}
    \bilinear[g]{u}{v}
    = \innerh[\rmL^2(I, \rmL^2(\Omega))]{\partial_t u - \Delta u}{\partial_t v - \Delta v}
    + \innerh[\rmL^2(I, \rmL^2(\partial\Omega))]{u}{v}
    + \innerh[\rmL^2(\Omega)]{u(0, \cdot)}{v(0, \cdot)}.
\end{equation}
Considerations on the coercivity of the metric analogous to
\Cref{rmk:coercivity_poisson} apply here; see
\cite[Section~4.5]{zeinhofer_unified_2024} for details.

\paragraph{Results}
As shown in \Cref{fig:heat_h1_time}, Nystr{\"o}mNGD and RPCholNGD again achieve the same accuracy as NGD full at a reduced computational cost. However, \Cref{fig:heat_h1_iter} indicates that the preconditioned NGD methods require slightly more iterations to converge compared to NGD with a direct solver. This is likely because the singular value decay (see \Cref{fig:heat_svals_decay}) is less strong than in the Poisson example, resulting in a less effective preconditioner within the maximum rank $\ell_{\max} = 500$, a limit that is rapidly reached (see \Cref{fig:heat_ell_iter}). Nevertheless, all preconditioned NGD variants perform significantly better than unpreconditioned matrix-free NGD, which yields an accuracy two orders of magnitude worse and is also slower per iteration. This again underscores the importance of preconditioning in matrix-free NGD methods.

\begin{figure}[htbp]
    \centering
    \begin{subfigure}[t]{0.49\textwidth}
        \includegraphics[width=\textwidth]{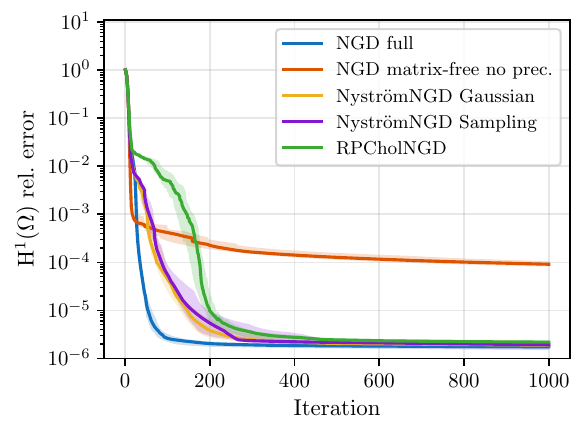}
        \caption{Relative $\rmH^1(\Omega)$ error vs.\ Iterations.}
        \label{fig:heat_h1_iter}
    \end{subfigure}
    \hfill
    \begin{subfigure}[t]{0.49\textwidth}
        \includegraphics[width=\textwidth]{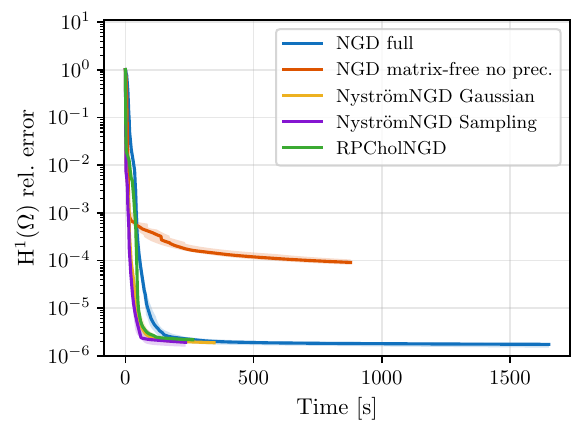}
        \caption{Relative $\rmH^1(\Omega)$ error vs.\ wall-clock time.}
        \label{fig:heat_h1_time}
    \end{subfigure}
    \hfill
    \begin{subfigure}[t]{0.49\textwidth}
        \includegraphics[width=\textwidth]{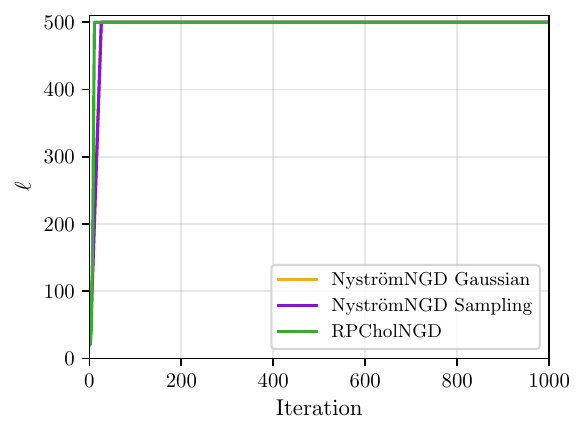}
        \caption{Evolution of the adaptive rank parameter $\ell$ of the
            randomized preconditioners during training; the imposed cap is $\ell_{\max} = 500$.}
        \label{fig:heat_ell_iter}
    \end{subfigure}
    \hfill
    \begin{subfigure}[t]{0.49\textwidth}
        \includegraphics[width=\textwidth]{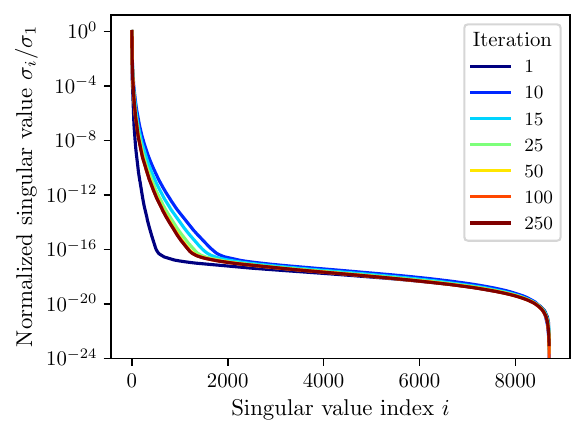}
        \caption{Decay of the normalized singular values
        $\{\sigma_i/\sigma_1\}_{i=1}^p$ of the Gramian during the first 250
        NGD-full iterations; each color corresponds to a different iterate and
        shows how rapidly the spectrum decays.}
        \label{fig:heat_svals_decay}
    \end{subfigure}
    \caption{Heat equation on $I \times \Omega = (0,1) \times (0,1)^3$ discretized via PINNs, with exact solution $u(t,x,y,z)=\big(\cos(\pi x)+\cos(\pi y)+\cos(\pi z)\big)e^{-\pi^2 t/4}$. Panels (a) and (b) compare NGD full, matrix-free NGD without preconditioning, Nystr{\"o}mNGD, and RPCholNGD in terms of relative $\rmH^1(\Omega)$ error versus iterations and versus wall-clock time. Panels (c) and (d) report, respectively, the adaptive rank of the preconditioners and the singular value decay of the Gramian, highlighting the low-rank structure exploited by the preconditioners.}
    \label{fig:main-error-heat}
\end{figure}

\subsubsection{Kovasznay flow (2D steady-state Navier--Stokes flow) -- PINNs}
\label{sec:kovasznay_pinn}
We consider the two-dimensional steady Navier--Stokes flow as originally
described in \cite{kovasznay_laminar_1948}, with Reynolds number $\mathrm{Re} =
    40$, following the setup in \cite[Section 4.1]{jnini_gauss-newton_2024}. The
system of PDEs reads
\begin{equation}
    \label{eq:navierstokes_system}
    \begin{cases}
        - \nu\Delta u  + (u \cdot\nabla)u+\nabla p = f & \Omega = (-\frac{1}{2},1) \times (-\frac{1}{2}, \frac{3}{2}), \\
        \nabla \cdot u = 0                             & \Omega,                                                       \\
        u = g_D                                        & \partial\Omega,
    \end{cases}
\end{equation}
with exact solution
\begin{equation*}
    \begin{split}
        u(x,y) & = \begin{bmatrix}
                       1 - e^{\lambda x} \cos(2\pi y) \\
                       \frac{\lambda}{2\pi}e^{\lambda x} \sin(2\pi y)
                   \end{bmatrix}, \\
        p(x,y) & = \frac{1}{2}(1-e^{2\lambda x}) - c,
    \end{split}
\end{equation*}
where $\lambda = \frac{1}{2\nu} - \sqrt{\frac{1}{4\nu^2}+4\pi^2}$, $\nu =
    \frac{1}{\mathrm{Re}} = \frac{1}{40}$, and $c$ is chosen such that $\int_\Omega
    p = 0$. Following \cite{jnini_gauss-newton_2024}, we consider the function
spaces $\rmV = \rmH^2(\Omega) \times \{p\in\rmH^1(\Omega): \int_\Omega p = 0\}$
and $\rmX = \rmL^2(\Omega)^2 \times \rmL^2(\Omega) \times
    \rmL^2(\partial\Omega)$. Problem~\eqref{eq:navierstokes_system} is written in
strong form as a nonlinear least-squares problem with residual
\begin{equation}
    \label{eq:navier_residual}
    \calR(u,p) = \begin{bmatrix}
        - \nu\Delta u  + (u \cdot\nabla)u+\nabla p - f \\
        \nabla \cdot u                                 \\
        u - g_D
    \end{bmatrix},
\end{equation}
and corresponding energy functional
\begin{equation}
    \label{eq:navier_loss}
    E(u,p)  = \frac{1}{2} \norm{- \nu\Delta u  + (u \cdot\nabla)u+\nabla p - f}^2_{\rmL^2(\Omega)^2}
    + \frac{1}{2} \norm{\nabla \cdot u}^2_{\rmL^2(\Omega)}
    + \frac{1}{2} \norm{u - g_D}^2_{\rmL^2(\partial\Omega)}.
\end{equation}
As a metric, we consider the Gauss--Newton metric
\begin{equation*}
    \bilinear[g_{(u,p)}]{(v, h)}{(w,t)} = \innerh{\rmD\calR(u,p)[(v,h)]}{\rmD\calR(u,p)[(w,t)]}.
\end{equation*}
Since the residual \eqref{eq:navier_residual} is nonlinear, the metric is
point-dependent. We refer to \cite{jnini_gauss-newton_2024} for a discussion on
coercivity.

To discretize \eqref{eq:navierstokes_system}, we parameterize the velocity and
pressure fields using two separate neural networks, $(u, p) =
    (u_{\vtheta_1}, p_{\vtheta_2})$, and concatenate their weights into a single
parameter vector $\vtheta = [\vtheta_1, \vtheta_2]$. To ensure the total parameter
count remains comparable to the other experiments, both networks are configured
with three hidden layers of $45$ neurons each. Furthermore, because we utilize
the Gauss--Newton metric for this problem, we adapt the regularization parameter
$\mu$ using the standard Levenberg--Marquardt heuristic \cite{wright_numerical_2006}.

\paragraph{Results}
In this example, Nystr{\"o}mNGD with Gaussian test matrices converges in fewer iterations compared to the other methods, as shown in \Cref{fig:navier_h1_iter}. However, Nystr{\"o}mNGD with sampling test matrices and RPCholNGD benefit from a lower per-iteration cost, allowing them to achieve comparable overall convergence times (see \Cref{fig:navier_h1_time}). Interestingly, the Nystr{\"o}m rank $\ell$ exhibits a pronounced peak during intermediate iterations, hitting the maximum rank $\ell_{\max} = 500$ before decreasing and stabilizing around $\ell = 400$ (see \Cref{fig:navier_ell_iter}). Similar behavior is typically observed when approximating PDE solutions using low-rank techniques \cite{bioli_preconditioned_2025}. Conversely, RPCholNGD steadily increases the rank to approximately $\ell=220$, well below the maximum $\ell_{\max} = 500$, and achieves similar accuracy with a reduced memory footprint. This demonstrates that the choice of preconditioner can impact both the convergence speed (in iterations and wall-clock time) and the memory requirements of NGD.

\begin{figure}[htbp]
    \centering
    \begin{subfigure}[t]{0.49\textwidth}
        \includegraphics[width=\textwidth]{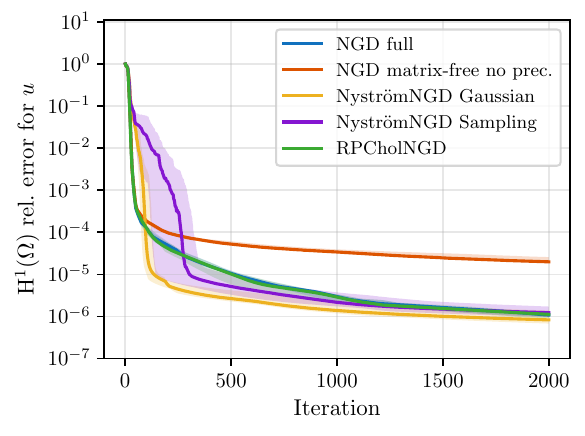}
        \caption{Relative $\rmH^1(\Omega)$ error on $u$ vs.\ Iterations.}
        \label{fig:navier_h1_iter}
    \end{subfigure}
    \hfill
    \begin{subfigure}[t]{0.49\textwidth}
        \includegraphics[width=\textwidth]{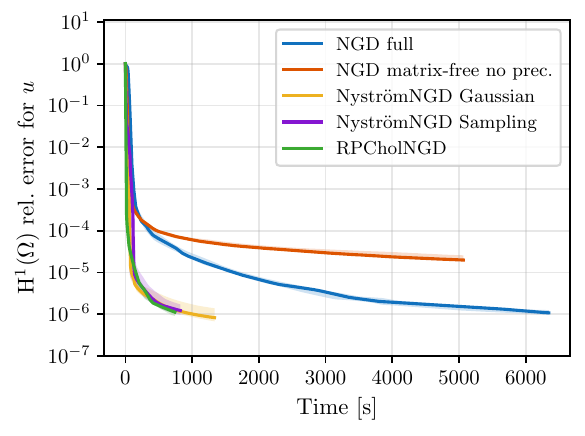}
        \caption{Relative $\rmH^1(\Omega)$ error on $u$ vs.\ wall-clock time.}
        \label{fig:navier_h1_time}
    \end{subfigure}
    \hfill
    \begin{subfigure}[t]{0.49\textwidth}
        \includegraphics[width=\textwidth]{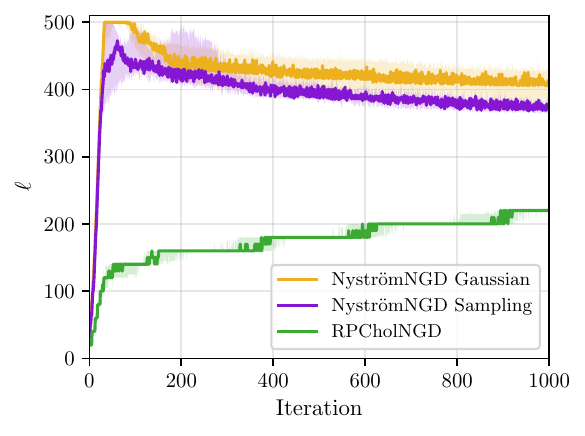}
        \caption{Evolution of the adaptive rank parameter $\ell$ of the
            randomized preconditioners during training; the imposed cap is $\ell_{\max} = 500$.}
        \label{fig:navier_ell_iter}
    \end{subfigure}
    \hfill
    \begin{subfigure}[t]{0.49\textwidth}
        \includegraphics[width=\textwidth]{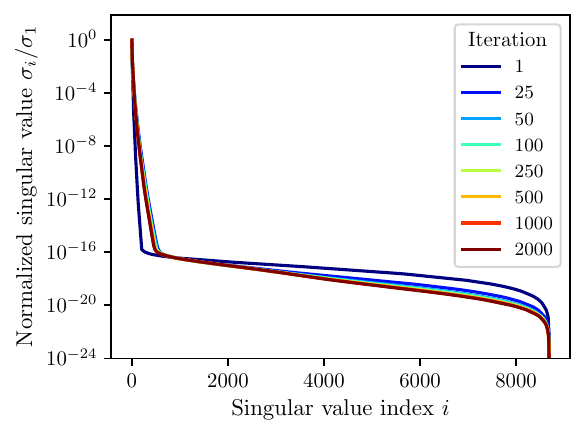}
        \caption{Decay of the normalized singular values
        $\{\sigma_i/\sigma_1\}_{i=1}^p$ of the Gramian during training using NGD with direct solver; each color corresponds to a different iterate and shows how rapidly the spectrum decays.}
        \label{fig:navier_svals_decay}
    \end{subfigure}
    \caption{Kovasznay flow, i.e., the 2D steady Navier--Stokes benchmark at Reynolds number $\mathrm{Re}=40$, discretized via PINNs. Panels (a) and (b) compare NGD full, matrix-free NGD without preconditioning, Nystr{\"o}mNGD, and RPCholNGD in terms of relative $\rmH^1(\Omega)$ error on the velocity field $u$ versus iterations and versus wall-clock time. Panels (c) and (d) report, respectively, the adaptive rank of the preconditioners and the singular value decay of the Gramian, highlighting the low-rank structure exploited by the preconditioners.}
    \label{fig:main-error-navier}
\end{figure}

\subsubsection{Poisson problem in 2D -- FEINNs Energy minimization (Deep Ritz-like) formulation}
\label{sec:poisson_deepritz}
We consider the Poisson problem on the 2D unit square
\begin{equation}
    \label{eq:poisson_deepritz_pde}
    \begin{cases}
        -\Delta u = f & \Omega = (0, 1)^2, \\
        u = g_D       & \partial\Omega,
    \end{cases}
\end{equation}
where $f \in \rmH^{-1}(\Omega)$ and $g_D \in
    \rmH^{\frac{1}{2}}(\partial\Omega)$. As exact solution we take $u(x,y) =
    \sin(\pi x)\sin(\pi y)$.

All equalities in \eqref{eq:poisson_deepritz_pde} are understood in the
variational sense. That is, given a lifting $\tilde{g}_D \in \rmH^1(\Omega)$
such that $\restr{\tilde{g}_D}{\partial\Omega} = g_D$, we seek $u_0 \in
    \rmH^1_0(\Omega)$ satisfying
\begin{equation}
    a(u_0, v) = \int_{\Omega} fv\, \mathrm{d}\vx - a(\tilde{g}_D, v) \quad \forall v \in \rmH^1_0(\Omega), \qquad \text{where} \quad a(u, v) = \int_{\Omega} \nabla u \cdot \nabla v \, \mathrm{d}\vx.
\end{equation}
The solution is then $u = u_0 + \tilde{g}_D$. The discretization spaces are $\rmV = \rmH^1_0(\Omega)$, $\rmX = \dual{\rmV} = \rmH^{-1}(\Omega)$, with linear operator
\[
    \calF: \rmH^1_0(\Omega) \to \rmH^{-1}(\Omega), \qquad
    \dualpair{\calF u}{v}_{\rmH^{-1} \times \rmH^1_0} = a(u,v).
\]
Since $a$ is symmetric and coercive on $\rmV \times \rmV$, we adopt an energy
minimization formulation, akin to the Deep Ritz method:
\begin{equation}
    \label{eq:energy_poissondeepritz}
    \min_{u_0 \in \rmH^1_0(\Omega)} E(u_0) \defas
    \int_{\Omega} \frac{1}{2} \norm{\nabla u_0}^2 - f u_0 + \nabla \tilde{g}_D \cdot \nabla u_0 \, \mathrm{d}\vx,
\end{equation}
with corresponding metric $g = a$.

Since the energy \eqref{eq:energy_poissondeepritz} may take negative values and does not vanish at the minimum, we cannot employ a regularization term of the form $\mu = \max(\gamma \cdot \hat{\lambda}_1 \cdot \epsilon_{\mathrm{mach}}, 10^{-4} \cdot L(\vtheta)^2)$ to enforce strong regularization during early training iterations. Instead, at each iteration $k$, we utilize $\mu = \max(\gamma \cdot \hat{\lambda}_1 \cdot \epsilon_{\mathrm{mach}}, 2^{-k})$.

\paragraph{Results}
Nystr{\"o}mNGD and RPCholNGD exhibit almost identical convergence in iterations compared to NGD full, reaching the same accuracy at a fraction of the computational cost, as shown in \Cref{fig:deepritz_h1_time}. The performance gap between the preconditioned NGD variants and NGD full is more pronounced in this example than in previous ones because a rank of $\ell\approx150$ suffices for effective preconditioning (see \Cref{fig:deepritz_ell_iter}). This is consistent with the significantly sharper singular value decay observed in \Cref{fig:deepritz_svals_decay} compared to earlier examples. Conversely, unpreconditioned matrix-free NGD remains uncompetitive, further underscoring the critical importance of preconditioning.

\begin{figure}[htbp]
    \centering
    \begin{subfigure}[t]{0.49\textwidth}
        \includegraphics[width=\textwidth]{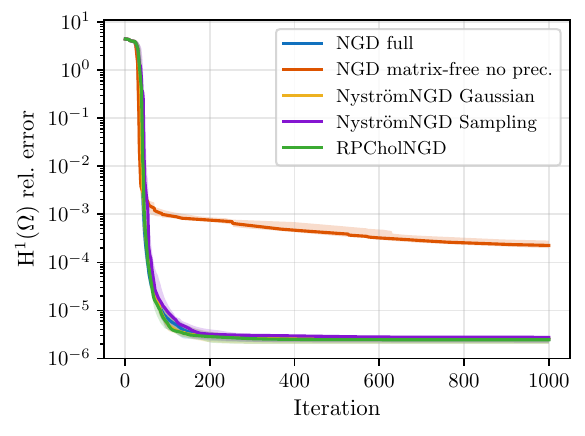}
        \caption{Relative $\rmH^1(\Omega)$ error vs.\ Iterations.}
        \label{fig:deepritz_h1_iter}
    \end{subfigure}
    \hfill
    \begin{subfigure}[t]{0.49\textwidth}
        \includegraphics[width=\textwidth]{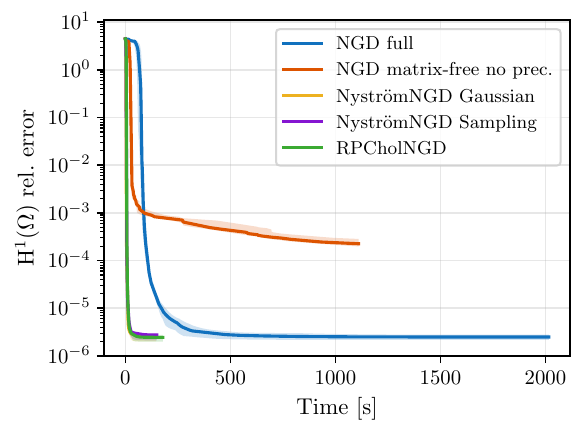}
        \caption{Relative $\rmH^1(\Omega)$ error vs.\ wall-clock time.}
        \label{fig:deepritz_h1_time}
    \end{subfigure}
    \hfill
    \begin{subfigure}[t]{0.49\textwidth}
        \includegraphics[width=\textwidth]{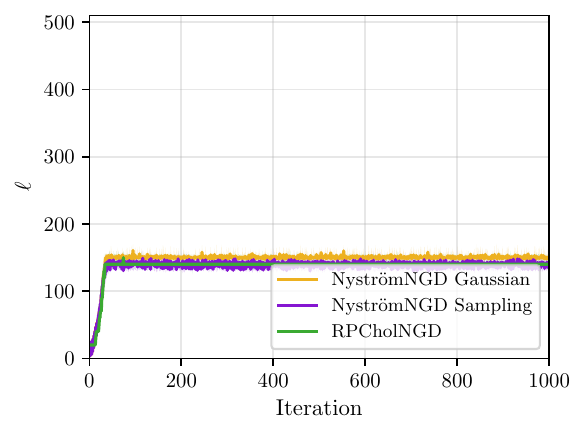}
        \caption{Evolution of the adaptive rank parameter $\ell$ of the
            randomized preconditioners during training; the imposed cap is $\ell_{\max} = 500$.}
        \label{fig:deepritz_ell_iter}
    \end{subfigure}
    \hfill
    \begin{subfigure}[t]{0.49\textwidth}
        \includegraphics[width=\textwidth]{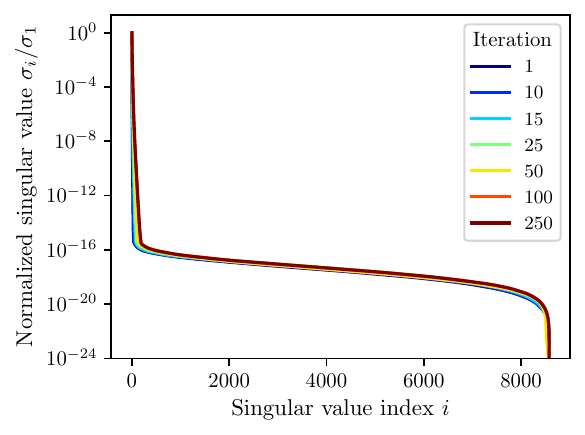}
        \caption{Decay of the normalized singular values
        $\{\sigma_i/\sigma_1\}_{i=1}^p$ of the Gramian during the first 250
        NGD-full iterations; each color corresponds to a different iterate and
        shows how rapidly the spectrum decays.}
        \label{fig:deepritz_svals_decay}
    \end{subfigure}
    \caption{Poisson equation on the unit square discretized via FEINNs in an energy-minimization (Deep Ritz-like) formulation, with exact solution $u(x,y)=\sin(\pi x)\sin(\pi y)$. Panels (a) and (b) compare NGD full, matrix-free NGD without preconditioning, Nystr{\"o}mNGD, and RPCholNGD in terms of relative $\rmH^1(\Omega)$ error versus iterations and versus wall-clock time. Panels (c) and (d) report, respectively, the adaptive rank of the preconditioners and the singular value decay of the Gramian, highlighting the low-rank structure exploited by the preconditioners.}
    \label{fig:error_poissondeepritz}
\end{figure}

\subsubsection{Diffusion-reaction-transport problem in 2D -- FEINNs}
\label{sec:canuto_feinn}
We consider the diffusion-reaction-transport problem
\begin{equation}
    \begin{dcases}
        -\nabla \cdot (\kappa \nabla u) + \vbeta \cdot \nabla u + \sigma u = F & \Omega = (0, 1)^2,               \\
        u = g                                                                  & \Gamma_D \subset\partial\Omega,  \\
        \kappa \nabla u \cdot \vn = \psi                                       & \Gamma_N \subset \partial\Omega,
    \end{dcases}
    \label{eq:canuto_pde}
\end{equation}
where $\kappa(x,y) = 2 + \sin(x + 2y)$, $\vbeta(x, y) = \begin{bmatrix} \sqrt{x
        - y^2 + 5} & \sqrt{y - x^2 + 5}\end{bmatrix}^\top$, and $\sigma(x,y) =
    e^{\frac{x}{2}-\frac{y}{3}} + 2$. The source term $F$ is chosen so that the
exact solution is
\[
    u(x,y) = \sin(3.2x(x-y))\cos(4.3y+x) + \sin(4.6(x+2y))\cos(2.6(y-2x)),
\]
as in \cite{berrone_variational_2022,badia_finite_2024}. The Dirichlet boundary
$\Gamma_D$ consists of the left and right sides of the domain, while the Neumann
boundary $\Gamma_N$ corresponds to the top and bottom edges.

The function spaces are $\rmV = \rmH^1_{0,\Gamma_D}(\Omega)$, $\rmX =
    \dual{\rmV} = \rmH^{-1}_{\Gamma_D}(\Omega)$, and the weak formulation of
\eqref{eq:canuto_pde} is defined as $\calF u = f$ where
\begin{equation*}
    \begin{split}
        \dualpair{\calF u}{v}_{\rmH^{-1}_{\Gamma_D} \times \rmH^1_{0,\Gamma_D}} & = a(u,v) = \int_{\Omega} (\kappa\nabla u) \cdot \nabla v + (\vbeta \cdot \nabla u)\, v + \sigma\,u \,v\,\,\mathrm{d}\vx \\
        \dualpair{f}{v}_{\rmH^{-1}_{\Gamma_D} \times \rmH^1_{0,\Gamma_D}}       & = \int_{\Omega} F\,v\,\,\mathrm{d}\vx + \int_{\Gamma_N} \psi \,v\,\,\mathrm{d}\vs.
    \end{split}
\end{equation*}
Since $\calF$ is in general not symmetric positive definite, we adopt the
least-squares formulation
\begin{equation}
    \label{eq:loss_canuto}
    \min_{u \in \rmH^1_0(\Omega)} E(u) \defas \frac{1}{2} \norm{\calF u - f}_{\rmH^{-1}_{\Gamma_D}}^2 = \frac{1}{2} \dualpair{\calF u - f}{\calR_{\rmV}^{-1} (\calF u - f)}_{\rmH^{-1}_{\Gamma_D} \times \rmH^1_{0,\Gamma_D}}.
\end{equation}
Here, $\calR_{\rmV}^{-1} = (-\Delta)^{-1} : \rmH^{-1}_{\Gamma_D} \to
    \rmH^1_{0,\Gamma_D}$ denotes the inverse Riesz operator, which involves
solving a Poisson equation. The least-squares metric would be $\innerh{v}{w}
    = \dualpair{\calF w}{\calR_{\rmV}^{-1}\calF v}$. To reduce the computational
cost, we consider the $\rmH^1_{0,\Gamma_D}$ norm $\bilinear[g]{u}{v} = \int_{\Omega}
    \nabla u \cdot \nabla v\,\mathrm{d}\vx$. This choice is motivated by the fact
that the $\rmH^1_{0,\Gamma_D}$ norm is spectrally equivalent to the least-squares metric
due to the inf-sup condition, and it allows us to avoid the costly
application of $\calR_{\rmV}^{-1}$ for each matvec with the Gramian.

When discretizing \eqref{eq:loss_canuto} with FEINNs, applying the discrete
Riesz operator $\calR_{\rmV_h}^{-1}$ corresponds to solving a Poisson problem
with FEM. To reduce computational cost, one could instead use a spectrally
equivalent approximation, such as a few cycles of a geometric multigrid
preconditioner, as done in \cite{badia_finite_2024}. However, for simplicity, we
solve the Poisson problem directly using FEM, computing a Cholesky factorization
of the Laplacian matrix once.

\paragraph{Results}
\Cref{fig:canuto_h1_iter} shows that Nystr{\"o}mNGD and RPCholNGD exhibit per-iteration convergence behavior almost identical to NGD with a direct solver, although Nystr{\"o}mNGD requires a slightly higher iteration count. In contrast, unpreconditioned matrix-free NGD plateaus at an accuracy two orders of magnitude worse. Furthermore, the preconditioned matrix-free NGD methods yield significant time savings (\Cref{fig:canuto_h1_time}). This efficiency stems from constructing the preconditioner with a rank of $\ell \approx 200$, which is approximately $40\times$ smaller than the total parameter count $p=8577$.

\begin{figure}[htbp]
    \centering
    \begin{subfigure}[t]{0.49\textwidth}
        \includegraphics[width=\textwidth]{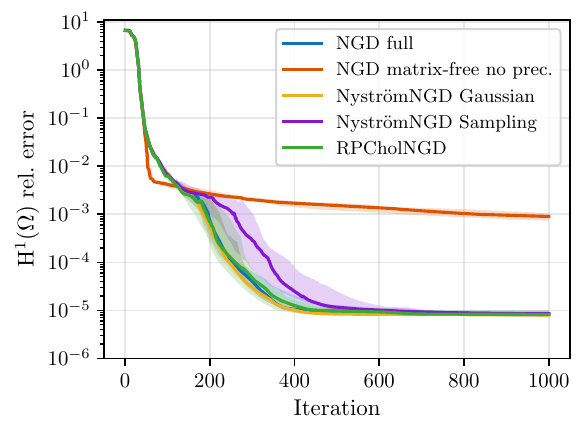}
        \caption{Relative $\rmH^1(\Omega)$ error vs.\ Iterations.}
        \label{fig:canuto_h1_iter}
    \end{subfigure}
    \hfill
    \begin{subfigure}[t]{0.49\textwidth}
        \includegraphics[width=\textwidth]{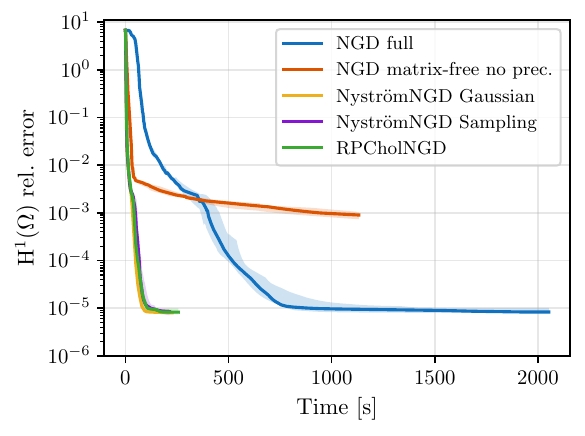}
        \caption{Relative $\rmH^1(\Omega)$ error vs.\ wall-clock time.}
        \label{fig:canuto_h1_time}
    \end{subfigure}
    \hfill
    \begin{subfigure}[t]{0.49\textwidth}
        \includegraphics[width=\textwidth]{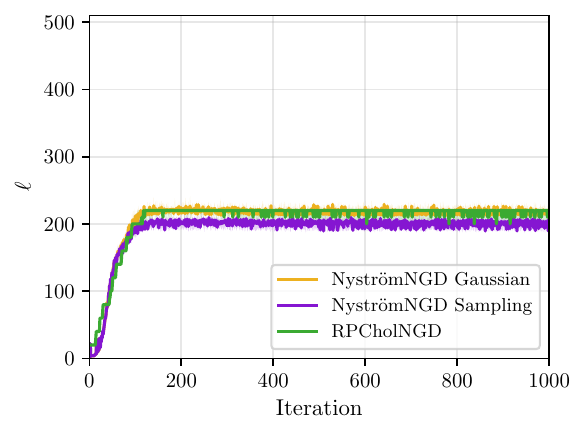}
        \caption{Evolution of the adaptive rank parameter $\ell$ of the
            randomized preconditioners during training; the imposed cap is $\ell_{\max} = 500$.}
        \label{fig:canuto_ell_iter}
    \end{subfigure}
    \hfill
    \begin{subfigure}[t]{0.49\textwidth}
        \includegraphics[width=\textwidth]{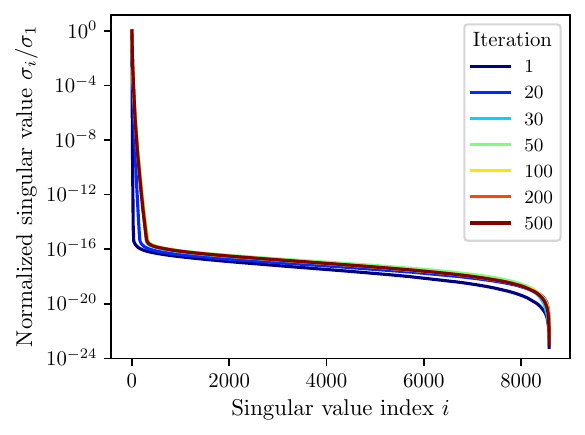}
        \caption{Decay of the normalized singular values
        $\{\sigma_i/\sigma_1\}_{i=1}^p$ of the Gramian during the first 500
        NGD-full iterations; each color corresponds to a different iterate and
        shows how rapidly the spectrum decays.}
        \label{fig:canuto_svals_decay}
    \end{subfigure}
    \caption{Diffusion-reaction-transport equation \eqref{eq:canuto_pde} on the unit square, with mixed Dirichlet--Neumann boundary conditions, discretized via FEINNs. Panels (a) and (b) compare NGD full, matrix-free NGD without preconditioning, Nystr{\"o}mNGD, and RPCholNGD in terms of relative $\rmH^1(\Omega)$ error versus iterations and versus wall-clock time. Panels (c) and (d) report, respectively, the adaptive rank of the preconditioners and the singular value decay of the Gramian, highlighting the low-rank structure exploited by the preconditioners.}
    \label{fig:error_canuto}
\end{figure}

\subsection{Comparison with other optimizers}
\label{sec:comparison_others}
In this section, we compare the optimizers from the NGD framework, and with more
care Nystr{\"o}mNGD and RPCholNGD, against a suite of standard and
state-of-the-art optimizers (Adam, L-BFGS, BFGS, and SSBroyden) on all five test
problems. Since each optimizer is run for a deliberately large number of
iterations, often well beyond the point of practical convergence, we quantify
performance at the iteration where each optimizer effectively plateaus.
Specifically, we define a \emph{plateau} as the first iteration after which the
relative error fails to improve by more than $10\%$ within a
trailing window covering $10\%$ of the total number of iterations.
\Cref{fig:comparison_NGD_others} shows the full convergence history for all
optimizers on each problem; a dot~($\bullet$) on each curve marks the detected
plateau iterate, confirming that the criterion reliably captures the end of
productive convergence. \Cref{tab:comparison_NGD_others} reports, for each
optimizer and test problem, the median error at the plateau. The following
line reports, in smaller font and in parentheses, the $25$th and $75$th
percentiles in compact form as $(a-b\,\times 10^k)$, while the last line reports
the corresponding cumulative wall-clock time in seconds.

As reported in \Cref{tab:comparison_NGD_others}, the final accuracy achieved by
NGD with a direct solver (NGD full) and the matrix-free preconditioned NGD
variants (Nystr{\"o}mNGD and RPCholNGD) is highly consistent across all test
problems. This shows that the primary objective of the paper is achieved in all
examples: improving the computational efficiency of existing NGD approaches
while preserving essentially the same final accuracy. On most problems, this
accuracy is also comparable to that of SSBroyden, which remains a very strong
baseline and, in some cases, attains a slightly better final error. The runner-up among
the baseline methods is BFGS; however, its final accuracy on the PINN examples
remains more than an order of magnitude worse. Crucially, both SSBroyden and
BFGS entail memory requirements of $\mathcal{O}(p^2)$, whereas the
preconditioned NGD methods operate with a memory footprint of
$\mathcal{O}(p\ell_{\max})$. When comparing our approaches to L-BFGS, which
possesses similar memory requirements given history size $\ell_{\max}$, the
final accuracy of Nystr{\"o}mNGD and RPCholNGD remains clearly better on the
considered benchmarks. Adam, being a first-order method, consistently yields the
poorest accuracy across all benchmarks.

In terms of computational efficiency, RPCholNGD and Nystr{\"o}mNGD are often
faster than SSBroyden while using less memory, although this behavior is not
uniform across all evaluated problems. Among the proposed variants,
Nystr{\"o}mNGD often offers the best trade-off between final accuracy and
computational time, with Gaussian test matrices serving as a robust default
choice. The main exception is the Kovasznay problem, where SSBroyden is faster
and reaches a lower error than the NGD-based methods. We deliberately retained
this example for fairness, as it shows that NGD methods are not uniformly
superior to SSBroyden on all possible benchmarks. At the same time, comparing
the preconditioned variants with NGD full on this problem indicates that the
accuracy bottleneck lies in the underlying NGD formulation rather than in the
randomized preconditioner, since the preconditioned methods recover essentially
the same accuracy in significantly less time. Overall, these comparisons should
be read primarily as evidence that the proposed methods improve
baseline NGD, while remaining competitive with strong state-of-the-art
optimizers on many problems; see \Cref{fig:comparison_NGD_others}.

\begin{table}[htbp]
    \centering
    \caption{Relative $\rmH^1(\Omega)$ error at the plateau iterate for each optimizer and test problem. The first line reports the median error, the second line reports in smaller font and in parentheses the $25$th and $75$th percentiles in compact form as $(a-b\,\times 10^k)$, and the third line reports the corresponding cumulative wall-clock time in seconds.
        The plateau criterion and the dot~($\bullet$) marking the corresponding
        iterate in the convergence curves are described in
        \Cref{sec:comparison_others}; see \Cref{fig:comparison_NGD_others}.}
    \label{tab:comparison_NGD_others}
    \footnotesize
    \resizebox{\textwidth}{!}{%
        \begin{tblr}{
                colspec = {l|ccccc},
                row{even} = {gray!15}, % Colors even rows
                row{1} = {white},      % Keeps the header background white
            }
            \toprule
            \textbf{Optimizer}
             & \textbf{Poisson 3D}
             & \textbf{Heat 3+1D}
             & \textbf{Kovasznay }
             & \textbf{Deep-Ritz Poisson}
             & \textbf{FEINNs 2D}         \\
            \midrule
            \textbf{NGD full}
             & {$1.56\times10^{-6}$       \\ {\scriptsize $(1.37-1.77\,\times10^{-6})$} \\ 317.7 s}
             & {$1.75\times10^{-6}$       \\ {\scriptsize $(1.49-1.86\,\times10^{-6})$} \\ 279.2 s}
             & {$1.08\times10^{-6}$       \\ {\scriptsize $(0.965-1.25\,\times10^{-6})$} \\ 5708.3 s}
             & {$1.63\times10^{-5}$       \\ {\scriptsize $(1.62-1.65\,\times10^{-5})$} \\ 411.3 s}
             & {$6.61\times10^{-5}$       \\ {\scriptsize $(6.60-6.61\,\times10^{-5})$} \\ 852.1 s} \\
            {\textbf{NGD matrix-free}     \\ \textbf{(no precond.)}}
             & {$3.44\times10^{-4}$       \\ {\scriptsize $(3.36-3.72\,\times10^{-4})$} \\ 409.2 s}
             & {$9.13\times10^{-5}$       \\ {\scriptsize $(8.19-10.5\,\times10^{-5})$} \\ 363.8 s}
             & {$1.99\times10^{-5}$       \\ {\scriptsize $(1.90-2.57\,\times10^{-5})$} \\ 3726.2 s}
             & {$2.34\times10^{-4}$       \\ {\scriptsize $(1.82-3.62\,\times10^{-4})$} \\ 923.8 s}
             & {$8.42\times10^{-4}$       \\ {\scriptsize $(6.84-10.9\,\times10^{-4})$} \\ 1016.0 s} \\
            \textbf{RPCholNGD}
             & {$1.69\times10^{-6}$       \\ {\scriptsize $(1.58-1.82\,\times10^{-6})$} \\ 74.1 s}
             & {$2.18\times10^{-6}$       \\ {\scriptsize $(1.94-2.42\,\times10^{-6})$} \\ 111.9 s}
             & {$1.12\times10^{-6}$       \\ {\scriptsize $(0.909-1.33\,\times10^{-6})$} \\ 654.4 s}
             & {$1.63\times10^{-5}$       \\ {\scriptsize $(1.62-1.64\,\times10^{-5})$} \\ 31.9 s}
             & {$6.61\times10^{-5}$       \\ {\scriptsize $(6.60-6.61\,\times10^{-5})$} \\ 101.8 s} \\
            {\textbf{Nystr\"omNGD}        \\ \textbf{Gaussian}}
             & {$1.66\times10^{-6}$       \\ {\scriptsize $(1.47-1.89\,\times10^{-6})$} \\ 66.8 s}
             & {$1.89\times10^{-6}$       \\ {\scriptsize $(1.79-1.93\,\times10^{-6})$} \\ 106.1 s}
             & {$8.24\times10^{-7}$       \\ {\scriptsize $(6.74-13.8\,\times10^{-7})$} \\ 878.5 s}
             & {$1.63\times10^{-5}$       \\ {\scriptsize $(1.62-1.65\,\times10^{-5})$} \\ 23.7 s}
             & {$6.60\times10^{-5}$       \\ {\scriptsize $(6.60-6.61\,\times10^{-5})$} \\ 86.4 s} \\
            {\textbf{Nystr\"omNGD}        \\ \textbf{Sampling}}
             & {$1.97\times10^{-6}$       \\ {\scriptsize $(1.74-2.15\,\times10^{-6})$} \\ 81.2 s}
             & {$1.93\times10^{-6}$       \\ {\scriptsize $(1.56-2.11\,\times10^{-6})$} \\ 59.0 s}
             & {$1.22\times10^{-6}$       \\ {\scriptsize $(0.965-1.71\,\times10^{-6})$} \\ 548.0 s}
             & {$1.63\times10^{-5}$       \\ {\scriptsize $(1.62-1.64\,\times10^{-5})$} \\ 29.0 s}
             & {$6.63\times10^{-5}$       \\ {\scriptsize $(6.60-6.66\,\times10^{-5})$} \\ 101.6 s} \\
            \textbf{BFGS}
             & {$8.91\times10^{-5}$       \\ {\scriptsize $(8.58-10.3\,\times10^{-5})$} \\ 540.4 s}
             & {$1.03\times10^{-4}$       \\ {\scriptsize $(0.783-1.26\,\times10^{-4})$} \\ 684.0 s}
             & {$2.92\times10^{-5}$       \\ {\scriptsize $(2.63-3.16\,\times10^{-5})$} \\ 418.6 s}
             & {$3.64\times10^{-5}$       \\ {\scriptsize $(3.03-4.35\,\times10^{-5})$} \\ 141.8 s}
             & {$7.30\times10^{-5}$       \\ {\scriptsize $(7.06-7.35\,\times10^{-5})$} \\ 171.0 s} \\
            \textbf{SSBroyden}
             & {$2.69\times10^{-6}$       \\ {\scriptsize $(2.60-2.78\,\times10^{-6})$} \\ 482.6 s}
             & {$2.39\times10^{-6}$       \\ {\scriptsize $(2.18-2.56\,\times10^{-6})$} \\ 477.4 s}
             & {$5.45\times10^{-7}$       \\ {\scriptsize $(4.91-6.36\,\times10^{-7})$} \\ 329.3 s}
             & {$1.60\times10^{-5}$       \\ {\scriptsize $(1.60-1.60\,\times10^{-5})$} \\ 80.1 s}
             & {$6.61\times10^{-5}$       \\ {\scriptsize $(6.61-6.61\,\times10^{-5})$} \\ 139.8 s} \\
            \textbf{L-BFGS}
             & {$7.91\times10^{-4}$       \\ {\scriptsize $(6.83-11.1\,\times10^{-4})$} \\ 1104.9 s}
             & {$3.91\times10^{-4}$       \\ {\scriptsize $(3.59-4.36\,\times10^{-4})$} \\ 867.5 s}
             & {$1.04\times10^{-4}$       \\ {\scriptsize $(0.980-1.30\,\times10^{-4})$} \\ 774.2 s}
             & {$1.68\times10^{-3}$       \\ {\scriptsize $(1.25-2.22\,\times10^{-3})$} \\ 396.8 s}
             & {$4.94\times10^{-3}$       \\ {\scriptsize $(3.76-5.97\,\times10^{-3})$} \\ 607.8 s} \\
            \textbf{Adam}
             & {$1.95\times10^{-2}$       \\ {\scriptsize $(1.92-2.02\,\times10^{-2})$} \\ 412.7 s}
             & {$1.87\times10^{-3}$       \\ {\scriptsize $(1.79-2.06\,\times10^{-3})$} \\ 1077.9 s}
             & {$2.47\times10^{-3}$       \\ {\scriptsize $(2.32-2.94\,\times10^{-3})$} \\ 609.8 s}
             & {$2.37\times10^{-2}$       \\ {\scriptsize $(2.11-2.59\,\times10^{-2})$} \\ 82.0 s}
             & {$2.54\times10^{-2}$       \\ {\scriptsize $(2.26-2.94\,\times10^{-2})$} \\ 563.7 s} \\
            \bottomrule
        \end{tblr}%
    }
\end{table}

\begin{figure}[htbp]
    % --- Poisson 3D (PINNs) ---
    \begin{subfigure}[t]{\textwidth}
        \begin{minipage}[t]{0.49\textwidth}
            \includegraphics[width=\textwidth]{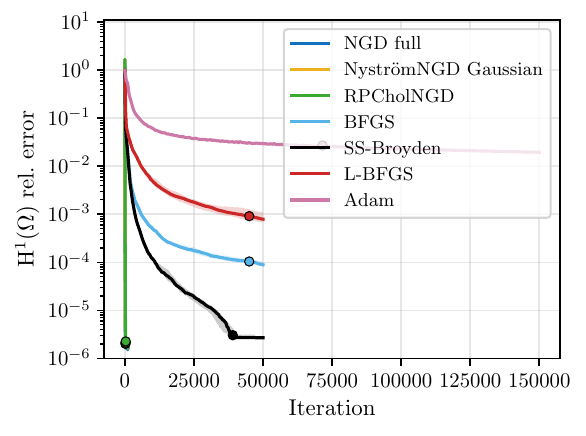}
        \end{minipage}
        \hfill
        \begin{minipage}[t]{0.49\textwidth}
            \includegraphics[width=\textwidth]{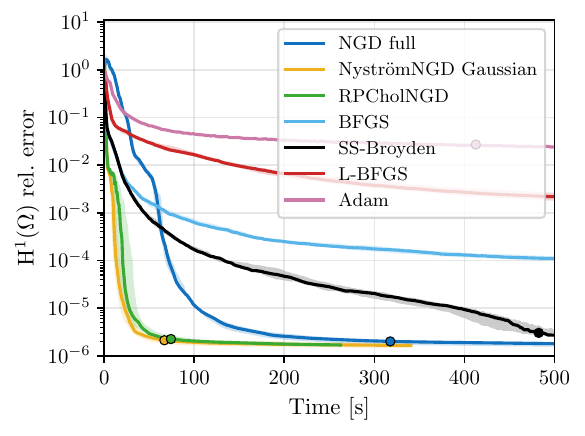}
        \end{minipage}
        \caption{Poisson problem in 3D discretized via PINNs; see \Cref{sec:poisson3d_pinn}.}
    \end{subfigure}
    \medskip
    % --- Heat equation 3+1D (PINNs) ---
    \begin{subfigure}[t]{\textwidth}
        \begin{minipage}[t]{0.49\textwidth}
            \includegraphics[width=\textwidth]{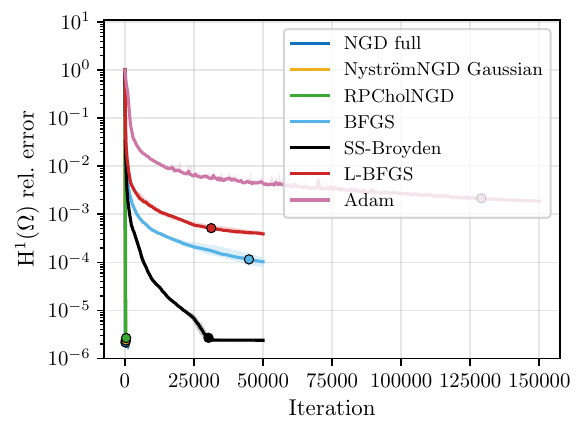}
        \end{minipage}
        \hfill
        \begin{minipage}[t]{0.49\textwidth}
            \includegraphics[width=\textwidth]{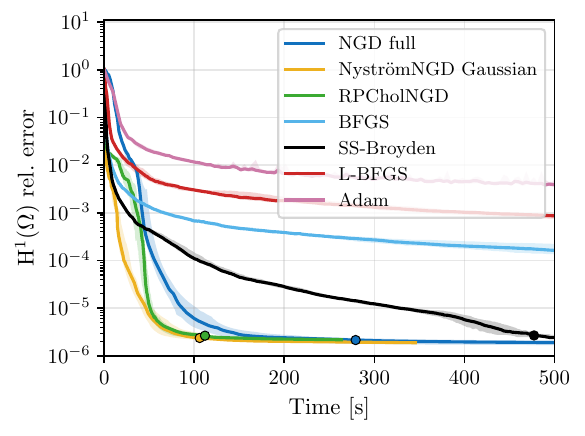}
        \end{minipage}
        \caption{Heat equation in 3+1D discretized via PINNs; see \Cref{sec:heat3d1_pinn}.}
    \end{subfigure}
    \medskip
    % --- Kovasznay flow / Navier-Stokes 2D (PINNs) ---
    \begin{subfigure}[t]{\textwidth}
        \begin{minipage}[t]{0.49\textwidth}
            \includegraphics[width=\textwidth]{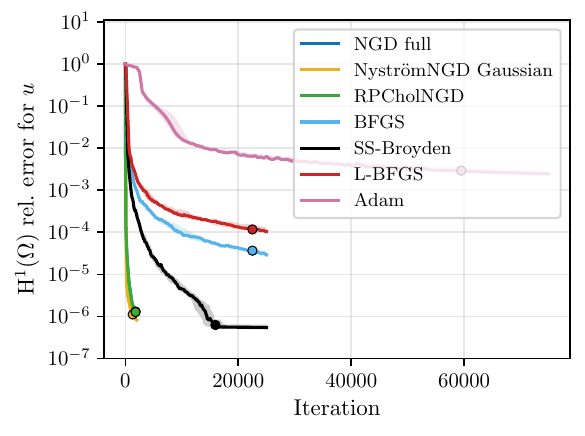}
        \end{minipage}
        \hfill
        \begin{minipage}[t]{0.49\textwidth}
            \includegraphics[width=\textwidth]{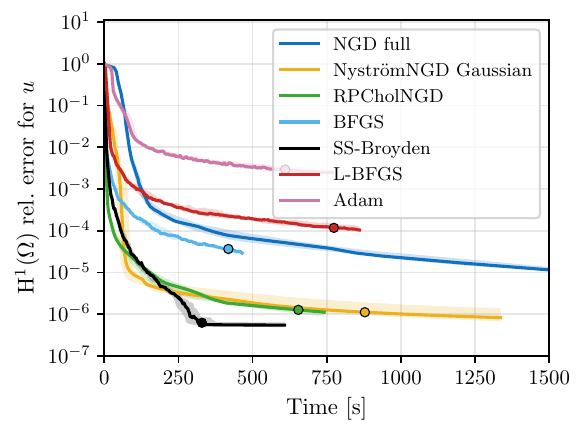}
        \end{minipage}
        \caption{Kovasznay Flow (2D steady-state Navier--Stokes flow) discretized via PINNs; see \Cref{sec:kovasznay_pinn}. Relative $\rmH^1(\Omega)$ error on~$u$.}
    \end{subfigure}
\end{figure}
\begin{figure}[htbp]
    \ContinuedFloat
    % --- Poisson 2D (FEINN, energy minimization) ---
    \begin{subfigure}[t]{\textwidth}
        \begin{minipage}[t]{0.49\textwidth}
            \includegraphics[width=\textwidth]{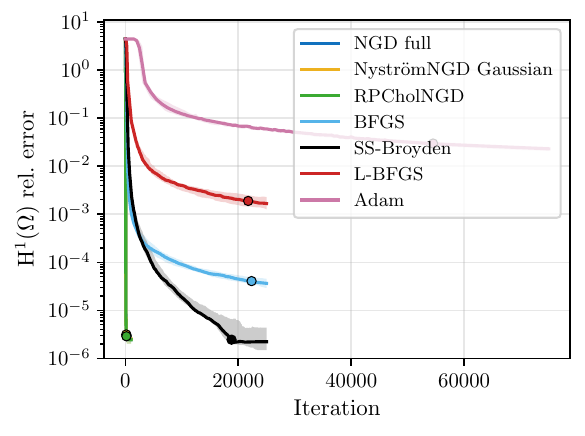}
        \end{minipage}
        \hfill
        \begin{minipage}[t]{0.49\textwidth}
            \includegraphics[width=\textwidth]{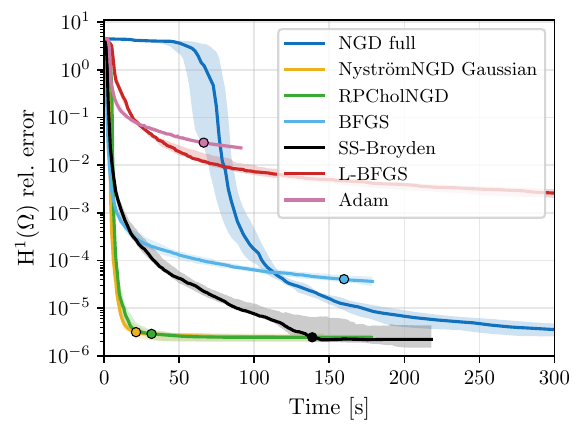}
        \end{minipage}
        \caption{Poisson problem in 2D discretized via FEINNs Energy minimization (Deep Ritz-like) formulation; see \Cref{sec:poisson_deepritz}.}
    \end{subfigure}
    \medskip
    % --- Diffusion-reaction-transport 2D (FEINN-H1) ---
    \begin{subfigure}[t]{\textwidth}
        \begin{minipage}[t]{0.49\textwidth}
            \includegraphics[width=\textwidth]{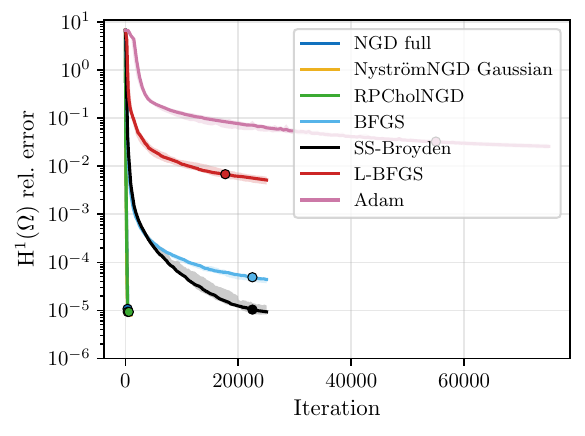}
        \end{minipage}
        \hfill
        \begin{minipage}[t]{0.49\textwidth}
            \includegraphics[width=\textwidth]{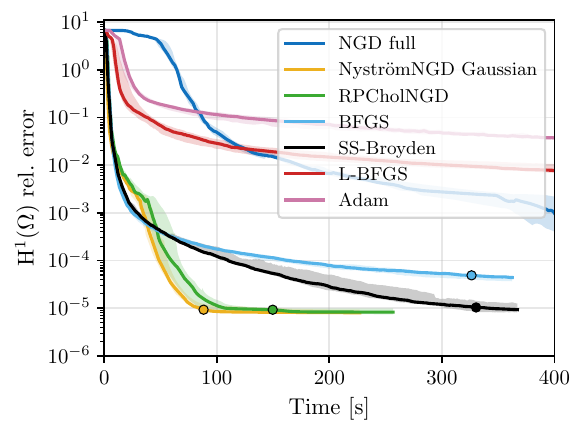}
        \end{minipage}
        \caption{Diffusion-reaction-transport problem in 2D discretized via FEINNs; see \Cref{sec:canuto_feinn}.}
    \end{subfigure}
    \caption{Comparison of Adam, L-BFGS, BFGS, SSBroyden, NGD full, matrix-free NGD without preconditioning, Nystr{\"o}mNGD (Gaussian and sampling), and RPCholNGD across all five test problems. In each row, the left panel reports relative $\rmH^1(\Omega)$ error versus iterations and the right panel reports the same error versus cumulative wall-clock time, restricted to the early part of the run to distinguish the fastest methods. The dot~($\bullet$) marks the plateau iterate used to record the wall-clock time reported in \Cref{tab:comparison_NGD_others}.}
    \label{fig:comparison_NGD_others}
\end{figure}

\subsection{Scalability experiments -- Poisson problem in 3D}
\label{sec:scalability}
We consider the 3D Poisson problem \eqref{eq:poisson3d}, discretized via PINNs
as in \Cref{sec:poisson3d_pinn}, and investigate the scalability of the
preconditioned NGD methods with respect to the number of parameters $p$ and the
number of quadrature points $q$. First, we keep $p$ fixed and increase the total
number of quadrature points to $q=10^5$, such that the full Jacobian
$\mJ(\vtheta)$ exceeds memory capacity, requiring batching over the quadrature
points (i.e., the ``Recompute batch Jacobian on the fly'' strategy from
\Cref{tab:matvec_costs}). Second, we also increase the number of parameters to
$p\approx 3.3\times 10^4$ by using a wider network with three hidden layers of
$128$ neurons each, rendering the $\calO(p^2)$ memory requirement of full NGD
and SSBroyden prohibitive under our hardware constraints. In both experiments,
we employ a batch size of $B=1000$ quadrature points. The results are reported
in \Cref{tab:scalability}.

When increasing the number of quadrature points from $q=10^4$ to $q=10^5$,
Nystr{\"o}mNGD (Gaussian) and Nystr{\"o}mNGD (Sampling) exhibit a per-iteration
cost that grows almost linearly in $q$, with only a slight overhead due to
batching, which requires recomputing the (batched) Jacobian at each CG
iteration. In contrast, the per-iteration cost grows significantly more for
RPCholNGD. As discussed in \Cref{sec:compare_nystrom_rpcholesky}, constructing
the Nystr{\"o}m preconditioner performs all matvecs in a single batched
operation, whereas RPCholesky accesses the column blocks of the Gramian
sequentially, thereby incurring the $\calO(q\fp)$ overhead of recomputing the
Jacobian for each block.

Furthermore, our preconditioned matrix-free NGD methods can also be run on the
larger $3\times128\times128\times128\times1$ architecture, for which NGD full
and SSBroyden become prohibitive due to their $\calO(p^2)$ memory footprint. The
wider network also yields a more accurate final approximation. Finally, comparing
the two architectures at $q=10^5$, the per-iteration cost scales approximately
linearly with the number of parameters, as expected.

\begin{remark}
    \label{rmk:peak_memory}
    Empirical peak GPU memory figures depend heavily on the software stack, for
    instance on memory pre-allocation, Just-In-Time compilation, and
    compiler-level buffer reuse, rather than on the intrinsic requirements of
    the algorithms. We therefore quantify memory consumption only through the
    theoretical footprints derived in
    \Cref{sec:matvecs_strategies,sec:putting_together} and summarized in
    \Cref{tab:matvec_costs}, while acknowledging the absence of empirical peak
    GPU memory measurements as a limitation of the present study.
\end{remark}

\begin{table}[htbp]
    \centering
    \caption{Scalability experiments on the 3D Poisson problem discretized via
        PINNs; see \Cref{sec:scalability}. Relative $\rmH^1(\Omega)$ error at the
        plateau iterate for each optimizer and architecture. The first line reports
        the median error, the second line reports in smaller font and in
        parentheses the $25$th and $75$th percentiles in compact form as
        $(a-b\,\times 10^k)$, and the third line reports the corresponding
        cumulative wall-clock time in seconds, while the fourth line reports the
        per-iteration cost in seconds. Each column corresponds to a
        different network architecture and number of quadrature points $q$, as
        indicated in the header.}
    \label{tab:scalability}
    \footnotesize
    \begin{tblr}{
            colspec = {l|ccc},
            row{even} = {gray!15}, % Colors even rows
            row{1} = {white},      % Keeps the header background white
        }
        \toprule
        \textbf{Optimizer}
         & {$\mathbf{3\times64\times64\times64\times1}$    \\ $\mathbf{q=10^4}$}
         & {$\mathbf{3\times64\times64\times64\times1}$    \\ $\mathbf{q=10^5}$}
         & {$\mathbf{3\times128\times128\times128\times1}$ \\ $\mathbf{q=10^5}$} \\
        \midrule
        \textbf{NGD full}
         & {$1.56\times10^{-6}$                            \\ {\scriptsize $(1.37-1.77\,\times10^{-6})$} \\ 317.7 s \\ 1.634 s/it}
         & {$1.18\times10^{-6}$                            \\ {\scriptsize $(1.11-1.29\,\times10^{-6})$} \\ 2643.8 s \\ 13.48 s/it}
         & {--}                                            \\
        \textbf{RPCholNGD}
         & {$1.69\times10^{-6}$                            \\ {\scriptsize $(1.58-1.82\,\times10^{-6})$} \\ 74.1 s \\ 0.2624 s/it}
         & {$1.13\times10^{-6}$                            \\ {\scriptsize $(1.06-1.22\,\times10^{-6})$} \\ 2080.6 s \\ 10.21 s/it}
         & {$8.39\times10^{-7}$                            \\ {\scriptsize $(8.31-8.48\,\times10^{-7})$} \\ 6762.0 s \\ 33.2 s/it}                     \\
        {\textbf{Nystr\"omNGD}                             \\ \textbf{Gaussian}}
         & {$1.66\times10^{-6}$                            \\ {\scriptsize $(1.47-1.89\,\times10^{-6})$} \\ 66.8 s \\ 0.3404 s/it}
         & {$1.05\times10^{-6}$                            \\ {\scriptsize $(0.980-1.16\,\times10^{-6})$} \\ 783.3 s \\ 4.792 s/it}
         & {$7.63\times10^{-7}$                            \\ {\scriptsize $(7.55-7.71\,\times10^{-7})$} \\ 2772.9 s \\ 16.98 s/it}                    \\
        {\textbf{Nystr\"omNGD}                             \\ \textbf{Sampling}}
         & {$1.97\times10^{-6}$                            \\ {\scriptsize $(1.74-2.15\,\times10^{-6})$} \\ 81.2 s \\ 0.2275 s/it}
         & {$1.54\times10^{-6}$                            \\ {\scriptsize $(1.45-1.68\,\times10^{-6})$} \\ 1481.3 s \\ 4.909 s/it}
         & {$1.07\times10^{-6}$                            \\ {\scriptsize $(1.06-1.08\,\times10^{-6})$} \\ 5169.7 s \\ 17.13 s/it}                    \\
        \textbf{SSBroyden}
         & {$2.69\times10^{-6}$                            \\ {\scriptsize $(2.60-2.78\,\times10^{-6})$} \\ 482.6 s \\ 0.01341 s/it}
         & {$2.16\times10^{-6}$                            \\ {\scriptsize $(2.11-2.24\,\times10^{-6})$} \\ 5461.5 s \\ 0.1689 s/it}
         & {--}                                            \\
        \bottomrule
    \end{tblr}
\end{table}

\section{Discussion and conclusion}
\label{sec:discussion}
We presented a modular framework for efficient natural gradient descent (NGD)
and proposed the use of techniques from Randomized Numerical Linear Algebra
(RandNLA), specifically the Nystr{\"o}m approximation and the Randomly Pivoted
Partial Cholesky (RPCholesky) factorization, to construct a preconditioner for
the inner conjugate gradient (CG) solver in matrix-free NGD. This led to
\emph{Nystr{\"o}mNGD} and \emph{RPCholNGD}, two novel variants of NGD for
neural-network-based PDE solvers. Moreover, our work extends matrix-free NGD to
a broader class of metrics relevant to PDE-related problems. Our approach
significantly reduces the cost of the inner solver, resulting in substantial
wall-clock savings compared to NGD with a direct solver and unpreconditioned
matrix-free NGD.

We validated Nystr{\"o}mNGD and RPCholNGD on various PDE problems, both in
strong form discretized with PINNs and in weak form discretized with FEINNs. In
all considered examples, the primary objective of the paper is achieved: the
preconditioned variants improve existing NGD approaches by matching or improving
the accuracy of NGD full and unpreconditioned matrix-free NGD at a lower
computational cost. We also compared them with strong baselines such as
SSBroyden to better position the methods within the broader optimization
literature. On most of the analyzed problems, our methods attain an accuracy comparable to
SSBroyden with lower memory requirements and often shorter wall-clock times;
however, this behavior is not uniform, and the Kovasznay example shows that
NGD-based methods are not universally superior to all state-of-the-art
optimizers. This does not affect the main conclusion of the paper, since our
goal is to improve the efficiency and scalability of NGD itself rather than to
outperform every competing optimizer across all PDEs.

The use of Randomized NLA low-rank approximations for preconditioning builds
on the empirically observed low-rank structure of the Gramian. Our results
confirm that the Gramian often admits strong low-rank compression, and low-rank
preconditioners remain highly effective even when they do not fully capture the
Gramian. Furthermore, the core idea of preconditioning the inner CG solver
remains effective even when the Gramian does not admit a low-rank approximation,
and we plan to investigate alternative Gramian approximations for
preconditioning purposes. Moreover, Nystr{\"o}mNGD and RPCholNGD are
particularly promising for neural-network-based solvers targeting weak
formulations of PDEs, where the underlying variational structure can naturally
guide the choice of metric. Our results in
\Cref{sec:poisson_deepritz,sec:canuto_feinn} strongly support this potential and
motivate further research in this direction. Finally, it would be valuable to
extend these methods to operator learning, which requires accommodating the
structure of operator spaces rather than just function spaces, and to broader
large-scale machine learning tasks where NGD has already shown promise.

\section*{CRediT authorship contribution statement}
% \textbf{Name Surname:} Conceptualization, Methodology, Software, Validation,
% Formal analysis, Investigation, Resources, Data Curation, Writing -- Original
% Draft, Writing -- Review \& Editing, Visualization, Supervision, Project
% administration, Funding acquisition

\indent \textbf{Ivan Bioli:} Writing -- Review \& Editing, Writing -- Original
Draft, Visualization, Validation, Software, Methodology, Investigation, Formal
analysis, Data Curation, Conceptualization

\textbf{Carlo Marcati:} Writing -- Review \& Editing, Writing -- Original Draft,
Validation, Supervision, Methodology, Formal analysis, Conceptualization

\textbf{Giancarlo Sangalli:} Writing -- Review \& Editing, Writing -- Original
Draft, Validation, Supervision, Resources, Project administration, Methodology,
Funding acquisition, Formal analysis, Conceptualization

\section*{Declaration of competing interest}
The authors declare that they have no known competing financial interests or personal relationships that could have appeared to influence the work reported in this paper.

\section*{Data availability}
No data was used for the research described in the article.

\section*{Acknowledgments}
The authors thank the anonymous reviewers for their constructive comments, which
led to a more complete assessment of strategies for computing the Gramian and
matrix--vector products (\Cref{sec:assembly_and_matvecs}), to the inclusion of
additional comparisons with state-of-the-art optimizers and to other
improvements in the manuscript.

This work has received funding from the European Union’s Horizon Europe research
and innovation programme under the Marie Skłodowska-Curie grant agreement
No.~101119556 (IN-DEEP). The authors are members of the Gruppo Nazionale
Calcolo Scientifico - Istituto Nazionale di Alta Matematica (GNCS-INdAM). C.
Marcati and G. Sangalli acknowledge support from the Italian MUR through the
PRIN 2022 PNRR project NOTES (No. P2022NC97R).  G. Sangalli further acknowledges
support from the Italian MUR through  PNRR-M4C2-I1.4-NC-HPC-Spoke6. C.~Marcati
further acknowledges the PRIN 2022 project ASTICE (No.~202292JW3F) and
the French Agence Nationale de la Recherche and Ministère de
l’Enseignement Supérieur et de la Recherche.

{\small
  \bibliographystyle{plain}
  \bibliography{bibliography}}

\appendix
\section{Algorithmic details}
\subsection{Pseudocode for Randomized Nystr{\"o}m Approximation}
\label{apx:sec:randomized_nystrom}
We report in \Cref{alg:randomized_nystrom} a numerically stable algorithm for the Randomized Nystr{\"o}m
Approximation that also yields the eigendecomposition of
$\hat{\mG}_{\mathrm{nys}}$, as described in \cite[Algorithm
  16]{martinsson_randomized_2020}. An alternative implementation relying solely on
Cholesky decompositions, rather than Singular Value Decompositions (SVDs), was
proposed in \cite{guzman-cordero_improving_2025}. While Cholesky decompositions
are more amenable to GPU acceleration than SVDs, this approach is less
numerically stable. In practice, we observed that it leads to catastrophic
failure if the rank parameter $\ell$ is overestimated; consequently, we chose
not to adopt this implementation.
\begin{algorithm}[htbp]
  \begin{algorithmic}[1] % The number determines the frequency of the line numbering
    \Require{Symmetric positive semidefinite matrix $\mG\in\R^{p\times p}$ (as a function $\vv\mapsto\mG\vv$), rank $\ell$}
    \Ensure{Nystr{\"o}m approximation $\hat{\mG}_{\mathrm{nys}} = \mU \hat{\mLambda} \mU^\top$}

    \vspace{0.5pc}

    \State $\mOmega = \textrm{randn}(p, \ell)$
    \Comment{Gaussian test matrix}
    \State $\mOmega = \textrm{qr}(\mOmega,0)$
    \Comment{Thin QR decomposition}
    \State $\mY = \mG\mOmega$
    \Comment{$\ell$ matvecs with $\mG$}
    \State $\nu = \textrm{eps}(\textrm{norm}(\mY,\textrm{'fro'}))$
    \Comment{Compute shift}
    \State $\mY_{\nu} = \mY+\nu\mOmega$
    \Comment{Shift for stability}
    \State $\mC = \textrm{chol}(\mOmega^\top\mY_{\nu})$
    \State $\mB = \mY_{\nu}/\mC$ \Comment{Triangular solve}
    \State $[\mU,\mSigma,\sim] = \textrm{svd}(\mB,0)$
    \Comment{Thin SVD}
    \State $\hat{\mLambda} = \max\{0,\mSigma^{2}-\nu \mI\}$ \Comment{Remove shift}
  \end{algorithmic}
  \caption{Randomized Nystr{\"o}m Approximation \cite{frangella_randomized_2023,martinsson_randomized_2020}}
  \label{alg:randomized_nystrom}
\end{algorithm}

\subsection{Pseudocode for (block) RPCholesky}
\label{apx:sec:rpchol}
We present in \Cref{alg:rpcholesky,alg:block_rpcholesky} the pseudocode for both the standard and block variants of the
RPCholesky algorithm, as detailed in \cite{epperly_embrace_2025}.

\begin{algorithm}[htbp]
  \begin{algorithmic}[1]
    \Require{Symmetric positive semidefinite matrix $\mG\in\R^{p\times p}$
      (column access $s \mapsto \mG[:,s]$), diagonal $\vd^{(0)} = \diag(\mG)$, maximum approximation rank $\ell$, absolute
      error tolerance $\varepsilon$}
    \Ensure{$\mF\in\R^{p\times r}$ s.t.\ $\mG \approx \mF\mF^\top$ and $r < \ell$, pivot set $S$}

    \vspace{0.5pc}

    \State Initialize $\mF \leftarrow \mathbf{0}_{p\times\ell}$, $S \leftarrow \emptyset$, $\vd \leftarrow \vd^{(0)}$
    \For{$i = 1, \dots, \ell$}
    \State Sample $s \sim \vd\,/\!\textstyle\sum_j d_j$ \Comment{Sample proportionally to diagonal}
    \State $S \leftarrow S \cup \{s\}$
    \State $\vf \leftarrow \mG[:,s] - \mF[:,1:i{-}1]\,\mF[s,1:i{-}1]^\top$ \Comment{Residual pivot column}
    \State $\mF[:,i] \leftarrow \vf\,/\!\sqrt{\vf[s]}$
    \State $\vd \leftarrow \vd - \left(\mF[:,i]\right)^{\odot 2}$
    \Comment{Update diagonal}
    \If{$\sum_j d_j < \varepsilon$}
    \State \textbf{break}
    \EndIf
    \EndFor
    \State Remove zero columns from $\mF$
  \end{algorithmic}
  \caption{RPCholesky \cite[Algorithm D.1]{epperly_embrace_2025}}
  \label{alg:rpcholesky}
\end{algorithm}

\begin{algorithm}[htbp]
  \begin{algorithmic}[1]
    \Require{Positive semidefinite matrix $\mG\in\R^{p\times p}$ (column
      access), diagonal $\vd^{(0)} = \diag(\mG)$, block size $b > 1$, number of rounds $t$, absolute error tolerance $\varepsilon$}
    \Ensure{$\mF\in\R^{p\times r}$ s.t.\ $\mG \approx \mF\mF^\top$ and $r < tb$, pivot set $S$}

    \vspace{0.5pc}

    \State Initialize $\mF \leftarrow \mathbf{0}_{p\times bt}$, $S \leftarrow \emptyset$, $\vd \leftarrow \vd^{(0)}$
    \For{$i = 0, \dots, t{-}1$}
    \State Sample $s_{ib+1},\dots,s_{ib+b} \overset{\mathrm{i.i.d.}}{\sim} \vd\,/\!\textstyle\sum_j d_j$
    \State $S' \leftarrow \textsc{Unique}\!\left(\{s_{ib+1},\dots,s_{ib+b}\}\right)$
    \State $S \leftarrow S \cup S'$
    \State $\mC \leftarrow \mG[:,S'] - \mF[:,1:ib]\,\mF[S',1:ib]^\top$ \Comment{Residual block}
    \State $\mR \leftarrow \mathrm{chol}(\mC[S',:])$
    \State $\mF[:,ib{+}1:ib{+}|S'|] \leftarrow \mC\,\mR^{-1}$
    \State $\vd \leftarrow \vd - \textsc{SquaredRowNorms}\!\left(\mF[:,ib{+}1:ib{+}b]\right)$ \Comment{Update diagonal}
    \If{$\sum_j d_j < \varepsilon$}
    \State \textbf{break}
    \EndIf
    \EndFor
    \State Remove zero columns from $\mF$
  \end{algorithmic}
  \caption{Block RPCholesky \cite[Algorithm D.2]{epperly_embrace_2025}}
  \label{alg:block_rpcholesky}
\end{algorithm}

\section{Extended experimental details}
\label{apx:sec:numerical_details}

\paragraph{Implementation details}
All methods are implemented in JAX \cite{bradbury_jax_2018}. The standard optimization routines are sourced from Optimistix \cite{rader_optimistix_2024} and Optax \cite{deepmind2020jax}. For the implementation of RPCholesky, we adapted strategies from the Matfree library \cite{kramer_pnkraemermatfree_2025}. The code is publicly available at \url{https://github.com/IvanBioli/preconditioned-natural-gradient.git}~\cite{bioli_randngd_2026}; the quasi-Newton baselines are available in the package \texttt{ssbroyden-optimistix}~\cite{bioli_self-scaled_2026,bioli_ssbroyden_2026}.

\paragraph{Baseline optimizer configurations}
For the quasi-Newton baselines (L-BFGS, BFGS, and SSBroyden), we utilize a Zoom line search algorithm \cite{wright_numerical_2006} to ensure sufficient decrease and satisfy the Wolfe conditions. For Adam, we apply a cosine-decay learning rate schedule with a linear warmup phase of $5000$ iterations: the learning rate initializes at $10^{-4}$, peaks at $10^{-3}$, and decays to a final value of $10^{-5}$.

\paragraph{Problem-specific architectures and iteration budgets}
The total number of optimization iterations and the specific neural network architectures vary depending on the complexity of the PDE problem. A detailed breakdown is provided in \Cref{tab:apx_iterations_architectures}. For the 2D Navier--Stokes problem, the velocity field ($\mathbf{u}$) and the pressure field ($p$) are parametrized by two separate neural networks.

\begin{table}[htbp]
  \centering
  \caption{Iteration budgets for the evaluated optimizers and neural network architectures for each PDE experiment. Network shapes are denoted as \texttt{[input\_dim, hidden\_1, ..., hidden\_N, output\_dim]}.}
  \label{tab:apx_iterations_architectures}
  \resizebox{\textwidth}{!}{
    \begin{tblr}{
        colspec = {lccccc},
        row{even} = {gray!15}, % Alternates row colors
        row{1} = {white},      % Keeps the header white
      }
      \toprule
      \textbf{Experiment} & \textbf{NGD (all)} & \textbf{(L-)BFGS} & \textbf{SSBroyden} & \textbf{Adam} & \textbf{Network Architecture}     \\
      \midrule
      Poisson 3D          & 1\,000             & 50\,000           & 50\,000            & 150\,000      & \texttt{[3, 64, 64, 64, 1]}       \\
      Heat 3+1D           & 1\,000             & 50\,000           & 50\,000            & 150\,000      & \texttt{[4, 64, 64, 64, 1]}       \\
      Kovasznay           & 2\,000             & 25\,000           & 25\,000            & 75\,000       & {$u$: \texttt{[2, 45, 45, 45, 2]} \\ $p$: \texttt{[2, 45, 45, 45, 1]}} \\
      Deep-Ritz Poisson   & 1\,000             & 25\,000           & 25\,000            & 75\,000       & \texttt{[2, 64, 64, 64, 1]}       \\
      FEINNs 2D           & 1\,000             & 25\,000           & 25\,000            & 75\,000       & \texttt{[2, 64, 64, 64, 1]}       \\
      \bottomrule
    \end{tblr}
  }
\end{table}

\paragraph{Discretization and sampling strategies}
The spatial and temporal domain discretizations depend heavily on whether the
problem is formulated strongly, via PINNs, or weakly, via FEINNs. For PINNs, we
employ Monte Carlo quadrature: integration points are sampled uniformly at random
once and fixed throughout training, and errors are always computed using an
order of magnitude more quadrature points than during training, sampled
independently. For all experiments employing variational formulations (FEINNs),
the neural networks are interpolated onto a finite element (FE) space of degree
$d = 4$ over a uniform quadrilateral mesh with a mesh size of $h = 1/30$. This
configuration yields approximately $14{,}500$ degrees of freedom. Integrals
defining the loss and the metric are evaluated by interpolation onto the FE
space followed by Gauss--Jacobi quadrature of order $2d = 8$. For the final
evaluation of errors, we employ a Gauss--Jacobi quadrature of order $4d = 16$. A
summary of the discretization setups is provided in
\Cref{tab:apx_discretization}.

\begin{table}[htbp]
  \centering
  \caption{Domain discretization, boundary conditions, and evaluation setups for each numerical experiment.}
  \label{tab:apx_discretization}
  \resizebox{\textwidth}{!}{
    \begin{tblr}{
        colspec = {llll},
        row{even} = {gray!15}, % Alternates row colors
        row{1} = {white},      % Keeps the header white
      }
      \toprule
      \textbf{Experiment} & \textbf{Interior/Mesh configuration} & \textbf{Boundary Points} & \textbf{Evaluation set}  \\
      \midrule
      Poisson 3D          & 10\,000 random                         & 1\,000 random            & 100\,000 random          \\
      Heat 3+1D           & 10\,000 random                         & 1\,000 random            & 100\,000 random          \\
      Kovasznay           & 10\,000 random                         & 1\,000 random            & 100\,000 random          \\
      \midrule
      Deep-Ritz Poisson   & Quad mesh, $h=1/30$, $d=4$ (14.5k DoF) & -                        & Gauss--Jacobi, order $16$ \\
      FEINNs 2D           & Quad mesh, $h=1/30$, $d=4$ (14.5k DoF) & -                        & Gauss--Jacobi, order $16$ \\
      \bottomrule
    \end{tblr}
  }
\end{table}

\section{Sensitivity analysis}
\label{apx:sec:sensitivity}
We conduct a sensitivity analysis of Nystr{\"o}mNGD and RPCholNGD with respect to
their key hyperparameters. Specifically, we examine the impact of the
regularization parameter $\mu$ and the rank adaptivity strategy on the
convergence behavior and computational efficiency. For simplicity, we focus on
Nystr{\"o}mNGD with Gaussian test matrices, as we found the behavior to be
similar when using sampling test matrices.

\paragraph{Regularization parameter $\mu$}
In all experiments except the Kovasznay flow, which employs the
Levenberg--Marquardt adaptation strategy for $\mu$, we set $\mu = \gamma \cdot
  \epsilon_{\textrm{mach}} \cdot \hat{\lambda}_1$ to regularize the least squares
problem, solving with $(\mG(\vtheta) + \mu \mI)^{-1}$ instead of
$\mG(\vtheta)^\dagger$. The default value is $\gamma = 10$. Setting $\gamma$
much larger leads to excessive regularization and a search direction that
deviates from the natural gradient, leading to a loss of accuracy. Setting it
much smaller introduces numerical instability in both methods, though with
different sensitivity. For RPCholNGD, divergence is already visible at $\gamma =
  1$ on the strong-form problems (Poisson 3D, Heat 3+1D). Nystr{\"o}mNGD is more
robust: $\gamma = 1$ still yields good accuracy, but $\gamma = 0.1$ causes
divergence on the variational problems and also increases the Nystr{\"o}m rank
parameter $\ell$, raising per-iteration computational cost. Results are reported
in \Cref{apx:tab:mu_sensitivity_nystrom,apx:tab:mu_sensitivity_rpchol}.

\begin{table}[htbp]
  \centering
  \caption{Sensitivity analysis of $\gamma$ for Nystr{\"o}mNGD with Gaussian test
    matrices. The first line reports the median relative $\rmH^1(\Omega)$ error,
    the second line reports in smaller font and in parentheses the $25$th and
    $75$th percentiles in compact form as $(a-b\,\times 10^k)$, and the third line
    reports the wall-clock time relative to the default $\gamma = 10$.}
  \label{apx:tab:mu_sensitivity_nystrom}
  \begin{tblr}{
      colspec = {l|cccc},
      row{even} = {gray!15},
      row{1} = {white},
    }
    \toprule
    $\boldsymbol{\gamma}$
     & \textbf{Poisson 3D}
     & \textbf{Heat 3+1D}
     & \textbf{Deep-Ritz Poisson}
     & \textbf{FEINNs 2D}         \\
    \midrule
    $\mathbf{0.1}$
     & {$1.87\times10^{-6}$       \\ {\scriptsize $(0.886-3.27\,\times10^{-6})$} \\ $1.004\times$}
     & {$7.42\times10^{-6}$       \\ {\scriptsize $(6.05-8.60\,\times10^{-6})$} \\ $1.023\times$}
     & {$7.49\times10^{-1}$       \\ {\scriptsize $(3.99-10.4\,\times10^{-1})$} \\ $1.267\times$}
     & {$8.69\times10^{-3}$       \\ {\scriptsize $(8.63-10.1\,\times10^{-3})$} \\ $1.258\times$} \\
    $\mathbf{1}$
     & {$5.57\times10^{-7}$       \\ {\scriptsize $(5.51-5.86\,\times10^{-7})$} \\ $1.000\times$}
     & {$2.17\times10^{-6}$       \\ {\scriptsize $(1.80-2.35\,\times10^{-6})$} \\ $1.004\times$}
     & {$1.61\times10^{-5}$       \\ {\scriptsize $(1.61-1.65\,\times10^{-5})$} \\ $1.096\times$}
     & {$1.36\times10^{-4}$       \\ {\scriptsize $(1.27-39.9\,\times10^{-4})$} \\ $1.111\times$} \\
    $\mathbf{10}$ \textbf{(def.)}
     & {$1.66\times10^{-6}$       \\ {\scriptsize $(1.47-1.89\,\times10^{-6})$} \\ $1.00\times$}
     & {$1.89\times10^{-6}$       \\ {\scriptsize $(1.79-1.93\,\times10^{-6})$} \\ $1.00\times$}
     & {$1.63\times10^{-5}$       \\ {\scriptsize $(1.62-1.65\,\times10^{-5})$} \\ $1.00\times$}
     & {$6.60\times10^{-5}$       \\ {\scriptsize $(6.60-6.61\,\times10^{-5})$} \\ $1.00\times$} \\
    $\mathbf{100}$
     & {$3.72\times10^{-6}$       \\ {\scriptsize $(3.54-3.98\,\times10^{-6})$} \\ $0.963\times$}
     & {$3.80\times10^{-6}$       \\ {\scriptsize $(3.41-4.13\,\times10^{-6})$} \\ $0.997\times$}
     & {$1.70\times10^{-5}$       \\ {\scriptsize $(1.68-1.71\,\times10^{-5})$} \\ $0.961\times$}
     & {$6.62\times10^{-5}$       \\ {\scriptsize $(6.62-6.63\,\times10^{-5})$} \\ $0.988\times$} \\
    $\mathbf{1000}$
     & {$8.37\times10^{-6}$       \\ {\scriptsize $(8.20-9.14\,\times10^{-6})$} \\ $0.871\times$}
     & {$6.52\times10^{-6}$       \\ {\scriptsize $(5.26-7.65\,\times10^{-6})$} \\ $0.959\times$}
     & {$2.65\times10^{-5}$       \\ {\scriptsize $(2.31-2.66\,\times10^{-5})$} \\ $0.972\times$}
     & {$6.78\times10^{-5}$       \\ {\scriptsize $(6.78-6.80\,\times10^{-5})$} \\ $0.911\times$} \\
    $\mathbf{10000}$
     & {$1.35\times10^{-5}$       \\ {\scriptsize $(1.07-1.75\,\times10^{-5})$} \\ $0.720\times$}
     & {$7.01\times10^{-6}$       \\ {\scriptsize $(6.23-8.25\,\times10^{-6})$} \\ $0.759\times$}
     & {$3.10\times10^{-5}$       \\ {\scriptsize $(2.98-3.48\,\times10^{-5})$} \\ $0.851\times$}
     & {$7.30\times10^{-5}$       \\ {\scriptsize $(7.28-7.67\,\times10^{-5})$} \\ $0.784\times$} \\
    \bottomrule
  \end{tblr}
\end{table}

\begin{table}[htbp]
  \centering
  \caption{Sensitivity analysis of $\gamma$ for RPCholNGD. The first line
    reports the median relative $\rmH^1(\Omega)$ error, the second line reports
    in smaller font and in parentheses the $25$th and $75$th percentiles in compact
    form as $(a-b\,\times 10^k)$, and the third line reports the wall-clock
    time relative to the default $\gamma = 10$.}
  \label{apx:tab:mu_sensitivity_rpchol}
  \begin{tblr}{
      colspec = {l|cccc},
      row{even} = {gray!15},
      row{1} = {white},
    }
    \toprule
    $\boldsymbol{\gamma}$
     & \textbf{Poisson 3D}
     & \textbf{Heat 3+1D}
     & \textbf{Deep-Ritz Poisson}
     & \textbf{FEINNs 2D}         \\
    \midrule
    $\mathbf{0.1}$
     & {$8.42\times10^{-3}$       \\ {\scriptsize $(7.79-8.49\,\times10^{-3})$} \\ $1.062\times$}
     & {$7.48\times10^{-3}$       \\ {\scriptsize $(5.24-9.63\,\times10^{-3})$} \\ $1.055\times$}
     & {$2.12$                    \\ {\scriptsize $(1.08-2.53)$} \\ $1.224\times$}
     & {$1.10\times10^{-2}$       \\ {\scriptsize $(1.03-1.12\,\times10^{-2})$} \\ $1.226\times$} \\
    $\mathbf{1}$
     & {$7.67\times10^{-3}$       \\ {\scriptsize $(7.00-8.15\,\times10^{-3})$} \\ $1.066\times$}
     & {$5.96\times10^{-4}$       \\ {\scriptsize $(0.577-18.7\,\times10^{-4})$} \\ $1.027\times$}
     & {$2.17\times10^{-5}$       \\ {\scriptsize $(1.89-646\,\times10^{-5})$} \\ $1.105\times$}
     & {$2.64\times10^{-4}$       \\ {\scriptsize $(1.74-60.4\,\times10^{-4})$} \\ $1.101\times$} \\
    $\mathbf{10}$ \textbf{(def.)}
     & {$1.69\times10^{-6}$       \\ {\scriptsize $(1.58-1.82\,\times10^{-6})$} \\ $1.00\times$}
     & {$2.18\times10^{-6}$       \\ {\scriptsize $(1.94-2.42\,\times10^{-6})$} \\ $1.00\times$}
     & {$1.63\times10^{-5}$       \\ {\scriptsize $(1.62-1.64\,\times10^{-5})$} \\ $1.00\times$}
     & {$6.61\times10^{-5}$       \\ {\scriptsize $(6.60-6.61\,\times10^{-5})$} \\ $1.00\times$} \\
    $\mathbf{100}$
     & {$4.16\times10^{-6}$       \\ {\scriptsize $(3.86-4.42\,\times10^{-6})$} \\ $0.967\times$}
     & {$3.36\times10^{-6}$       \\ {\scriptsize $(3.06-3.78\,\times10^{-6})$} \\ $0.979\times$}
     & {$1.69\times10^{-5}$       \\ {\scriptsize $(1.68-1.71\,\times10^{-5})$} \\ $1.010\times$}
     & {$6.63\times10^{-5}$       \\ {\scriptsize $(6.62-6.63\,\times10^{-5})$} \\ $0.933\times$} \\
    $\mathbf{1000}$
     & {$9.50\times10^{-6}$       \\ {\scriptsize $(9.12-10.3\,\times10^{-6})$} \\ $0.846\times$}
     & {$5.42\times10^{-6}$       \\ {\scriptsize $(4.84-6.39\,\times10^{-6})$} \\ $0.949\times$}
     & {$2.67\times10^{-5}$       \\ {\scriptsize $(2.36-2.70\,\times10^{-5})$} \\ $0.947\times$}
     & {$6.84\times10^{-5}$       \\ {\scriptsize $(6.81-6.86\,\times10^{-5})$} \\ $0.887\times$} \\
    $\mathbf{10000}$
     & {$1.23\times10^{-5}$       \\ {\scriptsize $(1.16-1.51\,\times10^{-5})$} \\ $0.720\times$}
     & {$8.74\times10^{-6}$       \\ {\scriptsize $(7.29-10.1\,\times10^{-6})$} \\ $0.780\times$}
     & {$3.37\times10^{-5}$       \\ {\scriptsize $(3.10-3.64\,\times10^{-5})$} \\ $0.885\times$}
     & {$7.77\times10^{-5}$       \\ {\scriptsize $(7.54-7.92\,\times10^{-5})$} \\ $0.834\times$} \\
    \bottomrule
  \end{tblr}%
\end{table}

\paragraph{Rank adaptivity}
We adapt the Nystr{\"o}mNGD rank parameter $\ell$ as the smallest index
satisfying $\hat{\lambda}_{\ell} / \mu < r$, following the recommendation $r =
  10$ from \cite{frangella_randomized_2023}. Decreasing $r$ increases the rank
$\ell$, improving the quality of the preconditioner at the cost of higher
computational expense. The results in \Cref{apx:tab:sensitivity_r} confirm this
trade-off across all five test problems, including the Kovasznay flow. Notably,
Nystr{\"o}mNGD is very robust to the choice of $r$: accuracy remains essentially
unchanged even at $r = 1000$, while the per-iteration cost decreases
significantly. The default value $r = 10$ achieves a good balance between
accuracy and per-iteration efficiency, although higher values also achieve good
accuracy at lower cost.

\begin{table}[htbp]
  \centering
  \caption{Sensitivity analysis of the parameter $r$ for Nystr{\"o}mNGD with
    Gaussian test matrices. The first line reports the median relative
    $\rmH^1(\Omega)$ error, the second line reports in smaller font and in
    parentheses the $25$th and $75$th percentiles in compact form as
    $(a-b\,\times 10^k)$, and the third line reports the wall-clock time
    relative to the default $r = 10$.}
  \label{apx:tab:sensitivity_r}
  \footnotesize
  \resizebox{\textwidth}{!}{%
  \begin{tblr}{
      colspec = {l|ccccc},
      row{even} = {gray!15},
      row{1} = {white},
    }
    \toprule
    $\boldsymbol{r}$
     & \textbf{Poisson 3D}
     & \textbf{Heat 3+1D}
     & \textbf{Kovasznay}
     & \textbf{Deep-Ritz Poisson}
     & \textbf{FEINNs 2D}         \\
    \midrule
    $\mathbf{0.1}$
     & {$1.68\times10^{-6}$       \\ {\scriptsize $(1.55-1.75\,\times10^{-6})$} \\ $1.00\times$}
     & {$2.03\times10^{-6}$       \\ {\scriptsize $(2.03-2.12\,\times10^{-6})$} \\ $1.00\times$}
     & {$6.95\times10^{-7}$       \\ {\scriptsize $(6.80-7.48\,\times10^{-7})$} \\ $1.08\times$}
     & {$1.64\times10^{-5}$       \\ {\scriptsize $(1.63-1.68\,\times10^{-5})$} \\ $1.14\times$}
     & {$6.61\times10^{-5}$       \\ {\scriptsize $(6.60-6.61\,\times10^{-5})$} \\ $1.13\times$} \\
    $\mathbf{1}$
     & {$1.60\times10^{-6}$       \\ {\scriptsize $(1.54-1.68\,\times10^{-6})$} \\ $1.00\times$}
     & {$2.22\times10^{-6}$       \\ {\scriptsize $(2.08-2.26\,\times10^{-6})$} \\ $1.00\times$}
     & {$8.57\times10^{-7}$       \\ {\scriptsize $(7.62-9.58\,\times10^{-7})$} \\ $1.08\times$}
     & {$1.65\times10^{-5}$       \\ {\scriptsize $(1.63-1.67\,\times10^{-5})$} \\ $1.13\times$}
     & {$6.61\times10^{-5}$       \\ {\scriptsize $(6.60-6.61\,\times10^{-5})$} \\ $1.06\times$} \\
    $\mathbf{10}$ \textbf{(def.)}
     & {$1.66\times10^{-6}$       \\ {\scriptsize $(1.47-1.89\,\times10^{-6})$} \\ $1.00\times$}
     & {$1.89\times10^{-6}$       \\ {\scriptsize $(1.79-1.93\,\times10^{-6})$} \\ $1.00\times$}
     & {$8.24\times10^{-7}$       \\ {\scriptsize $(6.74-13.8\,\times10^{-7})$} \\ $1.00\times$}
     & {$1.63\times10^{-5}$       \\ {\scriptsize $(1.62-1.65\,\times10^{-5})$} \\ $1.00\times$}
     & {$6.60\times10^{-5}$       \\ {\scriptsize $(6.60-6.61\,\times10^{-5})$} \\ $1.00\times$} \\
    $\mathbf{100}$
     & {$1.57\times10^{-6}$       \\ {\scriptsize $(1.48-1.62\,\times10^{-6})$} \\ $0.96\times$}
     & {$2.06\times10^{-6}$       \\ {\scriptsize $(1.75-2.14\,\times10^{-6})$} \\ $1.00\times$}
     & {$9.93\times10^{-7}$       \\ {\scriptsize $(8.04-9.96\,\times10^{-7})$} \\ $0.86\times$}
     & {$1.62\times10^{-5}$       \\ {\scriptsize $(1.61-1.62\,\times10^{-5})$} \\ $0.97\times$}
     & {$6.60\times10^{-5}$       \\ {\scriptsize $(6.60-6.61\,\times10^{-5})$} \\ $0.98\times$} \\
    $\mathbf{1000}$
     & {$1.59\times10^{-6}$       \\ {\scriptsize $(1.53-1.62\,\times10^{-6})$} \\ $0.85\times$}
     & {$2.66\times10^{-6}$       \\ {\scriptsize $(1.99-2.74\,\times10^{-6})$} \\ $0.92\times$}
     & {$7.48\times10^{-7}$       \\ {\scriptsize $(6.80-7.62\,\times10^{-7})$} \\ $0.66\times$}
     & {$1.62\times10^{-5}$       \\ {\scriptsize $(1.61-1.64\,\times10^{-5})$} \\ $0.94\times$}
     & {$6.62\times10^{-5}$       \\ {\scriptsize $(6.61-6.62\,\times10^{-5})$} \\ $0.90\times$} \\
    \bottomrule
  \end{tblr}}%
\end{table}

For RPCholNGD, rank adaptivity is instead controlled by an absolute error
tolerance $\varepsilon = r \cdot p \cdot \epsilon_{\mathrm{mach}}$ on the trace
of the residual matrix, with a default value of $r = 1$. RPCholesky terminates
early when the residual trace falls below $\varepsilon$, so larger $r$ yields a
lower rank and lower per-iteration cost, while smaller $r$ forces a higher rank
at greater expense. The results in \Cref{apx:tab:sensitivity_rpchol_r} show that
the method is robust for small $r$ (down to $r = 0.01$), with accuracy and
runtime essentially unchanged. For large $r$, accuracy degrades on the
variational problems, while the minor fluctuations on the strong-form problems are likely due to the stochasticity of the method.

\begin{table}[htbp]
  \centering
  \caption{Sensitivity analysis of the RPCholesky stopping tolerance scaling
    factor $r$ ($\varepsilon = r \cdot p \cdot \epsilon_{\mathrm{mach}}$) for
    RPCholNGD. The first line reports the median relative $\rmH^1(\Omega)$
    error, the second line reports in smaller font and in parentheses the
    $25$th and $75$th percentiles in compact form as $(a-b\,\times 10^k)$, and
    the third line reports the wall-clock time relative to the default $r = 1$.}
  \label{apx:tab:sensitivity_rpchol_r}
  \footnotesize
  \resizebox{\textwidth}{!}{%
  \begin{tblr}{
      colspec = {l|ccccc},
      row{even} = {gray!15},
      row{1} = {white},
    }
    \toprule
    $\boldsymbol{r}$
     & \textbf{Poisson 3D}
     & \textbf{Heat 3+1D}
     & \textbf{Kovasznay}
     & \textbf{Deep-Ritz Poisson}
     & \textbf{FEINNs 2D}         \\
    \midrule
    $\mathbf{0.01}$
     & {$1.83\times10^{-6}$       \\ {\scriptsize $(1.78-1.87\,\times10^{-6})$} \\ $1.00\times$}
     & {$2.11\times10^{-6}$       \\ {\scriptsize $(1.81-2.53\,\times10^{-6})$} \\ $1.00\times$}
     & {$9.79\times10^{-7}$       \\ {\scriptsize $(9.14-10.3\,\times10^{-7})$} \\ $1.00\times$}
     & {$1.64\times10^{-5}$       \\ {\scriptsize $(1.63-1.66\,\times10^{-5})$} \\ $1.14\times$}
     & {$6.60\times10^{-5}$       \\ {\scriptsize $(6.60-6.61\,\times10^{-5})$} \\ $1.09\times$} \\
    $\mathbf{0.1}$
     & {$1.88\times10^{-6}$       \\ {\scriptsize $(1.80-1.98\,\times10^{-6})$} \\ $1.00\times$}
     & {$2.17\times10^{-6}$       \\ {\scriptsize $(1.99-2.50\,\times10^{-6})$} \\ $1.00\times$}
     & {$9.79\times10^{-7}$       \\ {\scriptsize $(9.14-10.3\,\times10^{-7})$} \\ $1.00\times$}
     & {$1.63\times10^{-5}$       \\ {\scriptsize $(1.62-1.66\,\times10^{-5})$} \\ $1.10\times$}
     & {$6.60\times10^{-5}$       \\ {\scriptsize $(6.60-6.61\,\times10^{-5})$} \\ $1.05\times$} \\
    $\mathbf{1}$ \textbf{(def.)}
     & {$1.69\times10^{-6}$       \\ {\scriptsize $(1.58-1.82\,\times10^{-6})$} \\ $1.00\times$}
     & {$2.18\times10^{-6}$       \\ {\scriptsize $(1.94-2.42\,\times10^{-6})$} \\ $1.00\times$}
     & {$1.12\times10^{-6}$       \\ {\scriptsize $(9.09-13.3\,\times10^{-7})$} \\ $1.00\times$}
     & {$1.63\times10^{-5}$       \\ {\scriptsize $(1.62-1.64\,\times10^{-5})$} \\ $1.00\times$}
     & {$6.61\times10^{-5}$       \\ {\scriptsize $(6.60-6.61\,\times10^{-5})$} \\ $1.00\times$} \\
    $\mathbf{10}$
     & {$5.46\times10^{-6}$       \\ {\scriptsize $(3.43-34.1\,\times10^{-6})$} \\ $0.98\times$}
     & {$2.20\times10^{-6}$       \\ {\scriptsize $(2.01-2.32\,\times10^{-6})$} \\ $1.00\times$}
     & {$8.81\times10^{-7}$       \\ {\scriptsize $(7.98-9.07\,\times10^{-7})$} \\ $1.00\times$}
     & {$1.75\times10^{-5}$       \\ {\scriptsize $(1.70-58.8\,\times10^{-5})$} \\ $1.00\times$}
     & {$1.81\times10^{-3}$       \\ {\scriptsize $(0.938-6.51\,\times10^{-3})$} \\ $0.99\times$} \\
    $\mathbf{100}$
     & {$1.68\times10^{-6}$       \\ {\scriptsize $(1.63-1.71\,\times10^{-6})$} \\ $0.82\times$}
     & {$1.93\times10^{-6}$       \\ {\scriptsize $(1.82-2.09\,\times10^{-6})$} \\ $0.87\times$}
     & {$1.35\times10^{-6}$       \\ {\scriptsize $(1.34-1.50\,\times10^{-6})$} \\ $1.00\times$}
     & {$3.14\times10^{-1}$       \\ {\scriptsize $(1.67-8.24\,\times10^{-1})$} \\ $1.16\times$}
     & {$6.62\times10^{-5}$       \\ {\scriptsize $(6.62-6.64\,\times10^{-5})$} \\ $0.90\times$} \\
    \bottomrule
  \end{tblr}}%
\end{table}

\section{Choice of the metric}
\label{apx:sec:metric}
The link with operator preconditioning in \Cref{sec:operator_preconditioning_and_feinns} makes it evident that the choice of metric has a crucial influence on the convergence speed of NGD. In some cases, the problem's structure naturally suggests a metric, as in Variational Monte Carlo \cite{pfau_ab_2020,hermann_deep-neural-network_2020,muller_position_2024} or Sobolev gradient flows \cite{henning_sobolev_2020,chen_convergence_2024}. In general, the link between NGD and Riemannian gradient descent \cite{muller_achieving_2023} suggests selecting the metric $g$ to reduce the condition number of the Riemannian Hessian \cite{boumal_introduction_2023}. Indeed, it can be shown that NGD corresponds to Riemannian gradient descent \cite{boumal_introduction_2023} on the manifold $(\rmV, g)$, up to the $g$-orthogonal projection of the Riemannian gradient onto the model's generalized tangent space $\Tang_{\vtheta}\calN$, and a second-order correction; see \cite[Theorem 1]{muller_position_2024}. More precisely, the NGD update $\vtheta_{k+1} = \vtheta_{k} - \alpha_k \mG(\vtheta_k)^{\dagger}\nabla L(\vtheta_k)$ satisfies
\begin{equation*}
  u_{\vtheta_{k+1}} = u_{\vtheta_{k}} - \alpha_k \Pi_{\vtheta_k}^g\left(\grad[g] E(u_{\vtheta_k})\right) + \varepsilon_k,
\end{equation*}
where $\Pi_{\vtheta_k}^g$ denotes the $g$-orthogonal projection onto the generalized tangent space $\Tang_{\vtheta_k}\calN$, and $\varepsilon_k = \calO\big(\alpha_k^2\norm{\mG(\vtheta_k)^{\dagger}\nabla L(\vtheta_k)}^2\big)$. A principled strategy for selecting the metric $g$ is thus to adopt criteria from Riemannian optimization, choosing $g$ to reduce the condition number of the Riemannian Hessian \cite{boumal_introduction_2023,absil_optimization_2008}. Although convergence may differ from standard Riemannian gradient descent due to the projection onto $\Tang_{\vtheta}\calN$, whose effect depends on the geometry of $\calN$ and the parametrization, this strategy can be expected to remain effective if the ansatz $\calN$ is sufficiently expressive.

\subsection{Linear problems}
\label{apx_sec:linear_problems}
\begin{setting}
  \label{apx_setting:linear}
  Let $\rmV, \rmW$ be Hilbert spaces, and let $\calA:\rmV\to\dual{\rmW}$ be a linear isomorphism. Given $f\in\dual{\rmW}$, we aim at solving
  \begin{equation}
    \label{apx_eq:linear_problem}
    \calA u = f \in\dual{\rmW}.
  \end{equation}
\end{setting}
A natural approach is to consider the least-squares problem
\begin{equation}
  \label{apx_eq:linear_lossfun}
  \inf_{u\in\rmV} E(u) \defas \frac{1}{2} \norm{\calA u - f}_{\dual{\rmW}}^2.
\end{equation}

Choosing on $\rmV$ the (point-independent) metric
\begin{equation}
  \label{apx_eq:linear_metric}
  \bilinear[g]{v}{w} \defas \innerh[\dual{\rmW}]{\calA v}{\calA w},
\end{equation}
yields $\Hess[g] E = \mathrm{Id}$. Riemannian gradient descent on $(\rmV, g)$
thus converges in a single iteration, as the Riemannian gradient is given by
$\grad[g] E(u) = u - \calA^{-1}f$, i.e., computing it is equivalent to solving
\eqref{apx_eq:linear_problem}. NGD does not require applying $\calA^{-1}$:
computing the Gramian $\left[\mG(\vtheta)\right]_{ij} =
  \innerh[\dual{\rmW}]{\calA \partial_{\theta_i}u_{\vtheta}}{\calA
  \partial_{\theta_j}u_{\vtheta}}$ involves only applications of $\calA$.

An alternative is to choose the metric
\begin{equation}
  \label{apx_eq:linear_metric_dual}
  \bilinear[g]{v}{w} \defas \innerh[\rmV]{v}{w},
\end{equation}
which corresponds to the standard $\rmV$-inner product. In this case, the
Riemannian gradient is $\grad[g] E(u) = \calA^*(\calA u - f)$, and the
Riemannian Hessian is $\Hess[g] E = \calA^* \calA$. If an inf-sup condition with
constant $\beta > 0$ holds, i.e., $\inf_{v\in\rmV}\sup_{w\in\rmW}
  \frac{\dualpair[\rmW]{\calA v}{w}}{\norm{v}_\rmV \norm{w}_\rmW} \geq \beta$,
then the condition number of the Hessian is bounded by $\norm{\calA}^2 /
  \beta^2$.

\begin{remark}
  We formulate the linear problem from $\rmV$ to the \emph{dual} of $\rmW$ because one of our target applications involves variational formulations of PDEs, which naturally take the form \eqref{apx_eq:linear_problem}. For PINNs, which are based on strong formulations, $\rmW$ is typically a (product of) $\rmL^2$ spaces, so that $\dual{\rmW} \cong \rmW$ and the formulation recovers that of \cite{zeinhofer_unified_2024}, which does not employ the dual.
\end{remark}

\begin{remark}
  In Setting~\ref{apx_setting:linear}, it is typically straightforward to approximate numerically the inner product $\innerh[\rmW]{\cdot}{\cdot}$ and the duality pairing $\dualpair[\rmW]{\cdot}{\cdot}$, but not the dual inner product $\innerh[\dual{\rmW}]{\cdot}{\cdot}$. Let $\calR_\rmW: \rmW \to \dual{\rmW}$ denote the Riesz map, defined by
  \[
    \dualpair[\rmW]{\calR_\rmW v}{w} = \innerh[\rmW]{v}{w} \qquad \forall v,w \in \rmW.
  \]
  Then, for $\ell, h \in \dual{\rmW}$, we have
  \[
    \innerh[\dual{\rmW}]{\ell}{h} = \dualpair[\rmW]{\ell}{\calR_{\rmW}^{-1} h}.
  \]
  For example, when $\rmW = \rmH^1_0(\Omega)$, the dual space is $\dual{\rmW} = \rmH^{-1}(\Omega)$, and $\calR_{\rmW} = -\Delta$, the negative Laplace operator. Computing the dual inner product thus requires applying $(-\Delta)^{-1}$, i.e., solving a (discretized) Laplace problem. In contrast, for $\rmW = \rmL^2(\Omega)$, one has $\dual{\rmW} \cong \rmW$ and $\calR_{\rmW} = \mathrm{Id}$.
\end{remark}

\paragraph{Symmetric positive definite linear operators}
If $\rmV = \rmW$ and $\calA$ is symmetric positive definite, one can endow $\rmV$ with the $\calA$-induced inner product
\[
  \innerh[\calA]{v}{w} = \dualpair[\rmV]{\calA v}{w} \qquad \forall v,w \in \rmV.
\]
This corresponds to choosing $\calR_\rmV = \calA$, and the loss \eqref{apx_eq:linear_lossfun} becomes
\begin{equation}
  \label{apx_eq:linear_deepritz_lossfun}
  E(u) = \frac{1}{2}\dualpair[\rmV]{\calA u}{u} - \dualpair[\rmV]{f}{u} + \mathrm{const},
\end{equation}
where the constant is irrelevant for gradient-based optimization. This is precisely the variational form used in the Deep Ritz method \cite{e_deep_2017}. The corresponding natural metric is $g = \innerh[\calA]{\cdot}{\cdot}$.

\subsection{Nonlinear problems}
For general nonlinear problems, there is no canonical choice of metric as in the linear case. In some instances, the structure of the problem suggests multiple metrics; the choice of the metric impacts the convergence and the computational cost of the optimization schemes. Examples include Variational Monte Carlo (VMC) \cite{pfau_ab_2020,hermann_deep-neural-network_2020,muller_position_2024} or Sobolev gradient flows \cite{henning_sobolev_2020,chen_convergence_2024}. Below we outline some general strategies that can be used in nonlinear problems and are employed in our numerical experiments.

\paragraph{Newton's method}
\begin{setting}
  \label{apx_setting:nonlinear_newton}
  Let $(\rmV, a)$ be a Riemannian manifold and $E:\rmV\to\R$ a smooth function. We aim to solve $\inf_{u\in\rmV} E(u)$.
\end{setting}

If $\Hess[a] E(u)$ is symmetric positive definite for all $u\in\rmV$, a natural metric leading to Riemannian Newton's method \cite{boumal_introduction_2023} is
\begin{equation}
  \label{apx_eq:newton_metric}
  \bilinear[g_u]{v}{w} = \Hess[a] E(u)[v,w].
\end{equation}
In the case $(\rmV, a) = (\rmV, \innerh[\rmV]{\cdot}{\cdot})$ for a Hilbert space $\rmV$, this yields the Energy Natural Gradient Descent (ENGD) method \cite{muller_achieving_2023,muller_position_2024}. Note that to interpret ENGD as a Natural Gradient Descent, it is necessary that $\Hess[a] E(u)$ is SPD for all $u \in \rmV$, as otherwise \eqref{apx_eq:newton_metric} does not define a valid metric.

\paragraph{Nonlinear \emph{least-squares} problem: Gauss--Newton}
\begin{setting}
  \label{apx_setting:nonlinear_gauss}
  Let $\rmV, \rmW$ be Hilbert spaces, and let $\calR:\rmV\to\rmW$ be a (possibly nonlinear) residual. We aim at finding $u\in\rmV$ such that $\calR(u) = 0$.
\end{setting}

A natural approach is to minimize the nonlinear least-squares functional
\begin{equation}
  \label{apx_eq:nonlinear_lossfun}
  \inf_{u\in\rmV} E(u) \defas \inf_{u\in\rmV}\frac{1}{2} \norm{\calR(u)}_{\rmW}^2.
\end{equation}
The Hessian of $E$ reads
\begin{equation*}
  \rmD^2E(u)[v, w] = \innerh[\rmW]{\rmD\calR(u)[v]}{\rmD\calR(u)[w]} + \innerh[\rmW]{\calR(u)}{\rmD^2\calR(u)[v,w]},
\end{equation*}
and a symmetric positive semidefinite approximation is given by the Gauss--Newton method \cite{wright_numerical_2006}
\begin{equation}
  \label{apx_eq:gauss_metric}
  \bilinear[g_u]{v}{w} = \innerh[\rmW]{\rmD\calR(u)[v]}{\rmD\calR(u)[w]}.
\end{equation}
This choice of the metric leads to the Gauss--Newton NGD method
\cite{jnini_gauss-newton_2024,hao_gauss_2024}. Note that $g_u$ is, in general, only semidefinite. While this does not pose numerical issues when using the pseudoinverse or damping, the positive definiteness of $g_u$ is necessary to interpret it as a Riemannian metric. Whether this condition holds depends on the structure of the underlying nonlinear PDE.

\end{document}